\documentclass[notitlepage,leqno,10pt]{article}

\usepackage{graphicx}
\usepackage[utf8]{inputenc}
\usepackage[T1]{fontenc}
\usepackage{tgadventor}
\usepackage{amsthm, amssymb}
\usepackage{mathtools}
\usepackage{mathtools,xparse}
\usepackage{enumerate}
\DeclarePairedDelimiter{\abs}{\lvert}{\rvert}
\DeclarePairedDelimiter{\norm}{\lVert}{\rVert}
\mathtoolsset{centercolon}
\usepackage[english]{babel}
\usepackage{csquotes}
\usepackage{faktor}
\usepackage{xcolor}
\usepackage{hyperref}

\usepackage{latexsym}
\usepackage{amsmath}
\usepackage{amsthm}
\usepackage{amsfonts}
\usepackage{amssymb}
\usepackage{graphicx}
\renewcommand{\a }{\alpha }

\newcommand{\rh }{\rho }

\newcommand{\intbar}{\mathop{\int\makebox(-13.5,0){\rule[4pt]{.7em}{0.3pt}}%
\kern-6pt}\nolimits}

\newcommand{\be}{\begin{equation}}
\newcommand{\ee}{\end{equation}}

\newcommand{\R}{\mathbb{R}}

\newcommand{\Rm}{\mathrm{Rm}}
\newcommand{\Ric}{\mathrm{Ric}}
\newcommand{\Sc}{\mathrm{Sc}}
\newcommand{\Exp}{\mathrm{Exp}}
\newcommand{\Tr}{\mathrm{Tr}}
\newcommand{\Ker}{\mathrm{Ker}}
\newcommand{\Hess}{\mathrm{Hess}}
\newcommand{\cQ}{\mathcal{Q}}
\newcommand{\cL}{\mathcal{L}}
\newcommand{\cO}{\mathcal{O}}

\newcommand{\N}{\mathbb{N}}
\newcommand{\rP}{\mathrm{P}}
\newcommand{\rA}{\mathrm{A}}
\newcommand{\sfd}{\mathsf{d}}
\newcommand{\ADM}{{\rm ADM}}

\newtheorem{theorem}{Theorem}[section]

\newtheorem{lemma}[theorem]{Lemma}
\newtheorem{prop}[theorem]{Proposition}
\newtheorem{remark}[theorem]{Remark}
\newtheorem{definition}[theorem]{Definition}

\newcommand{\KN}{\mathbin{\bigcirc\mspace{-15mu}\wedge\mspace{3mu}}}
\begin{document}

\author{ Andrea Mondino\thanks{University of Oxford. Email address: Andrea.Mondino@maths.ox.ac.uk} \;  and\; Aidan Templeton-Browne\thanks{University of Warwick. Email address: A.Browne.2@warwick.ac.uk}}

\date{}

\title{Some rigidity results for the Hawking mass and a lower bound for the Bartnik capacity}

\maketitle

\

\

\begin{center}
\noindent {\sc abstract}. 
We prove rigidity results involving the Hawking mass for surfaces immersed in a $3$-dimensional, complete Riemannian manifold $(M,g)$ with non-negative scalar curvature (resp. with scalar curvature bounded below by $-6$). Roughly, the main result states that if an open subset $\Omega\subset M$ satisfies that every point  has a neighbourhood $U\subset \Omega$ such that the supremum of the Hawking mass of surfaces contained in $U$ is non-positive, then $\Omega$ is locally isometric to Euclidean $\R^3$ (resp. locally isometric to the Hyperbolic 3-space ${\mathbb H}^3$). Under mild asymptotic conditions on the manifold $(M,g)$ (which encompass as special cases the standard ``asymptotically flat'' or, respectively,  ``asymptotically hyperbolic'' assumptions) the previous quasi-local rigidity statement implies a \emph{global rigidity}: if every point  in $M$ has a neighbourhood $U$ such that the supremum of the Hawking mass of surfaces contained in $U$ is non-positive, then $(M,g)$ is globally isometric to Euclidean $\R^3$ (resp. globally isometric to the Hyperbolic 3-space ${\mathbb H}^3$). Also, if the space is not flat (resp. not of constant sectional curvature $-1$), the methods give a small yet explicit and strictly positive lower bound on the  Hawking mass of suitable spherical surfaces. We infer a small yet explicit and strictly positive lower bound on the Bartnik mass of open sets (non-locally isometric to Euclidean $\R^{3}$) in terms of curvature tensors. Inspired by these results, in the appendix we propose a notion of ``sup-Hawking mass'' which satisfies some natural properties of a  quasi-local mass.
\bigskip\bigskip


\bigskip

\centerline{\bf AMS subject classification: }
53C20, 53C21, 53C42, 83C99.
\end{center}

\section{Introduction}\label{s:in}
The goal of this paper is to prove some rigidity results  involving the Hawking mass for surfaces immersed in an asymptotically flat (resp. asymptotically hyperbolic; and some suitable milder asymptotic assumptions), $3$-dimensional, complete Riemannian manifold $(M,g)$ with non-negative scalar curvature (resp. with scalar curvature bounded below by $-6$). Let us start by recalling some motivations for studying such a geometric setting, coming from Mathematical General Relativity.  

\subsection{Some background and motivation}

In the framework of Mathematical General Relativity, a Riemannian manifold $(M^{3}, g)$ with non-negative scalar curvature $\Sc_{g}\geq 0$ represents a ``time-symmetric space-like slice'' of a 4-dimensional space-time satisfying the so-called \emph{Dominant Energy Condition} (DEC for short).  The condition $\Sc_{g}\geq 0$ is indeed a consequence of the DEC, coupled with the property of being a ``time-symmetric space-like slice'' (i.e. $M$ has vanishing second fundamental form as a hypersurface in the ambient 4-dimensional space-time) satisfying the  Einstein Constraint Equations.

The relationship between the notion of mass in General Relativity and the geometry of a manifold has been studied extensively in recent times. Let us recall some basic notation and fundamental results.
First of all, let us mention that, unless otherwise specified, a manifold is allowed to have non-empty boundary.

\begin{definition}[Asymptotically flat Riemannian manifold with horizon boundary]\label{def:AFRiemMan}
 A complete Riemannian manifold $(M^{3},g)$ is said to be \emph{asymptotically flat} (AF for short) if  there exists a compact subset $\mathcal K\Subset M$ such that $M^{3}\setminus \mathcal K$ is diffeomorphic to $\R^{3}\setminus\overline {B_{1}(0)}$ via a map $\Psi$ which induces a system of coordinates at infinity, and in such coordinates the metric is written as 
\begin{equation}
g_{ij}=\delta_{ij}+\sigma_{ij} \quad \text{where} \quad |x|^{|\alpha|} |(\partial^{\alpha} \sigma_{ij}) (x)|=\cO(|x|^{-\tau}) \quad \text{as }|x|\to \infty,
\end{equation}
for some $\tau>1/2$ and all multi-indices $\alpha$, with $|\alpha|=0,1,2,3$.  We also require that the scalar curvature $\Sc$ of $(M,g)$ is integrable. Moreover, if the boundary of $M$ is non-empty, we assume that it is minimal and that there are no closed minimal surfaces intersecting the interior of $M$;  in this case we say that the manifold is \emph{asymptotically flat with horizon boundary}. 
\end{definition}

Notice that  in Definition \ref{def:AFRiemMan} and throughout, we only consider one-ended AF manifolds with horizon boundary. For an AF  Riemannian manifold $(M,g)$ with horizon boundary as above there is natural notion of  ``total mass'' of $(M,g)$,  well known under the name of  $\ADM$-mass (after Arnowitt-Deser-Misner \cite{ADM}) and defined by
\begin{equation}\label{def:ADMmass}
m_{\ADM}(M,g)=\lim_{\rho\to \infty} \frac{1}{16\pi \rho} \int_{|x|=\rho} \sum_{i,j=1}^{3} (\partial_{i} g_{ij}- \partial_{j} g_{ii}) x^{j}.
 \end{equation}
Bartnik \cite{Bartnik}  showed that $m_{\ADM}(M,g)$ as defined in \eqref{def:ADMmass} is finite and independent of the chart at infinity.

 The fundamental \emph{Positive Mass Theorem}, proven first by Schoen-Yau \cite{SchoenYau1} using minimal surfaces methods and then by Witten \cite{Witten} via spinorial techniques, states that if $(M^{3},g)$ is an AF Riemannian manifold with non-negative scalar curvature, then $m_{\ADM}(M,g)\geq 0$ and equality holds if and only if $(M^{3},g)$ is isometric to the Euclidean $\R^{3}$.
 \\Another landmark example of the interplay between mass and geometry is the \emph{Riemannian Penrose Inequality},  proved independently by Huisken-Ilmanen  \cite{HuiskenIlmanen} and Bray \cite{Bray}. The Riemannian Penrose Inequality is a strengthening of the lower bound on the ADM mass in case of horizon boundary of an AF manifold with non-negative scalar curvature; namely, $m_{\ADM}(M,g)\geq  \sqrt {\tfrac{|\partial M|}{16 \pi}}$.
\\

In this paper, we further investigate the relationship between geometry and mass. More precisely between the  geometric conditions of non-negative scalar curvature and asymptotic flatness  (and some generalisations of it) on the one hand, and  two notions of \emph{quasi-local mass}:  those of Hawking \cite{Hawking} and Bartnik \cite{Bartnik2}.  Let us recall that, while in Newtonian gravity it is possible to define the mass contained in a region $\Omega$  simply by integration of a ``mass density function'' over $\Omega$,  defining a corresponding concept in the setting of General Relativity is much more subtle. Indeed, due to the Equivalence Principle, there is no  pointwise notion of ``gravitational energy density'' (see for instance \cite[Section 20.4]{Thorne} or \cite{Penrose}).
\\Nevertheless there are several proposals for a notion of ``quasi-local mass'' contained in a region $\Omega$ (see for instance \cite[Chap. 6]{LeeBook} or \cite{Szaba}); we will focus on the  Hawking \cite{Hawking} and Bartnik \cite{Bartnik2} quasi-local masses.

\subsubsection*{The Hawking mass}

\begin{definition}[Hawking mass]\label{def:5}
Let $(M^3,g)$ be a Riemannian manifold with non-negative scalar curvature and let $\Sigma \hookrightarrow M$ be an immersed sphere. The Hawking mass of $\Sigma$ is defined as:
\begin{equation}\label{eq:defHaw}
    m_{H}(\Sigma) := \sqrt{\frac{|\Sigma |}{(16\pi)^3}} \;\Big(16\pi - W(\Sigma) \Big)
\end{equation}
where $|\Sigma|$ is the area,  $W(\Sigma) := \int_{\Sigma}H^2 dV_{\Sigma}$ is the Willmore functional and where we adopt the convention that the mean curvature $H$ is the \emph{sum} of the principal curvatures.
\end{definition}

Evidently, if $\Sigma$ is a minimal surface, then its Hawking mass is positive. It is also readily checked that the Hawking mass of a  round sphere in $\R^3$ is equal to zero, as its Willmore functional equals $16\pi$. A classical inequality due to Willmore \cite{Willmore} asserts that  the round sphere is the unique minimiser of the Willmore functional (up to scaling). Thus the Hawking mass of any surface in $\R^{3}$ is less or equal to zero, with equality if and only if the surface is a round sphere. This fact is already suggesting that it could be appropriate to consider the supremum of the Hawking mass, over a suitable class of surfaces (see the Appendix \ref{SS:supHawk} for an implementation of this idea for a notion of quasi-local mass). 
\\Nevertheless, Christodoulou and Yau \cite{Christodoulou} showed that the Hawking mass is non-negative for stable mean curvature spheres in $3$-manifolds with non-negative scalar curvature (see also the more recent work  \cite{MiaoWangXie} by Miao-Wang-Xie).
 The popularity of the Hawking mass is arguably due to the very powerful monotone property (Eardley, Geroch,
Jang-Wald \cite{JangWald}) along the Inverse Mean Curvature Flow, which was key  in Huisken-Ilmanen’s \cite{HuiskenIlmanen} proof of the Riemannian Penrose inequality for a single black hole (see Bray \cite{Bray} for the multiple black holes case and for a different proof). 
  
 \subsubsection*{The Bartnik mass} 
 First of all let us mention that, after Bartnik \cite{Bartnik2}  introduced the quasi-local mass named after him, several variants appeared in the literature. For convenience, here we adopt a version proposed in  \cite{Bray, HuiskenIlmanen}.  We refer to the recent \cite{McCormick, Jauregui}  for reconciliation of some of the various versions and for a discussion of the topic.

Let $(M^{3},g)$ be an AF Riemannian manifold (possibly with horizon boundary) with non-negative scalar curvature and let $\Omega\subset M$ be a bounded open set  with smooth topological boundary $\partial \Omega$. Recall that $\partial \Omega$ is said to be \emph{outer-minimising} if $ \rP(\Omega)  \leq \rP(\Omega')$ for any set $\Omega'\subset M$ of finite perimeter (denoted with $\rP(\Omega')$) and finite volume  such that $\Omega\subset \Omega'$. Define the \emph{Bartnik mass} $m_B(\Omega)$ of $\Omega$ (also known as Bartnik capacity of $\Omega$) as
\begin{equation}\label{eq:defmB}
    m_B(\Omega) := \inf \{m_{\ADM}(\Tilde{M}, \Tilde{g}) :  (\Tilde{M}, \Tilde{g})\in {\mathcal{A}} \}
\end{equation}
where ${\mathcal A}$ is the set of AF manifolds  (possibly with horizon boundary)  with non-negative scalar curvature  into which $\Omega$ isometrically embeds such that $\partial\Omega\subset \Tilde{M}$ is outer-minimising.
\\Notice that the Positive Mass Theorem (or the Riemannian Penrose inequality, in case all elements in  ${\mathcal{A}}$ have non-empty horizon boundary) immediately yields the non-negativity of the Bartnik mass. Another advantage of the Bartnik mass is that it satisfies a natural monotonicity property under inclusion %(see for instance \cite[Theorem 4.4]{BrayChrusciel}):
\begin{equation}\label{eq:monmB}
    m_B(\Omega_1) \leq m_B(\Omega_2) \quad \forall \text{$\Omega_1 \Subset \Omega_2\subset M$, with $\partial\Omega_1$ outer-minimising in $\Omega_2$ and $\partial \Omega_{2}$ outer minimising in $M$.}
\end{equation}
Moreover, as a consequence of the proof of the Riemannian Penrose inequality via Inverse Mean Curvature Flow by Huisken-Ilmanen \cite{HuiskenIlmanen}, it holds that
\begin{equation}\label{eq:mADmgeqmH}
    m_{\ADM}(M,g) \geq m_H(\partial \Omega)
\end{equation}
for every $\partial \Omega\subset M$ outer-minimising in the AF manifold $(M,g)$ with non-negative scalar curvature. 
\\Since every smooth extension $(\tilde{M}, \tilde{g}) \in {\mathcal A}$ induces the same mean curvature on $\partial\Omega\subset \tilde{M}$, the inequality \eqref{eq:mADmgeqmH} combined with the very definition \eqref{eq:defmB}   implies 
\begin{equation}\label{mBgeqmH}
    m_B(\Omega) \geq m_H(\partial\Omega).
\end{equation}
Computing the Bartnik mass of a subset $\Omega$ is in general a non-trivial task. For some recent works in this direction, see for instance  Mantoulidis-Schoen \cite{SchoenMant}, CabreroPacheco-Cederbaum-McCormick-Miao \cite{PachecoCederbaumMcCormickMiao} and Miao-Xie \cite{MiaoXie}: while it is clear from the definition that it is conceivable to expect upper bounds by direct comparison with somewhat explicit competitors, the issue of finding explicit lower bounds is more subtle. The latter is one of the goals of the present paper.

%FINO QUI, forse togliere la sezione sulla discussione informale
%
%
%\subsubsection*{Informal discussion of the results}
%
%Our main rigidity result involving the Hawking mass  (Theorem \ref{thm:1}) can be informally stated as follows: if  the Hawking quasi-local mass is non-positive locally anywhere in a given 3-dimensional Riemannian manifold with non-negative scalar curvature, then the only possible geometries for $(M,g)$ are flat, i.e. $(M,g)$ has constant sectional curvature zero.  Adding the  assumption of asymptotic flatness (or the weaker assumption of ``asymptotically locally simply connectedness'' see ) reduces the options to just one: Euclidean  $\R^3$.
% \\We will then use the Hawking mass rigidity to prove two results involving the Bartnik mass  (Theorems \ref{thm:7} and \ref{thm:8}). The first one is also of rigidity-type and is in fact already contained in \cite{HuiskenIlmanen}; we provide a different proof, yet still using some of the techniques therein. In the second we give a lower bound  for the Bartnik mass in terms of the Hawking mass of suitable spherical surfaces of small area, which in turn can be expressed as a kind of Taylor expansion involving curvature tensors. 
% \\  Inspired by these results, in the appendix we propose a notion of ``sup-Hawking mass'' which satisfies very natural properties as quasi-local mass.
%
%

\subsection{Statement of the main results}

The first main result can be informally stated as follows: if  the Hawking quasi-local mass is non-positive locally everywhere in an open set $\Omega$ having non-negative scalar curvature, then $\Omega$ is locally isometric to Euclidean $\R^3$. Below is the precise statement:

\begin{theorem}[Quasi-local rigidity Theorem \ref{prop:Riem0}]\label{thm:HawkMassRigOmega}
Let $(M^3,g)$ be a three-dimensional Riemannian manifold  and let $\Omega\subset M$ be an open subset with non-negative scalar curvature.
Assume that  every $p\in \Omega\setminus \partial M$ admits a neighbourhood $U\subset M\setminus \partial M$ such that
\begin{equation}\label{eq:mHUleq0Intro}
\sup\{m_{H}(\Sigma) \,:\, \Sigma \subset U \emph{ is an immersed 2-dimensional surface}\} \leq 0.
\end{equation}
Then $\Omega\setminus \partial M$ is locally isometric to Euclidean $\R^3$.
\end{theorem}

In order to obtain a \emph{global} rigidity result out of the \emph{quasi-local} rigidity Theorem \ref{thm:HawkMassRigOmega}, it is useful to add a suitable asymptotic condition at infinity. A fairly mild asymptotic assumption is the next one (that to the best of our knowledge has not appeared in the literature before). In order to state it, recall that  a sequence $\{p_n\}\subset M$ is said to be \emph{diverging} if, for some (and thus for any) fixed $\Bar{p} \in M$, we have ${\mathsf d}(\Bar{p}, p_n) \rightarrow \infty \  as \  n \rightarrow \infty$, where ${\mathsf d}$ is the Riemannian distance function induced by $g$.

\begin{definition}\label{def:2}
Let $(M,g)$ be a complete Riemannian manifold. We say $(M,g)$  is \emph{asymptotically locally simply connected} (ALSC for short) if it is non-compact and for every  $R > 0$, and any diverging sequence $\{p_n\}\subset M$, there exists $N=N(R)\geq 1$ such that  the ball $B^g_R(p_n)$ is simply connected, for every  $n>N(R)$.    
\end{definition}

\begin{remark}
Note that the ALSC condition is satisfied by AF manifolds since, in this case, the balls $B^g_R(p_n)$ are eventually diffeomorphic to  Euclidean balls, which of course are simply connected. However, ALSC is a much milder condition than AF, as it merely requires a \emph{local} (mild) topological control and  no assumption on the metric tensor and on the \emph{global} topology of $M$ (outside of a compact set). 
Non AF examples of ALSC manifolds include for instance asymptotically conical manifolds, the Bryant soliton \cite{Bryant} and, more generally, $C^{0}$-asymptotically locally Euclidean manifolds \cite{MondNard}.
\end{remark}

\begin{theorem}[Global rigidity Theorem \ref{thm:1}]\label{thm:1Intro}
Let $(M^3,g)$ be a connected, complete Riemannian manifold without boundary and with non-negative scalar curvature. If every $p\in M$ admits a neighbourhood $U$ satisfying the local non-positivity condition \eqref{eq:mHUleq0Intro} for the Hawking mass,  then 
$(M^3,g)$ is isometric to a flat space form. Furthermore, if $(M^3,g)$ is ALSC then it is globally isometric to Euclidean $\R^3$.
\end{theorem}

We mention again that, inspired by the above rigidity results, in the Appendix \ref{SS:supHawk} we propose a notion of  ``sup-Hawking mass'' which satisfies some natural properties of a quasi-local mass.
\smallskip

For the sake of the introduction, we confined the presentation to manifolds with non-negative scalar curvature. The reader interested in the extensions to the case of scalar curvature bounded below by a negative (or strictly positive constant) is referred to Section \ref{Sec:OtherRigRes}. Let us recall that such ambifent spaces are particularly relevant when the cosmological constant in Einstein's Equations is non-zero.
\smallskip

In order to prove the above rigidity results, we will compute accurate Taylor expansions of the Hawking mass on suitable spherical surfaces of small area (see Section \ref{SS:someIdeas} for some idea of the methods, and Proposition \ref{prop:expHawk} for the precise statement). In  Theorem \ref{Prop:OuterMin}, we will show that such spherical surfaces are outer-minimising, and thus provide a lower bound on the Bartnik mass (thanks to the monotonicity property \eqref{eq:monmB} and the bound \eqref{mBgeqmH}). As a consequence, we obtain the next lower bound on the Bartnik mass in terms of curvature tensors.

\begin{theorem}[Lower Bound on the Bartnik mass]\label{thm:7-8}
Let $(M^{3},g)$ be an AF Riemannian manifold with non-negative scalar curvature and with (possibly empty) horizon boundary $\partial M$.  Let $\Omega\subset M$ be a bounded open set  with smooth boundary $\partial \Omega$.
Let $p\in \Omega \setminus \partial M$ and let $\bar{\rho}=\bar{\rho}(p):=\inf_{q\in \partial M \cup \partial \Omega} \sfd(p,q)$. Then, for all $\rho\in (0, \bar{\rho}/2)$ the following lower bound holds: 
\begin{equation}\label{eq:LBBart}
    m_B(\Omega) \geq \frac{1}{12}\Sc_p\rho^{3} + \left(\frac{1}{120}\Delta \Sc(p) + \frac{1}{90}\norm{S_{p}}^2 - \frac{1}{144}\Sc_p^2\right)\rho^5 + \cO_{p}(\rho^6),
\end{equation}
where $S:= \Ric-\frac{1}{3} \Sc\, g$ denotes the trace-free Ricci tensor and $\limsup_{\rho\to 0^{+}} \rho^{-6} \, |\cO_{p}(\rho^6)|<\infty$.
\\ In particular,   if $m_B(\Omega) = 0$ then $\Omega \setminus \partial M $ is locally isometric to Euclidean $\R^3$.
\end{theorem}

\begin{remark}
Notice that \eqref{eq:LBBart}  gives a strictly positive (yet small) lower bound on the Bartnik mass of $\Omega$, provided that 
\begin{itemize}
\item either: $\Omega$ contains a point $p$ with $\Sc_p>0$ 
\item or: $\Sc \equiv 0$ on $\Omega$ and  there exists $p\in \Omega$ with  $\|S_p\|\neq 0$.
\end{itemize}
As observed in the proof of Theorem  \ref{prop:Riem0},  if $\Omega$ has non-negative scalar curvature and it does not have constant sectional curvature, then one of the two conditions above must be satisfied.
\end{remark}

\begin{remark}
Note that  the lower bound \eqref{eq:LBBart} is interesting only when $\Omega$ does not contain some connected component of $\partial M$. Indeed if $\Omega$ contains some connected component $\Sigma$ of   $\partial M$, then  the Riemannian Penrose Inequality  \cite{Bray, HuiskenIlmanen} yields that  $m_B(\Omega)\geq \sqrt{ \frac{ |\Sigma| }{16 \pi}}$. This would give a \emph{definite} lower bound on  $m_B(\Omega)$ in contrast with the \emph{infinitesimal} lower bound \eqref{eq:LBBart}.  However, if $\partial M \cap \Omega= \emptyset$,  the lower bound \eqref{eq:LBBart} seems to be new and interesting.   

Let us also mention the recent work  \cite{Wiygul} by Wiygul, where the first order Taylor expansion of the Bartnik mass is computed for closed geodesic balls of small radius $\rho>0$ and center $p\in M$, giving  $\frac{1}{12}\Sc_p\rho^{3}$.  Under the additional condition that the Riemann curvature tensor vanishes at $p$,  the first order Taylor expansion of the Bartnik mass for such geodesic balls is given by  $\frac{1}{120}\Delta \Sc(p) \rho^5$. 
\\Note that these results are in accordance with the lower bound given in \eqref{eq:LBBart}, which holds without the assumption that the Riemann curvature tensor vanishes at $p$.
\end{remark}

%\begin{remark}
%A careful reader may wonder why we assume $\partial M\cap \Omega=\emptyset$ in Theorem \ref{thm:7-8}. The reason is the following: if $\Omega$ contains some connected component $\Sigma$ of   $\partial M$ then  the Riemannian Penrose Inequality  \cite{Bray, HuiskenIlmanen} yields that  $m_B(\Omega)\geq \sqrt{ \frac{ |\Sigma| }{16 \pi}}$. This would give a \emph{definite} lower bound on  $m_B(\Omega)$ in contrast with the \emph{infinitesimal} lower bound \eqref{eq:LBBart}. Thus  \eqref{eq:LBBart} seems to give some new and interesting information only in the case  $\partial M\cap \Omega=\emptyset$, i.e. under the assumptions of Theorem \ref{thm:7-8}.
%\end{remark}

%\begin{remark}
%Note that  the lower bound \eqref{eq:LBBart} is interesting only when is $\partial \Omega$ is \emph{not minimal}. Indeed,  in  case  $\partial \Omega$ is minimal, the Riemannian-Penrose inequality would give a better lower bound: as we will prove in Theorem  \ref{Prop:OuterMin}, for small $\rho>0$, the geodesic sphere $S_{p,\rho}$ is outer-minimising and thus $|\partial \Omega| \geq |S_{p,\rho}| = \cO(\rho^2)$; hence, the Riemannian-Penrose inequality gives $m_{\ADM}(M,g) = m_{\ADM}(M\setminus {\rm int}(\Omega),g) \geq |\partial \Omega|$ for every extension $M$. This would give the better lower bound $m_B(\Omega) \geq |\partial \Omega| \geq \cO(\rho^2)$. However, in case $\partial \Omega$ is \emph{not minimal}, the lower bound \eqref{eq:LBBart} seems to be new and interesting.   
%\end{remark}

\begin{remark}
The fact that  $m_B(\Omega) = 0$ forces $\Omega  \setminus \partial M$ to be locally isometric to Euclidean $\R^3$ was obtained with different methods by Huisken-Ilmanen (see \cite[Positivity Property 9.2]{HuiskenIlmanen}).
In the same paper, Huisken-Ilmanen write \emph{``careful examination of the proof (of the positivity property 9.2) should give a small, but
explicit lower bound for  $m_B(\Omega)$  in terms of the scalar curvature in a
small region''}.  Our methods permit to implement this: indeed \eqref{eq:LBBart} gives a small, but quantitative lower bound for $m_B(\Omega)$ in terms of the scalar curvature (at first order, and in terms of other curvature tensors for higher order) in a small region (in our arguments, such a region will be a perturbed geodesic sphere of small radius).
\end{remark}

\subsection{Some ideas of the proofs}\label{SS:someIdeas}

In order to prove the quasi-rigidity Theorem \ref{thm:HawkMassRigOmega}, it is of course key to identify suitable competitors in order to test the condition \eqref{eq:mHUleq0Intro}.
A first attempt would be to use geodesic spheres $S_{p,\rho}$ contained in the open set $U$. However such surfaces are not ``optimal enough'' for the Hawking mass. Indeed, using the expansions of the paper (which in turn build on top of \cite{Pacard, Mondino2, Mondino}), one can check that
\begin{align*}
	m_H(S_{p,\rho}) &= \sqrt{\frac{|S_{p,\rho}|_{\mathring{g}}}{(16\pi)^3}}\left( \frac{8\pi}{3}\Sc_p\rho^2 - \left[ \frac{4\pi}{27}\Sc_p^2 - \frac{4\pi}{15}\Delta \Sc(p) \right]\rho^4 + \cO(\rho^5) \right) \\&= \frac{1}{12}\Sc_p\rho^{3} - \left( \frac{1}{144}\Sc_p^2 - \frac{1}{120}\Delta \Sc(p)\right)\rho^5 + \cO(\rho^6).
\end{align*}
Note that such expansions give no interesting information in case $\Sc \equiv 0$. The idea is thus to ``optimally perturb'' the geodesic spheres by suitable normal graphs, motivated by the fact (not strictly necessary for the arguments, but useful as a motivation) that the maximisers of the Hawking mass under small area constraint are indeed perturbed geodesic spheres. This principle has been already observed by Lamm-Metzger \cite{LammMetzger2013}, who proved $W^{2,2}$-closeness to a geodesic sphere under a small energy assumption, and by Laurain-Mondino \cite{MondinoLaurain}, who proved smooth convergence to  a geodesic sphere under a milder energy assumption. For the reader's convenience, we will give a self-contained proof of this fact in the exact framework of the present paper in Proposition \ref{prop:SigmaSprhow} in the appendix.
\smallskip

In order to compute the Hawking mass for such optimal competitors, in Lemma \ref{lem:estw} we prove that the  graph function $w_{p,\rho}$ for an optimally perturbed geodesic sphere must satisfy a precise expansion in terms of $\rho$ and of curvature tensors at $p$; namely it has the form
\begin{equation}\label{eq:estwIntro}
    w_{p,\rho}(\Theta) = \Big(-\frac{1}{6}\Ric(\Theta,\Theta) + \frac{1}{18}\Sc(p)\Big)\rho^2 + \cO_{p}(\rho^3),
\end{equation}
where $\Theta\in {\mathbb S}^{2}$ is the parametrising coordinate, and  $\limsup_{\rho\to 0}  \rho^{-3} \, \left \| \cO_{p}(\rho^{3}) \right \|_{C^{4,\alpha}({\mathbb S}^{2})}<\infty$.
\\Plugging such an expansion of the optimal normal graph  $w_{p,\rho}$ into the definition of the Hawking mass gives, after some computational efforts, the expansion of  $m_{H}$ on such optimal competitors:
\begin{equation}\label{eq:mHSprhowIntro}
    m_H(S_{p,\rho}(w_{p,\rho})) =\frac{1}{12} \Sc_p\rho^{3} +  \left(\frac{1}{120} \Delta \Sc(p) + \frac{1}{90}\norm{S_p}^2 - \frac{1}{144}\Sc_p^2  \right)  \rho^5 + \cO_{p}(\rho^6). 
\end{equation}
 This is proved in Proposition \ref{prop:expHawk}, which represents the key technical result of the paper. Indeed, once \eqref{eq:mHSprhowIntro} is proved, it is not hard to obtain the quasi-local rigidity Theorem \ref{thm:HawkMassRigOmega} (see Theorem \ref{prop:Riem0} in the body of the paper): basically the assumption  $\Sc\geq 0$ on $\Omega$ coupled with the condition that $m_{H}(S_{p,\rho}(w_{p,\rho}))\leq 0$ for every $p\in \Omega$ forces $\Omega$ to be Ricci flat and then flat,  thanks to the expansion \eqref{eq:mHSprhowIntro}. 
 \smallskip
 
 The global rigidity Theorem  \ref{thm:1Intro} (Theorem  \ref{thm:1} in the body of the paper) follows from the  quasi-local rigidity Theorem \ref{thm:HawkMassRigOmega} and the classification of flat manifolds (see for instance \cite{Wolf}), plus a  case by case analysis when applying the ALSC condition.
  \smallskip
 
 In order to obtain the lower bound on the Bartnik mass (Theorem \ref{thm:7-8}) out of the expansion \eqref{eq:mHSprhowIntro}, we prove that the optimally perturbed geodesics spheres $S_{p,\rho}(w_{p,\rho})$ are outer-minimising. This is achieved in Theorem \ref{Prop:OuterMin}  via a blow-up argument. In order to gain strong enough compactness in such an argument, we employ the regularity theory for quasi-minimisers of the perimeter by Tamanini \cite{Tamanini}, refining previous celebrated results by De Giorgi \cite{DeGiorgi}.
 \\Once it is established that $S_{p,\rho}(w_{p,\rho})$ are outer-minimising, Theorem \ref{thm:7-8} follows from the expansion \eqref{eq:mHSprhowIntro} combined with the monotonicity property \eqref{eq:monmB} and the bound \eqref{mBgeqmH}.

\bigskip

\noindent {\bf Acknowledgments}. 
A.M. is supported by the European Research Council (ERC), under the European's Union Horizon 2020 research and innovation programme, via the ERC Starting Grant ``CURVATURE'', grant agreement No. 802689.
\\A.T. is supported by the Engineering and Physical Sciences Research Council (EPSRC), via the grant EP/N509796/1.
\\The authors are grateful to the reviewers whose careful reading and comments improved the exposition of the manuscript.

\section{Preliminaries and notation}

%
%
%A key step in the proof of our main results will be establishing refined expansions of the Hawking mass associated to a \textit{perturbed geodesic sphere}, i.e. a normal graphs over a geodesic sphere:
%
%\begin{definition}[Perturbed geodesic sphere]\label{def:3}
%The perturbed geodesic sphere centered at $p\in M$ of radius $\rho > 0$, written $S_{p,\rho}(w)$, is given as the following image:
%
%\begin{equation*}
%    T_pM \supset \Theta \rightarrow \exp_p(\rho(1 - w(\Theta))\Theta)
%\end{equation*}
%
%where $\Theta$ is the usual parameterisation of the unit sphere given by:
%
%\begin{equation*}
%    \Theta = (x,y,z) = (\sin{\theta^1}\cos{\theta^2}, \sin{\theta^1}\sin{\theta^2}, \cos{\theta^1})\quad 0<\theta^1<\pi ,\  0<\theta^2<2\pi
%\end{equation*}
%
%where $\rho$ is small enough to ensure $\exp_p$ is a diffeomorphism and:
%
%\begin{equation*}
%    w \in C^{4,\alpha}({\mathbb S}^2) \cap Ker[\Delta_{{\mathbb S}^2}(\Delta_{{\mathbb S}^2} + 2)]^{\perp} \subset L^2({\mathbb S}^2)
%\end{equation*}
%\end{definition}
%
%\begin{remark}\label{Rem:3}
%The function space chosen for the perturbation function $w$ will allow us to simplify some calculations later because it can be shown that $w = O(\rho^2)$ as $\rho \rightarrow 0$ (\cite{Mondino2}, Lemma 4.4).
%\end{remark}

\qquad 1) We will use the convention that  greek index letters (e.g. $\mu,\nu,\eta,$ etc.) varies from $1$ to $3$ while latin index letters (e.g. $i,j,k,l$ etc.) vary from $1$ to $2$.  We will adopt the Einstein convention for summation over repeated indices. 
\\

2) For a 3-dimensional Riemannian manifold $(M^3,g)$ with Levi-Civita connection $\nabla$,  the Riemann curvature endomorphism is given by:
\begin{equation*}
    {\mathcal R}(X,Y)Z = \nabla_X\nabla_YZ - \nabla_Y\nabla_XZ - \nabla_{[X,Y]}Z
\end{equation*}
for vector fields $X,Y,Z$ on $M$. The associated Riemann curvature $(0,4)$-tensor is:
\begin{equation*}
    \Rm(X,Y,Z,W) = g({\mathcal R}(Z,W)Y,X).
\end{equation*}
The Ricci curvature tensor is the trace over the first and third indices of $\Rm$: i.e. if $\{E_{\mu}\}_{\mu=1,2,3}\subset T_pM$ is an orthonormal basis, we have
\begin{equation*}
    \Ric(X,Y) = \sum_{\mu=1}^{3} \Rm(E_\mu,X,E_\mu,Y), \quad \forall X,Y\in T_{p}M. %\\&= g({\mathcal R}(E_1,Y)X,E_1) + g({\mathcal R}(E_2,Y)X,E_2) + g({\mathcal R}(E_3,Y)X,E_3) \\&= - g({\mathcal R}(Y,E_1)X,E_1) - g({\mathcal R}(Y,E_2)X,E_2) - g({\mathcal R}(Y,E_3)X,E_3)
\end{equation*}
 The scalar curvature,  denoted by $\Sc$, is the trace of the Ricci tensor:
\begin{equation*}
   \Sc(p) =\sum_{\mu=1}^{3} \Ric(E_{\mu},E_{\mu}).
 \end{equation*}
 We adopt the standard index notation $R_{\mu\nu}:=\Ric(E_{\mu},E_{\mu})$. A key tensor for this paper is the traceless Ricci tensor
 \begin{equation}\label{eq:defTrFreeRic}
 S:= \Ric-\frac{1}{3} \Sc\; g.
 \end{equation}
 It is easily seen that $\|S\|^{2}= \|\Ric\|^{2}-\frac{1}{3} \Sc^{2}$.
 
3) Let $(\mathring{M},\mathring{g})\hookrightarrow (M,g)$ be an isometrically immersed, closed, 2-sided,  2-dimensional surface with inward pointing normal unit vector $N$.  The (scalar) second fundamental form $h$ is defined by
\begin{equation}\label{eq:13}
    h(X,Y) = g(\nabla_XY, N) = -g(\nabla_X N,Y) ,
\end{equation}
for $X,Y\subset T\Sigma$ vector fields tangent to $\Sigma$. The two eigenvalues $k_1$ and $k_2$ of $h$ at  $p\in \Sigma$ are called \emph{the principal curvatures}. We set 
$$
H := k_1 + k_2 \quad \text{and} \quad D := k_1k_2,
$$ 
where $H$ is called \emph{the mean curvature}. Let $\mathring{g}$ denote the restriction of $g$ to $T\Sigma$ (with matrix $\mathring{g}_{ij}$ with respect to a fixed set of coordinates) and let $\mathring{g}^{ij}$ denote the matrix of $\mathring{g}^{-1}$. It holds
\begin{equation*}
    H = \Tr_{\mathring{g}}(h) = \mathring{g}^{ij}h_{ij} \qquad D = \det(\mathring{g}^{ik}h_{kj}) = \frac{\det h_{ij}}{\det\mathring{g}_{kl}} .
\end{equation*}

4) Large positive constants are always denoted by $C$. The value of
$C$ is allowed to vary from formula to formula and also within the
same line. When we want to stress the dependence of the constants on
some parameter (or parameters), we either add subscripts to $C$, e.g. $C_{\delta}$, or we mention the dependence with parenthesis, e.g. $C(\delta)$. Also constants with subscripts (or with parenthesis) are allowed to vary.
%The area of $\Sigma$ is:
%
%\begin{equation*}
%   {\rm Area}_{g}(\Sigma) := \int_{\Sigma} dV_{\mathring{g}} = \int_{\phi(\Sigma)} (\phi^{-1})^*(dV_{\mathring{g}}) = \int_{\phi(\Sigma)} \sqrt{\det \mathring{g}_{ij}}dx^1dx^2
%\end{equation*}
%
%for some local coordinates $x^i$ and chart $\phi$. We will also consider measurable subsets of $(M,g)$ and so we recall the $k$-dimensional Hausdorff measure:
%
%\begin{equation*}
%    \mathcal{H}_g^k := \lim_{\epsilon \rightarrow 0}\inf_{\mathcal{F}}\sum_{F\in\mathcal{F}}(diam(F))^k
%\end{equation*}
%
%where $\mathcal{F}$ is a countable covering of $E$ by sets $F\subset M$ such that $diam(F) < \epsilon$, and the diameter is with respect to the Riemannian distance induced by the metric $g$. The relative perimeter of a measurable set $F$ inside an open set $V\subset M$ is defined as:
%
%\begin{equation}\label{eq:27} 
%    P_g(F,V) := \sup_f\{\int_{F\cap V} div_g(f)dV_g \ |\ f\in C^1_c(V,TM), ||f||_{\infty,g}\leq 1\}
%\end{equation}
%
%where $||f||_{\infty,g}$ is worked out in the normal coordinates induced by $g$. When $V=M$ we get the full perimeter $P_g(F)$.

\subsection{Perturbed geodesic spheres}

\subsubsection{Notation about perturbed geodesic spheres}
Denote with ${\mathbb S}^2$ the standard unit sphere in the Euclidean $3$-dimensional space $\R^3$, $\Theta \in {\mathbb S}^2$  the radial unit vector with components $\Theta^\mu$  parametrised by the polar coordinates $0<\theta^1<\pi$ and  $0<\theta^2< 2\pi$ chosen in order to satisfy
\begin{displaymath} 
\left\{ \begin{array}{ll}
\Theta^1= \sin \theta ^1 \cos \theta ^2 \\
\Theta^2= \sin \theta ^1 \sin \theta ^2 \\
\Theta^3= \cos \theta ^1. \\
\end{array} \right.
\end{displaymath}
Call with  $\Theta_i$ the coordinate vector fields on ${\mathbb S}^2$: 
\begin{equation}\label{eq:defTheta}
\Theta_1:= \frac{\partial \Theta}{\partial \theta^1}, \quad \Theta_2:= \frac{\partial \Theta}{\partial \theta^2},
\end{equation}
and $\bar{\theta}_i$ or $\bar{\Theta}_i$ the corresponding normalised vectors: 
$$\bar{\theta}_1=\bar{\Theta}_1:=\frac{\Theta_1}{\|\Theta_1\|}, \quad \bar{\theta_2}=\bar{\Theta}_2:=\frac{\Theta_2}{\|\Theta_2\|}.$$
 We next define the perturbed geodesic spheres $S_{p,\rh}(w)$ in the 3-dimensional Riemannian manifold $(M,g)$.
\\Fix a point $p\in M\setminus \partial M$ and consider the exponential map 
$\Exp_{p}$ centred at $p$. For $\rh>0$ small enough, the sphere $\rh {\mathbb S}^2\subset T_pM$ is contained in the injectivity radius of the exponential map. We  call $S_{p,\rh}$ the \emph{geodesic sphere} of center $p$ and radius $\rh$, parametrised by
$$\Theta \in {\mathbb S}^2 \subset T_pM \mapsto \Exp_{p}[\rh \Theta].$$
The perturbed geodesic spheres are normal graphs on geodesic spheres, by a function $w$ belonging to the following suitable function spaces (chosen for technical reasons in order to apply Schauder estimates in Lemma \ref{lem:estw}). 
\\Denote with $C^{4,\a}({\mathbb S}^2)$ (or simply $C^{4,\a}$)  the set of the $C^4$ functions on ${\mathbb S}^2$ whose fourth derivatives, with respect to the tangent vector fields, are $\a$-H\"older continuous ($0<\a<1$). The Laplace-Beltrami operator on ${\mathbb S}^2$ is denoted by $\Delta_{{\mathbb S}^2}$ or, if there is no ambiguity, by $\Delta$.
The fourth order elliptic operator $\Delta (\Delta +2)$ induces the following orthogonal splitting of $L^2({\mathbb S}^2)$:
$$L^2({\mathbb S}^2)= \Ker[\Delta (\Delta+2)]\oplus \Ker[\Delta (\Delta+2)]^\perp; $$
note that the splitting makes sense since the kernel is finite (four) dimensional, thus a closed subspace.
\\If we consider $C^{4,\a}({\mathbb S}^2)$ as a subspace of $L^2({\mathbb S}^2)$, we can define 
$$C^{4,\alpha}({\mathbb S}^2)^\perp :=C^{4,\alpha}({\mathbb S}^2)\cap \Ker[\Delta (\Delta+2)]^\perp.$$
Of course ${C^{4,\alpha}({\mathbb S}^2)}^\perp$ is a Banach space with respect to the $C^{4,\a}$ norm; it is the space from which we will draw the perturbation $w$. If there is no confusion $C^{4,\alpha}({\mathbb S}^2)^\perp$ will be called simply ${C^{4,\a}}^\perp$.

We can now define the \emph{perturbed geodesic spheres} that we will use as ``test'' surfaces for the Hawking mass.
Fix $p\in M$, $\rh>0$ and a small $C^{4,\a}({\mathbb S}^2)$ function $w$; the perturbed geodesic sphere $S_{p,\rh}(w)$ is the surface parametrised by
$$\Theta \in {\mathbb S}^2 \mapsto \Exp_{p}[\rho\big(1-w(\Theta)\big) \Theta].$$
The tangent vector fields on $S_{p,\rh}(w)$ induced by the canonical polar coordinates on ${\mathbb S}^2$ are denoted by $Z_i$.
\\

Following the notation of \cite{Pacard}, given $a \in \N$, we denote with  $\cL_p^{(a)}(w)$ an arbitrary linear combination of the function $w$ together 
with its derivatives with respect to the tangent vector fields $\Theta_i$ up to order $a$. The coefficients of $\cL_p^{(a)}$ 
may depend on $\rh$ and $p$ but, for all $k \in \N$, there exists a constant $C=C_{p}>0$ independent on $\rh \in (0,1)$  such that 
$$\|\cL_p^{(a)}(w)\|_{C^{k,\a}({\mathbb S}^2)}\leq C \|w \|_{C^{k+a,\a}({\mathbb S}^2)}.$$  

 Similarly,  given $a,b\in \N$, we denote with $\cQ_p^{(b)(a)}(w)$ an arbitrary nonlinear combination, of order at least $b$, of the function $w$ together with its derivatives with respect to the tangent vector fields $\Theta_i$ up to order $a$ such that  $\cQ_p^{(b)(a)}(0)=0$, for every $p \in M$. The coefficients of the Taylor expansion of $\cQ_p^{(b)(a)}(w)$ in powers of $w$ and its partial derivatives may depend on $\rh$ and $p$ but, for all $k \in \N$, there exists a constant $C=C_{p}>0$ independent on $\rh \in (0,1)$  such that 
\begin{equation}\label{eq:estQ}
\|\cQ_p^{(b)(a)}(w_2)-\cQ_p^{(b)(a)}(w_1) \|_{C^{k,\a}({\mathbb S}^2)}\leq C \big(\|w_2\|_{C^{k+a,\a}({\mathbb S}^2)}+\|w_1\|_{C^{k+a,\a}({\mathbb S}^2)}\big)^{b-1}
\times \|w_2-w_1\|_{C^{k+a,\a}({\mathbb S}^2)},
\end{equation}
provided $\|w_l\|_{C^{a}({\mathbb S}^2)}\leq 1$, $l=1,2$. 
We also agree that $\cO_p(\rho ^d)$ denotes an arbitrary smooth function on ${\mathbb S}^2$ that might depend on $p$ but which is bounded by a constant (possibly dependent on $p$) times $\rho ^d$ in $C^k(B_{1}(p))$ topology, for all $k \in \N$.

\subsubsection{Expansions of geometric quantities}\label{SS:GeomExp}
In this subsection we recall the Taylor expansion of the geometric quantities associated to the a perturbed geodesic sphere $S_{p,\rho}(w)$, appearing in the Willmore functional and its first derivative (for the proofs see \cite{Pacard} and \cite[Section 3.1]{Mondino}). These will be used in later sections and hold for any Riemannian 3-manifold.
\\For the following expansions we will fix the (polar) coordinate vector fields $\Theta_{i}$ on  ${\mathbb S}^{2}$ defined in \eqref{eq:defTheta}, and  the corresponding coordinate vector fields $Z_{i}$ on $S_{p,\rho}(w)$; i.e. we will use the notation $ \mathring{g}_{ij}:= \mathring{g}(Z_{i}, Z_{j})$ (resp. $(g_{{\mathbb S}^2})_{ij}:=g(\Theta_{i}, \Theta_{j})$) and analogously $h_{ij}:=h(Z_{i}, Z_{j})$.   The derivatives of the function $w:{\mathbb S}^{2}\to \R$ with respect to $\Theta_{i}$ are denoted by $w_{i}$. All the curvature terms, all the covariant derivatives and all the scalar products are meant to be evaluated at $p$ (since we fixed normal coordinates centred at $p$, at $p$ the metric is Euclidean).

The following expansions hold:

{\small

\begin{align*}
    g_{\mu\nu} &= \delta_{\mu\nu} + \frac{1}{3}g({\mathcal R}(\Theta,E_{\mu})\Theta,E_{\nu})(1 - w)^2\rho^2 + \frac{1}{6}g(\nabla_{\Theta}{\mathcal R}(\Theta,E_{\mu})\Theta,E_{\nu})(1 - w)^3\rho^3 \\ &\quad+ \frac{1}{20}g(\nabla_{\Theta \Theta}^2{\mathcal R}(\Theta,E_{\mu})\Theta,E_{\nu})(1 - w)^4\rho^4 \\ &\quad+ \frac{2}{45}g({\mathcal R}(\Theta,E_{\mu})\Theta,E_{\tau})g({\mathcal R}(\Theta,E_{\nu})\Theta,E_{\tau})(1 - w)^4\rho^4 \\
    & \quad + \cO_{p}(\rho^5) + \rho^5 \cL_p^{(0)}(w) +  \rho^5 \cQ_{p}^{(2)(0)}(w)
\end{align*}

\begin{align}\label{eq:Expgij}
   \mathring{g}_{ij} &= g_{ij}^{{\mathbb S}^2}(1 - w)^2\rho^2 + w_iw_j\rho^2 + \frac{1}{3}g({\mathcal R}(\Theta,\Theta_i)\Theta,\Theta_j)(1 - w)^4\rho^4 \nonumber\\
    &\quad+ \frac{1}{6}g(\nabla_{\Theta}{\mathcal R}(\Theta,\Theta_i)\Theta,\Theta_j)(1 - w)^5\rho^5 + \frac{1}{20}g(\nabla_{\Theta \Theta}^2{\mathcal R}(\Theta,\Theta_i)\Theta,\Theta_j)(1 - w)^6\rho^6 \\ &\quad+ \frac{2}{45} \delta^{\mu\nu}g({\mathcal R}(\Theta,\Theta_i)\Theta,E_{\mu})g({\mathcal R}(\Theta,\Theta_j)\Theta,E_{\nu})(1 - w)^6\rho^6 
     \nonumber\\ &\quad + \cO(\rho^7) + \rho^7 \cL_p^{(0)}(w)+  \rho^7 \cQ_{p}^{(2)(0)}(w)\nonumber
\end{align}

%
%\begin{align*}
%(1-w)^{-2} \rho^{-2} \mathring{g}_{ij}=&\, g^{{\mathbb S}^{2}}_{ij} + (1-w)^{-2}w_i w_j +\frac {1}{3} g({\mathcal R}(\Theta, \Theta_{i}) \Theta, \Theta_{j} )\rho ^2 (1-w)^2  +\frac {1}{6}  g(\nabla_\Theta {\mathcal R}(\Theta, \Theta_{i}) \Theta, \Theta_{j} ) \rho ^3 (1-w)^3 \nonumber\\
%&+\Big[ \frac {1}{20}  g(\nabla^{2} _{\Theta\Theta}{\mathcal R}(\Theta, \Theta_{i}) \Theta, \Theta_{j} )+\frac {2}{45}\delta^{\mu\nu} g({\mathcal R}(\Theta, \Theta_{i}) \Theta, E_{\mu} ) g({\mathcal R}(\Theta, \Theta_{j}) \Theta, E_{\nu} ) \Big]\rho ^4 (1-w)^4 \nonumber\\
%&+ \cO_p(\rho^5)+\rho^5 \cL_p^{(0)}(w)+\rho^5 \cQ_p^{(2)(0)}(w), 
%\end{align*}

\begin{align}\label{eq:ExpgijInv}
\mathring{g}^{ij} = &\, g_{{\mathbb S}^2}^{ij}(1 - w)^{-2}\rho^{-2} - g_{{\mathbb S}^2}^{il}g_{{\mathbb S}^2}^{kj}w_lw_k(1 - w)^{-4}\rho^{-2} - \frac{1}{3}g_{{\mathbb S}^2}^{il}g({\mathcal R}(\Theta,\Theta_l)\Theta,\Theta_k)g_{{\mathbb S}^2}^{kj} \nonumber \\
&- \frac{1}{6}g_{{\mathbb S}^2}^{il}g(\nabla_{\Theta}{\mathcal R}(\Theta,\Theta_l)\Theta,\Theta_k)g_{{\mathbb S}^2}^{kj}(1 - w)\rho - \frac{1}{20}g_{{\mathbb S}^2}^{il}g(\nabla^{2}_{\Theta\Theta} {\mathcal R}(\Theta,\Theta_l)\Theta,\Theta_k)g_{{\mathbb S}^2}^{kj}(1 - w)^2\rho^2 \nonumber \\
 &- \frac{2}{45}\delta^{\mu\nu} g_{{\mathbb S}^2}^{il}g({\mathcal R}(\Theta,\Theta_l)\Theta,E_{\mu})g({\mathcal R}(\Theta,\Theta_k)\Theta,E_{\nu})g_{{\mathbb S}^2}^{kj}(1 - w)^2\rho^2  \\
  &+ \frac{1}{9}g_{{\mathbb S}^2}^{il}g({\mathcal R}(\Theta,\Theta_l)\Theta,\Theta_k)g_{{\mathbb S}^2}^{kn}g({\mathcal R}(\Theta,\Theta_n)\Theta,\Theta_m)g_{{\mathbb S}^2}^{mj}(1 - w)^2\rho^2  \nonumber\\
  &+ \cO_{p}(\rho^3) +   \rho^3 \cL_{p}^{(0)}(w) +  \rho^2 \cQ_{p}^{(2)(0)}(w) + \rho^{-2}  \cQ_{p}^{(4)(1)}(w). \nonumber
\end{align}

For the details of the derivation of the next expansion, the interested reader can see the Appendix \ref{SS:AppSecFF}.
\begin{align} \label{eq:Exphij}
h_{ij} =&\, g_{ij}^{{\mathbb S}^{2}}(1 - w)\rho + (\Hess_{{\mathbb S}^2}(w))_{ij}\rho    \nonumber\\
&+\frac{1}{2} g_{ij}^{{\mathbb S}^{2}}  g^{kl}_{{\mathbb S}^{2}} w_{k} w_{l} \rho + w_{k} g_{{\mathbb S}^2}^{kl}\Big(g^{{\mathbb S}^{2}}_{jl} w_{i} +g^{{\mathbb S}^{2}}_{il} w_{j}-g^{{\mathbb S}^{2}}_{ij} w_{l} \Big)\rho  \nonumber\\
& + \frac{2}{3}g({\mathcal R}(\Theta,\Theta_i)\Theta,\Theta_j)(1 - w)^3\rho^3  \nonumber\\
&+ \frac{1}{6}w_kg_{{\mathbb S}^2}^{kn}g_{{\mathbb S}^2}^{ml}g({\mathcal R}(\Theta,\Theta_n)\Theta,\Theta_m)\Big(\partial_i g^{{\mathbb S}^{2}}_{jl } + \partial_j g^{{\mathbb S}^{2}}_{il}  - \partial_lg^{{\mathbb S}^{2}}_{ij} \Big)\rho^3   \\ 
&- \frac{1}{6}w_kg_{{\mathbb S}^2}^{kl}\Big(\partial_ig({\mathcal R}(\Theta,\Theta_j)\Theta,\Theta_l) + \partial_jg({\mathcal R}(\Theta,\Theta_i)\Theta,\Theta_l) - \partial_lg({\mathcal R}(\Theta,\Theta_i)\Theta,\Theta_j)\Big)\rho^3 \nonumber\\ 
    & + \frac{5}{12}g(\nabla_{\Theta}{\mathcal R}(\Theta,\Theta_i)\Theta,\Theta_j) \rho^4   + \frac{3}{20}g(\nabla^{2}_{\Theta\Theta}{\mathcal R}(\Theta,\Theta_i)\Theta,\Theta_j) \rho^5  \nonumber\\ 
    & + \frac{2}{15}\delta^{\mu\nu}g({\mathcal R}(\Theta,\Theta_i)\Theta,E_{\mu})g({\mathcal R}(\Theta,\Theta_j)\Theta,E_{\nu}) \rho^5  \nonumber\\ 
   &    + \cO_{p}(\rho^6) +  \rho^4 \cL^{(1)}_{p}(w) +  \rho \cQ^{(3)(2)}_{p}(w)+  \rho^3 \cQ_{p}^{(2)(1)}(w). \nonumber
\end{align}

Recalling that $H= \mathring{g}^{ij} h_{ij}$, the combination of \eqref{eq:ExpgijInv} and \eqref{eq:Exphij} gives:
\begin{align}\label{eq:ExpH}
H =&\,  2\rho^{-1} + (2 + \Delta_{{\mathbb S}^2})w\rho^{-1} + 2w(w+\Delta_{{\mathbb S}^2}w)\rho^{-1}  \nonumber\\ 
        &+ \frac{1}{6}w_kg_{{\mathbb S}^2}^{ij}g_{{\mathbb S}^2}^{kn}g_{{\mathbb S}^2}^{ml}g({\mathcal R}(\Theta,\Theta_n)\Theta,\Theta_m) \Big(\partial_i g^{{\mathbb S}^{2}}_{jl } + \partial_j g^{{\mathbb S}^{2}}_{il}  - \partial_lg^{{\mathbb S}^{2}}_{ij} \Big) \rho \nonumber\\
      &- \frac{1}{6}w_kg_{{\mathbb S}^2}^{ij}g_{{\mathbb S}^2}^{kl}\Big(\partial_ig({\mathcal R}(\Theta,\Theta_j)\Theta,\Theta_l) + \partial_jg({\mathcal R}(\Theta,\Theta_i)\Theta,\Theta_l) - \partial_lg({\mathcal R}(\Theta,\Theta_i)\Theta,\Theta_j)\Big)\rho \nonumber\\
      & -\frac{1}{3}g_{{\mathbb S}^2}^{il}g_{{\mathbb S}^2}^{kj}g({\mathcal R}(\Theta,\Theta_l)\Theta,\Theta_k)(\Hess_{{\mathbb S}^2}(w))_{ij}\rho - \frac{1}{3}\Ric(\Theta,\Theta)(1 - w)\rho \nonumber\\
       &+ \frac{1}{4}g_{{\mathbb S}^2}^{ij}g(\nabla_{\Theta}{\mathcal R}(\Theta,\Theta_i)\Theta,\Theta_j)(1 - w)^2\rho^2  \\
       &+ \Bigg[\frac{1}{10}g_{{\mathbb S}^2}^{ij}g(\nabla_{\Theta \Theta}^2{\mathcal R}(\Theta,\Theta_i)\Theta,\Theta_j) + \frac{4}{45}g_{{\mathbb S}^2}^{ij} \delta^{\mu\nu}g({\mathcal R}(\Theta,\Theta_i)\Theta,E_{\mu})g({\mathcal R}(\Theta,\Theta_j)\Theta,E_{\nu}) \nonumber\\
        &\qquad - \frac{1}{9}g_{{\mathbb S}^2}^{il}g_{{\mathbb S}^2}^{kj}g({\mathcal R}(\Theta,\Theta_i)\Theta,\Theta_j)g({\mathcal R}(\Theta,\Theta_l)\Theta,\Theta_k)\Bigg](1 - w)^3\rho^3 \nonumber\\
        & + \cO_{p}(\rho^4)+  \rho^2 \cL_{p}^{(1)}(w)  +  \rho^{-1} \cQ_{p}^{(3)(2)}(w)+   \rho \cQ_{p}^{(2)(1)}(w)\nonumber ,
\end{align}
and thus
\begin{align}\label{eq:ExpH2}
    H^2 =&\,  \Bigg[4 + 4(2 + \Delta_{{\mathbb S}^2})w + 8w(w+\Delta_{{\mathbb S}^2}w) + ((2 + \Delta_{{\mathbb S}^2})w)^2 \Bigg]\rho^{-2}\nonumber\\ 
       & + \Bigg[\frac{2}{3}w_kg_{{\mathbb S}^2}^{ij}g_{{\mathbb S}^2}^{kn}g_{{\mathbb S}^2}^{ml}g({\mathcal R}(\Theta,\Theta_n)\Theta,\Theta_m) ) \Big(\partial_i g^{{\mathbb S}^{2}}_{jl } + \partial_j g^{{\mathbb S}^{2}}_{il}  - \partial_lg^{{\mathbb S}^{2}}_{ij} \Big) \nonumber\\ 
    &\quad -\frac{2}{3}\Ric(\Theta,\Theta)(2 + \Delta_{{\mathbb S}^2}w) -\frac{4}{3}g_{{\mathbb S}^2}^{il}g_{{\mathbb S}^2}^{kj}g({\mathcal R}(\Theta,\Theta_l)\Theta,\Theta_k)(\Hess_{{\mathbb S}^2}(w))_{ij} \nonumber\\ 
    & \quad - \frac{2}{3}w_kg_{{\mathbb S}^2}^{ij}g_{{\mathbb S}^2}^{kl}\Big(\partial_ig({\mathcal R}(\Theta,\Theta_j)\Theta,\Theta_l) + \partial_jg({\mathcal R}(\Theta,\Theta_i)\Theta,\Theta_l) - \partial_lg({\mathcal R}(\Theta,\Theta_i)\Theta,\Theta_j)\Big)\Bigg] \nonumber\\ 
    &+\Bigg[g_{{\mathbb S}^2}^{ij}g(\nabla_{\Theta}{\mathcal R}(\Theta,\Theta_i)\Theta,\Theta_j)\Bigg]\rho \\ 
    &+\Bigg[\frac{2}{5}g_{{\mathbb S}^2}^{ij}g(\nabla^{2}_{\Theta \Theta} {\mathcal R}(\Theta,\Theta_i)\Theta,\Theta_j) + \frac{16}{45}\delta^{\mu\nu}g_{{\mathbb S}^2}^{ij}g({\mathcal R}(\Theta,\Theta_i)\Theta,E_{\mu})g({\mathcal R}(\Theta,\Theta_j)\Theta,E_{\nu}) \nonumber\\  
    &\quad - \frac{4}{9}g_{{\mathbb S}^2}^{il}g_{{\mathbb S}^2}^{kj}g({\mathcal R}(\Theta,\Theta_i)\Theta,\Theta_j)g({\mathcal R}(\Theta,\Theta_l)\Theta,\Theta_k) + \frac{1}{9}\left(\Ric(\Theta,\Theta)\right)^{2}\Bigg]\rho^2 \nonumber\\ 
      & + \cO_{p}(\rho^3)+  \rho \cL_{p}^{(1)}(w) +  \cQ_{p}^{(2)(1)}(w) +  \rho^{-2} \cQ_{p}^{(3)(2)}(w). \nonumber
\end{align}

The determinant of the first fundamental form will be useful to compute integrals on $S_{p,\rho}(w)$ and can be expanded as:

\begin{align}\label{eq:Expdetg}
 \det\mathring{g} =& \,  \sin^2{\theta^1}\rho^4\Bigg[(1 - w)^4 + g_{{\mathbb S}^2}^{ij}w_iw_j - \frac{1}{3}\Ric(\Theta,\Theta)(1 - w)^6\rho^2 \nonumber\\
  &+ \frac{1}{6}g_{{\mathbb S}^2}^{ij}g(\nabla_{\Theta}{\mathcal R}(\Theta,\Theta_i)\Theta,\Theta_j)(1 - w)^7\rho^3 + \frac{1}{20}g_{{\mathbb S}^2}^{ij}g(\nabla^{2}_{\Theta \Theta}{\mathcal R}(\Theta,\Theta_i)\Theta,\Theta_j)(1 - w)^8\rho^4 \nonumber\\
  &+ \frac{2}{45} \delta^{\mu\nu} g_{{\mathbb S}^2}^{ij} g({\mathcal R}(\Theta,\Theta_i)\Theta,E_{\mu})g({\mathcal R}(\Theta,\Theta_j)\Theta,E_{\nu})(1 - w)^8\rho^4 \\
   & + \frac{1}{9}g({\mathcal R}(\Theta,\Theta_1)\Theta,\Theta_1)g({\mathcal R}(\Theta,\Bar{\Theta}_2)\Theta,\Bar{\Theta}_2)(1 - w)^8\rho^4 - \frac{1}{9}g({\mathcal R}(\Theta,\Theta_1)\Theta,\Bar{\Theta}_2)^2(1 - w)^8\rho^4\Bigg]  \nonumber  \\ 
   &+ \cO_{p}(\rho^9) + \rho^9 \cL_{p}^{(0)}(w) + \rho^6 \cQ_{p}^{(2)(1)}(w) + \rho^4 \cQ_{p}^{(4)(1)}(w). \nonumber
\end{align}

\section{Hawking mass of an optimally perturbed geodesic sphere}

Motivated by Proposition \ref{prop:SigmaSprhow} in the appendix, it is natural to choose the class of perturbed geodesic spheres $S_{p,\rho}(w)$ as surfaces to ``test'' the positivity of the Hawking mass. 
More precisely, by using the Euler-Lagrange equation of the Willmore functional (under area constraint), we will find an expansion of the perturbation $w$ which is necessary for the perturbed geodesic sphere $S_{p,\rho}(w)$ to be a critical point (under area constraint). Let us stress that in Lemma \ref{lem:estw} we do not claim to construct area-constrained Willmore surfaces centred at any point (that would be false, as a necessary condition would be that such a point is a critical point for the scalar curvature; see  \cite{LammMetzger2013, MondinoLaurain}).  The goal of Lemma \ref{lem:estw} is thus merely to suggest an \emph{ansatz} for an  ``optimal'' perturbation (so the reader could take for granted that a convenient choice of perturbation $w$ is given by the expansion \eqref{eq:estw}).  Such ``optimally'' perturbed geodesic spheres will be the key geometric objects to prove our main theorems. Throughout the section, $(M,g)$ will be an arbitrary Riemannian 3-manifold.

\subsection{An optimal pertubation}
In this subsection we compute the expansion (as $\rho\to 0$)  that a perturbation $w$ has to satisfy if  the perturbed geodesic sphere $S_{p,\rho}(w)$ is a critical point (under area constraint) of the Willmore functional or, equivalently, of the Hawking mass.

\begin{lemma}\label{lem:estw}
For a fixed a compact subset ${\mathcal K}\Subset M \setminus \partial M$, there exists $\rho_{0}>0, \,r>0$ and a map $w_{(\cdot, \cdot)}: {\mathcal K}\times (0,\rho_{0}]\to C^{4,\alpha}({\mathbb S}^{2})^{\perp}$, $(p,\rho)\mapsto w_{p,\rho}$ such that if $S_{p,\rho}(w)$ is a critical point of the Willmore functional under area constraint (or equivalently, a critical point of the Hawking mass under area constraint) with $(p,\rho,w)\in {\mathcal K}\times (0,\rho_{0}] \times B_{C^{4,\alpha}({\mathbb S}^{2})^{\perp}}(0,r)$ then $w=w_{p,\rho}$. Moreover, for every $(p,\rho)\in {\mathcal K}\times (0,\rho_{0}]$ the following expansion holds:
\begin{equation}\label{eq:estw}
    w_{p,\rho}(\Theta) = \Big(-\frac{1}{6}\Ric(\Theta,\Theta) + \frac{1}{18}\Sc(p)\Big)\rho^2 + \cO_{p}(\rho^3),
\end{equation}
where $\Theta\in {\mathbb S}^{2}$ is the parametrising coordinate, and $\limsup_{\rho\to 0}  \rho^{-3} \, \left \| \cO_{p}(\rho^{3}) \right \|_{C^{4,\alpha}({\mathbb S}^{2})}<\infty$.

\end{lemma}

\begin{proof}

\textbf{Step 1}: We show that, for every  fixed  compact subset ${\mathcal K}\Subset M$, there exists $\rho_{0}>0, \,r>0$ and a map $w_{(\cdot, \cdot)}: {\mathcal K}\times (0,\rho_{0}]\to C^{4,\alpha}({\mathbb S}^{2})^{\perp}$, $(p,\rho)\mapsto w_{p,\rho}$ such that if $S_{p,\rho}(w)$ is a critical point of the Willmore functional under area constraint (or equivalently, a critical point of the Hawking mass under area constraint) with $(p,\rho,w)\in {\mathcal K}\times (0,\rho_{0}] \times B(0,r)$ then $w=w_{p,\rho}$ and  $\lim_{\rho\to 0} \|w_{p,\rho}\|_{C^{4,\alpha}({\mathbb S}^{2})}=0$.

If $S_{p,\rho}(w)$ is a critical point of the Willmore functional under area constraint then it satisfies  the area-constrained Euler-Lagrange equation for the Willmore functional  (see for instance \cite{Lamm} for a derivation of the formula):
\begin{equation}\label{eq:10}
    2\Delta H + H(H^2 - 4D + 2\Ric(N,N)) = \lambda H
\end{equation}
where $\lambda \in \R$ plays the role of Lagrange multiplier and $N$ is the inward pointing unit normal vector. As proved in \cite[Lemma 2.2]{MondinoLaurain}, the Lagrange multiplier $\lambda$ in \eqref{eq:10} remains bounded under the assumptions of the lemma. Using the geometric expansions of Section \ref{SS:GeomExp}, one can check that \eqref{eq:10} for a perturbed geodesic sphere $S_{p,\rho}(w)$ gives (see for instance \cite[Proposition 3.2]{Mondino2})
\begin{equation}\label{eq:ExpW'=0}
\Delta_{{\mathbb S}^{2}}(\Delta_{{\mathbb S}^{2}}+2)w+\cO_{p}(\rho^{2})+ \rho^{2} \cL_{p}^{(4)}(w)+\cQ_{p}^{(2)(4)}(w)=0.
\end{equation}
In particular, setting $P:L^{2}({\mathbb S}^{2})\to \Ker[\Delta_{{\mathbb S}^{2}}(\Delta_{{\mathbb S}^{2}} +2)]^{\perp}$ the orthogonal projection, a fortiori \eqref{eq:ExpW'=0} yields
\begin{equation} \label{eq:diffw}
P \Big[  \Delta _{{\mathbb S}^2}(\Delta _{{\mathbb S}^2} + 2 ) w +\cO_{p}(\rho^2) +\rho^{2} \cL_{p}^{(4)}(w)+ \cQ_{p}^{(2)(4)}(w) \Big] =0.
\end{equation}
Since the operator  $ \Delta _{{\mathbb S}^2}(\Delta_{{\mathbb S}^2} + 2 )$ is invertible on the space orthogonal to its Kernel and $w \in C^{4,\alpha}({\mathbb S}^2)^\perp=\Ker[ \Delta _{{\mathbb S}^2}(\Delta _{{\mathbb S}^2} + 2 )]^\perp \cap C^{4,\alpha}({\mathbb S}^2)$, setting 
$$K:= [\Delta _{{\mathbb S}^2}(\Delta_{{\mathbb S}^2} + 2 )]^{-1}: \Ker[ \Delta_{{\mathbb S}^2}(\Delta_{{\mathbb S}^2} + 2 )]^\perp\subseteq L^2({\mathbb S}^2) \to  \Ker[ \Delta_{{\mathbb S}^2}(\Delta_{{\mathbb S}^2} + 2 )]^\perp,$$ 
 equation \eqref{eq:diffw} is equivalent to the fixed point problem
\begin{equation} \label{eq:puntofissow}
 w = K\circ P \left[\cO_{p}(\rho^2) +\rho^{2} \cL_{p}^{(4)}(w)+ \cQ_{p}^{(2)(4)}(w)\right]=:F_{p,\rho} (w).
\end{equation}
Using Schauder estimates one can check that,  for every fixed compact set ${\mathcal K} \Subset M$, there exist $\rho_0>0$ and $r>0$ such that for all $p \in {\mathcal K}$ and $\rho\in [0,\rho_0]$ the map
$$F_{p,\rho}:B(0,r)\subset C^{4,\a}({\mathbb S}^2)^\perp \to C^{4,\a}({\mathbb S}^2)^\perp$$
is a contraction (see \cite[Lemma 4.4]{Mondino2} for the details). Thus, for every  $p \in {\mathcal K}$ and $\rho\in (0,\rho_0]$ there exists a unique $w_{p,\rho}\in B(0,r)\subset C^{4,\a}({\mathbb S}^2)^\perp$ such that  the surface $S_{p,\rho}(w_{p,\rho})$ is an area-constrained Willmore surface. By continuous dependence on parameters of fixed points in contractions, it also follows that  $\lim_{\rho\to 0} \|w_{p,\rho}\|_{C^{4,\alpha}({\mathbb S}^{2})}=0$ (again, see \cite[Lemma 4.4]{Mondino2} for the details).
\\

\textbf{Step 2}: we show that  the expansion \eqref{eq:estw} holds. 
\\We set the ansatz 
\begin{equation}\label{eq:ansatzw}
w_{p,\rho}=\rho^{2} \bar{w}_{p}+\cO_{p}(\rho^{3})
\end{equation}
 where $\bar{w}_{p}\in  C^{4,\a}({\mathbb S}^2)^\perp$ depends on $p$ but not on $\rho$ and $\limsup_{\rho\to 0} \rho^{-3} \left\| \cO_{p}(\rho^{3}) \right \|_{ C^{4,\a}({\mathbb S}^2)}<\infty$. In order to show that $w_{p,\rho}$ given in step 1 satisfies the ansatz \eqref{eq:ansatzw} and the expansion  \eqref{eq:estw}, we need to improve the expansion of the Euler-Lagrange equation of the Willmore functional \eqref{eq:10}-\eqref{eq:ExpW'=0}.
 To this aim, using the expansions of Section \ref{SS:GeomExp}, one can check that (see the proof of \cite[Proposition 3.9]{Mondino} for more details)
 \begin{align*}
 \Delta_{S_{p,\rho}(w)} H&= \frac{1}{\rho^{3}} \Delta_{{\mathbb S}^{2}}(\Delta_{{\mathbb S}^{2}}+2 ) w-\frac{1}{3\rho}  \Delta_{{\mathbb S}^{2}} \Ric(\Theta, \Theta) + \cO_{p}(\rho^{0})+\frac{1}{\rho^{2}} \cL_{p}^{(4)}(w)+\frac{1}{\rho^{3}} \cQ_{p}^{(2)(4)}(w),\\
 H^{2}-4D&=  \cO_{p}(\rho^{2})+ \cL_{p}^{(2)}(w)+\frac{1}{\rho^{2}} \cQ_{p}^{(2)(2)}(w),
 \end{align*}
 which, plugged into \eqref{eq:10}, give
\begin{equation}\label{eq:26}
    \Delta_{{\mathbb S}^2}(\Delta_{{\mathbb S}^2} + 2)w = \Big( \frac{1}{3}\Delta_{{\mathbb S}^2}\Ric(\Theta,\Theta) - 2\Ric(\Theta,\Theta) + \lambda \Big)\rho^2 + \cO_{p}(\rho^3) + \rho \cL_{p}^{(4)}(w) + \cQ_{p}^{(2)(4)}(w).
\end{equation}
Inserting  the ansatz \eqref{eq:ansatzw}, that we write  as $w = \bar{w}\rho^2 + \cO_{p}(\rho^3)$, in \eqref{eq:26} yields
\begin{equation}\label{eq:34}
    \Delta_{{\mathbb S}^2}(\Delta_{{\mathbb S}^2} + 2)\bar{w}  =  \frac{1}{3}\Delta_{{\mathbb S}^2}\Ric(\Theta,\Theta) - 2\Ric(\Theta,\Theta) + \lambda.
\end{equation}
We solve this PDE by Fourier methods, using the knowledge of the eigenfunctions of $\Delta_{{\mathbb S}^2}$.  To this aim, writing the radial unit vector as $\Theta = x^{\mu}E_{\mu} \in T_pM$ where $\{E_{\mu}\}_{\mu=1,2,3}$ is an orthonormal basis of $T_{p}M$, we have
\begin{align*}
    \Ric(\Theta,\Theta) &= \Ric(x^{\mu}E_{\mu},x^{\nu}E_{\nu}) = R_{\mu\nu}x^{\mu}x^{\nu} = \sum_{\mu \neq \nu}R_{\mu\nu}x^{\mu}x^{\nu} + \sum_{\mu}R_{\mu\mu}x^{\mu}x^{\mu} \\ &= \sum_{\mu \neq \nu}R_{\mu\nu}x^{\mu}x^{\nu} + \sum_{\mu}R_{\mu\mu}\Big((x^{\mu})^2 - \frac{1}{3}\Big) + \frac{1}{3}\sum_{\mu}R_{\mu\mu}, 
  \end{align*}
where $\sum_{\mu}R_{\mu\mu} = \Sc(p)$ is the scalar curvature at $p$ and where we used that $\sum_{\mu} (x^{\mu})^{2}=1$, since $\Theta\in {\mathbb S}^{2}$.
\\Recall that the eigenfunctions of $\Delta_{{\mathbb S}^{2}}$ relative to the second eigenvalue $\lambda_{2}=-6$ are $x^{\mu}x^{\nu}$, $\mu\neq \nu$, and $(x^{\mu})^{2}-(x^{\nu})^{2}$, $\mu\neq \nu$, hence
\begin{equation*}
    (x^1)^2 - \frac{1}{3} = \frac{1}{3}\Big([(x^1)^2 - (x^2)^2] + [(x^1)^2 - (x^3)^2]\Big)
\end{equation*}
is an element of the eigenspace relative to  $\lambda_{2}=-6$  (and analogously for the others $(x^{\mu})^{2}$).
Therefore, 
\begin{equation*}
 \Ric(\Theta,\Theta) - \frac{1}{3}\Sc(p) = \sum_{\mu \neq \nu}R_{\mu\nu}x^{\mu}x^{\nu} + \sum_{\mu}R_{\mu\mu}\Big((x^{\mu})^2 - \frac{1}{3}\Big) 
\end{equation*}
is an eigenfunction of $\Delta_{{\mathbb S}^{2}}$ with eigenvalue $-6$.  We can then rewrite equation \eqref{eq:34} as
\begin{align*}
    \Delta_{{\mathbb S}^2}(\Delta_{{\mathbb S}^2} + 2) \bar{w}  &=  \frac{1}{3}\Delta_{{\mathbb S}^2} \left(\Ric(\Theta,\Theta) - \frac{1}{3}\Sc(p)\right) - 2\Ric(\Theta,\Theta) + \lambda  \\
     &= - 4\left(\Ric(\Theta,\Theta) - \frac{1}{3}\Sc(p)\right) - \frac{2}{3}\Sc(p) + \lambda.
\end{align*}
Setting the value of the Lagrange multiplier as  $\lambda = \frac{2}{3}\Sc(p)$ and noting that
\begin{equation*}
    \Delta_{{\mathbb S}^2}(\Delta_{{\mathbb S}^2} + 2)\left[-\frac{1}{6} \left(\Ric(\Theta,\Theta) - \frac{1}{3}\Sc(p)\right)\right] = -4\left(\Ric(\Theta,\Theta) - \frac{1}{3}\Sc(p)\right)
\end{equation*}
we conclude that
\begin{equation}\label{eq:valuebarw}
 \bar{w} = -\frac{1}{6}\Ric(\Theta,\Theta) + \frac{1}{18}\Sc(p).
\end{equation}
Summarising, we showed that $w_{p,\rho}$ as in the ansatz \eqref{eq:ansatzw} with $\bar{w}_{p}=\bar{w}$ given in \eqref{eq:valuebarw} solves \eqref{eq:26} or, equivalently, \eqref{eq:10}. The proof is complete  once we recall that, by step 1, the solution  $w$ of  \eqref{eq:26} is unique provided $\|w\|_{C^{4,\alpha}({\mathbb S}^{2})}<r$.
\end{proof}

 \subsection{Computation of the Hawking mass}

\begin{prop}\label{prop:expHawk}
Let $S_{p,\rho}(w_{p,\rho})$ we an optimally perturbed geodesic sphere, i.e. $w_{p,\rho}$ is given in  \eqref{eq:estw}. Then the Hawking mass of  $S_{p,\rho}(w_{p,\rho})$ has the following expansion:
\begin{align}
    m_H(S_{p,\rho}(w_{p,\rho})) &:= \sqrt{\frac{|S_{p,\rho}(w_{p,\rho})|}{(16\pi)^3}}\Big(16\pi - W(S_{p,\rho}(w_{p,\rho}))\Big) \nonumber \\ 
    %&= \sqrt{\frac{|S_{p,\rho}(w)|}{(16\pi)^3}} \left(\frac{8\pi}{3}\Sc(p)\rho^2 + \left[\frac{4\pi}{15}\Delta \Sc(p) +  \frac{16\pi}{45}\norm{S_{p}}^2- \frac{4\pi}{27}\Sc(p)^2\right]\rho^4 + \cO_{p}(\rho^5)\right)\\
    &=\frac{1}{12} \Sc_p\rho^{3} +  \left(\frac{1}{120} \Delta \Sc(p) + \frac{1}{90}\norm{S_p}^2 - \frac{1}{144}\Sc_p^2  \right)  \rho^5 + \cO_{p}(\rho^6),  \label{eq:mHSprhow}
\end{align}
where $\Sc$ is the scalar curvature and $S:= \Ric-\frac{1}{3} \Sc\; g$ is the traceless Ricci tensor.
\end{prop} 
 
 \begin{proof}
 In order to keep notation short, throughout the proof we will write $w$ in place of $w_{p,\rho}$.
 \smallskip
 
 \textbf{Step 1}: the Willmore functional integrand.
\\  To compute the Hawking mass, we find the expansion for the Willmore functional $W(S_{p,\rho}(w)) := \int_{S_{p,\rho}(w)}H^2\ dV = \int_{{\mathbb S}^2}H^2\sqrt{\det\mathring{g}}\ d\theta^1d\theta^2$. Using \eqref{eq:Expdetg}, the Taylor expansion $\sqrt{1+x} = 1 + \frac{x}{2} - \frac{x^2}{8} + \cO(x^3)$ and the fact that $w = \cO_{p}(\rho^2)$, we have:
\begin{align}\label{eq:61}
    \sqrt{\det\mathring{g}} &= \sin{\theta^1}\rho^2\Bigg[(1-w)^2 + \frac{1}{2}g_{{\mathbb S}^2}^{ij}w_iw_j - \frac{1}{6}\Ric(\Theta,\Theta)\rho^2 + \frac{2}{3}w\Ric(\Theta,\Theta)\rho^2 \nonumber\\
    &\quad+ \frac{1}{12}g_{{\mathbb S}^2}^{ij}g(\nabla_{\Theta}{\mathcal R}(\Theta,\Theta_i)\Theta,\Theta_j)\rho^3 + \frac{1}{40}g_{{\mathbb S}^2}^{ij}g(\nabla_{\Theta \Theta}^2{\mathcal R}(\Theta,\Theta_i)\Theta,\Theta_j)\rho^4 \\ 
    &\quad+ \frac{1}{45}\delta^{\mu\nu} g_{{\mathbb S}^2}^{ij} g({\mathcal R}(\Theta,\Theta_i)\Theta,E_{\mu})g({\mathcal R}(\Theta,\Theta_j)\Theta,E_{\nu})\rho^4 - \frac{1}{18}g({\mathcal R}(\Theta,\Theta_1)\Theta,\Bar{\Theta}_2)^2\rho^4 \nonumber\\
    &\quad+ \frac{1}{18}g({\mathcal R}(\Theta,\Theta_1)\Theta,\Theta_1)g({\mathcal R}(\Theta,\Bar{\Theta}_2)\Theta,\Bar{\Theta}_2)\rho^4 - \frac{1}{72}\Ric(\Theta,\Theta)^2\rho^4\Bigg] + \cO_{p}(\rho^7)\nonumber.
\end{align}
Multiplying (\ref{eq:61}) with \eqref{eq:ExpH2} we obtain the integrand of the Willmore functional evaluated on a perturbed geodesic sphere: 
\begin{align*}
    H^2\sqrt{\det\mathring{g}} =& \sin{\theta^1}\Bigg[\Bigg[4 + 4\Delta_{{\mathbb S}^2}w + 4w\Delta_{{\mathbb S}^2}w + (\Delta_{{\mathbb S}^2}w)^2 + 2g_{{\mathbb S}^2}^{ij}w_iw_j\Bigg] \\ &+\Bigg[\frac{2}{3}w_kg_{{\mathbb S}^2}^{ij}g_{{\mathbb S}^2}^{kn}g_{{\mathbb S}^2}^{ml}g({\mathcal R}(\Theta,\Theta_n)\Theta,\Theta_m)\Big(\partial_ig^{{\mathbb S}^2}_{jl} + \partial_jg^{{\mathbb S}^2}_{il} - \partial_lg^{{\mathbb S}^2}_{ij}\Big) \\
     &-\frac{2}{3}w_kg_{{\mathbb S}^2}^{ij}g_{{\mathbb S}^2}^{kl}\Big(\partial_ig({\mathcal R}(\Theta,\Theta_j)\Theta,\Theta_l) + \partial_jg({\mathcal R}(\Theta,\Theta_i)\Theta,\Theta_l) - \partial_lg({\mathcal R}(\Theta,\Theta_i)\Theta,\Theta_j)\Big) \\ 
     &-\frac{4}{3}g_{{\mathbb S}^2}^{il}g_{{\mathbb S}^2}^{kj}g({\mathcal R}(\Theta,\Theta_l)\Theta,\Theta_k)(\Hess_{{\mathbb S}^2}(w))_{ij} - 2\Ric(\Theta,\Theta) + 4w\Ric(\Theta,\Theta) - \frac{4}{3}\Ric(\Theta,\Theta)\Delta_{{\mathbb S}^2}w\Bigg]\rho^2 \\ 
     &+\Bigg[\frac{4}{3}g_{{\mathbb S}^2}^{ij}g(\nabla_{\Theta}{\mathcal R}(\Theta,\Theta_i)\Theta,\Theta_j)\Bigg]\rho^3 \\ 
     &+\Bigg[\frac{1}{2}g_{{\mathbb S}^2}^{ij}g(\nabla_{\Theta \Theta}^2{\mathcal R}(\Theta,\Theta_i)\Theta,\Theta_j) + \frac{4}{9} \delta^{\mu\nu}g_{{\mathbb S}^2}^{ij}g({\mathcal R}(\Theta,\Theta_i)\Theta,E_{\mu})g({\mathcal R}(\Theta,\Theta_j)\Theta,E_{\nu}) \\ 
     &\quad- \frac{4}{9}g_{{\mathbb S}^2}^{il}g_{{\mathbb S}^2}^{kj}g({\mathcal R}(\Theta,\Theta_i)\Theta,\Theta_j)g({\mathcal R}(\Theta,\Theta_l)\Theta,\Theta_k) + \frac{2}{9}g({\mathcal R}(\Theta,\Theta_1)\Theta,\Theta_1)g({\mathcal R}(\Theta,\Bar{\Theta}_2)\Theta,\Bar{\Theta}_2) \\ &\quad- \frac{2}{9}g({\mathcal R}(\Theta,\Theta_1)\Theta,\Bar{\Theta}_2)^2 + \frac{5}{18}\Ric(\Theta,\Theta)^2\Bigg]\rho^4 + \cO_{p}(\rho^5)\Bigg].
\end{align*}
Inserting $w = \bar{w}\rho^2 + \cO_{p}(\rho^3)$ yields:
\begin{align}\label{eq:74}
    H^2\sqrt{\det\mathring{g}} =& \sin{\theta^1}\Bigg[4 + \Bigg[4\Delta_{{\mathbb S}^2}\bar{w} - 2\Ric(\Theta,\Theta)\Bigg]\rho^2 + \Bigg[\frac{4}{3}g_{{\mathbb S}^2}^{ij}g(\nabla_{\Theta}{\mathcal R}(\Theta,\Theta_i)\Theta,\Theta_j)\Bigg]\rho^3 \nonumber\\
    &\quad+ \Bigg[4\bar{w}\Delta_{{\mathbb S}^2}\bar{w} + (\Delta_{{\mathbb S}^2}\bar{w})^2 + 2g_{{\mathbb S}^2}^{ij}\bar{w}_i\bar{w}_j +4 \bar{w}\Ric(\Theta,\Theta) - \frac{4}{3}\Ric(\Theta,\Theta)\Delta_{{\mathbb S}^2}\bar{w} \nonumber\\ 
    &\quad+\frac{2}{3}\bar{w}_kg_{{\mathbb S}^2}^{ij}g_{{\mathbb S}^2}^{kn}g_{{\mathbb S}^2}^{ml}g({\mathcal R}(\Theta,\Theta_n)\Theta,\Theta_m)\Big(\partial_ig^{{\mathbb S}^2}_{jl} + \partial_jg^{{\mathbb S}^2}_{il} - \partial_lg^{{\mathbb S}^2}_{ij}\Big) \nonumber\\
     &\quad-\frac{2}{3}\bar{w}_kg_{{\mathbb S}^2}^{ij}g_{{\mathbb S}^2}^{kl}\Big(\partial_ig({\mathcal R}(\Theta,\Theta_j)\Theta,\Theta_l) + \partial_jg({\mathcal R}(\Theta,\Theta_i)\Theta,\Theta_l) - \partial_lg({\mathcal R}(\Theta,\Theta_i)\Theta,\Theta_j)\Big) \\ 
    &\quad-\frac{4}{3}g_{{\mathbb S}^2}^{il}g_{{\mathbb S}^2}^{kj}g({\mathcal R}(\Theta,\Theta_l)\Theta,\Theta_k)(\Hess_{{\mathbb S}^2}\bar{w})_{ij} - \frac{2}{9}g({\mathcal R}(\Theta,\Theta_1)\Theta,\Bar{\Theta}_2)^2 + \frac{5}{18}\Ric(\Theta,\Theta)^2 \nonumber\\ 
    &\quad- \frac{4}{9}g_{{\mathbb S}^2}^{il}g_{{\mathbb S}^2}^{kj}g({\mathcal R}(\Theta,\Theta_i)\Theta,\Theta_j)g({\mathcal R}(\Theta,\Theta_l)\Theta,\Theta_k) + \frac{2}{9}g({\mathcal R}(\Theta,\Theta_1)\Theta,\Theta_1)g({\mathcal R}(\Theta,\Bar{\Theta}_2)\Theta,\Bar{\Theta}_2) \nonumber\\
     &\quad+\frac{1}{2}g_{{\mathbb S}^2}^{ij}g(\nabla_{\Theta \Theta}^2{\mathcal R}(\Theta,\Theta_i)\Theta,\Theta_j) + \frac{4}{9} \delta^{\mu\nu}g_{{\mathbb S}^2}^{ij}g({\mathcal R}(\Theta,\Theta_i)\Theta,E_{\mu})g({\mathcal R}(\Theta,\Theta_j)\Theta,E_{\nu})\Bigg]\rho^4 + \cO_{p}(\rho^5)\Bigg]\nonumber
\end{align}
%
%We next work to simplify the terms multiplying $\rho^{4}$.  We will use the following identities, which follow from the definition of 
%the Ricci curvature and the fact that $\Theta, \Theta_1$ and $\Bar{\Theta}_2$ form an orthonormal basis of $T_pM$:
%\begin{align}\label{eq:35}
%    g({\mathcal R}(\Theta,\Theta_1)\Theta,\Theta_1) &\nonumber= - (\Ric(\Theta_1,\Theta_1) + g({\mathcal R}(\Theta_1,\Bar{\Theta}_2)\Theta_1,\Bar{\Theta}_2)) \\ 
%    g({\mathcal R}(\Theta,\Bar{\Theta}_2)\Theta,\Bar{\Theta}_2) &\nonumber= - (\Ric(\Bar{\Theta}_2,\Bar{\Theta}_2) + g({\mathcal R}(\Theta_1,\Bar{\Theta}_2)\Theta_1,\Bar{\Theta}_2)) \\ 
%    g({\mathcal R}(\Theta,\Theta_1)\Theta,\Bar{\Theta}_2) & = - \Ric(\Theta_1,\Bar{\Theta}_2) \\ 
%    \Ric(\Theta_1,\Theta_1) + \Ric(\Bar{\Theta}_2,\Bar{\Theta}_2) &\nonumber= \Sc(p) - \Ric(\Theta,\Theta) \\
%     g({\mathcal R}(\Theta_1,\Bar{\Theta}_2)\Theta_1,\Bar{\Theta}_2) &\nonumber= - \frac{1}{2}\Sc(p) + \Ric(\Theta,\Theta).
%\end{align}
%%

 \textbf{Step 2}: simplifying the Willmore functional integrand.
\\ We will use the following computations for the derivatives of $\bar{w}= -\frac{1}{6}\Ric(\Theta,\Theta) + \frac{1}{18}\Sc(p)$:
\begin{align}\label{eq:63}
    \bar{w}_k &:= \partial_k\left(-\frac{1}{6}\Ric(\Theta,\Theta) + \frac{1}{18}\Sc(p)\right) = -\frac{1}{6}\partial_k(\Ric(\Theta,\Theta))= -\frac{1}{3}\Ric(\Theta,\Theta_k) \\
    \bar{w}_{kj} &:= \partial_j\left(-\frac{1}{3}\Ric(\Theta,\Theta_k)\right) = -\frac{1}{3}\left(\Ric(\Theta_j,\Theta_k) + \Ric(\Theta,\Theta_{kj})\right)\nonumber
\end{align}
which, combined with the fact that
\begin{align}\label{eq:65}
    \Theta_{11} &= -\Theta \nonumber\\ \Theta_{12} &= \Theta_{21} = \cot{\theta^1}\Theta_2 \\ \Theta_{22} &= -\sin{\theta^1}\cos{\theta^1}\Theta_1 - \sin^2{\theta^1}\Theta\nonumber
\end{align}
yields
\begin{align}\label{eq:64}
    \bar{w}_{11} &= -\frac{1}{3}(\Ric(\Theta_1,\Theta_1) - \Ric(\Theta,\Theta)) \nonumber\\ 
    \bar{w}_{12} &= \bar{w}_{21} = -\frac{1}{3}(\Ric(\Theta_1,\Theta_2) + \cot{\theta^1}\Ric(\Theta,\Theta_2)) \\
     \bar{w}_{22} &= -\frac{1}{3}(\Ric(\Theta_2,\Theta_2) - \sin{\theta^1}\cos{\theta^1}\Ric(\Theta,\Theta_1) - \sin^2{\theta^1}\Ric(\Theta,\Theta)).\nonumber
\end{align}
Using  (\ref{eq:63})  and recalling that $\Delta_{{\mathbb S}^{2}} \bar{w}=-6 \bar{w}$,  we can rewrite the first line of the terms multiplying $\rho^{4}$ in \eqref{eq:74}  as:
\begin{align}\label{eq:73}
    &4\bar{w}\Delta_{{\mathbb S}^2}\bar{w} + (\Delta_{{\mathbb S}^2}\bar{w})^2 + 2g_{{\mathbb S}^2}^{ij}\bar{w}_i\bar{w}_j + 4\bar{w}\Ric(\Theta,\Theta) - \frac{4}{3}\Ric(\Theta,\Theta)\Delta_{{\mathbb S}^2}\bar{w} \nonumber\\
%    &= (-\frac{2}{3}\Ric(\Theta,\Theta) + \frac{2}{9}\Sc(p))(\Ric(\Theta,\Theta) - \frac{1}{3}\Sc(p)) + (\Ric(\Theta,\Theta) - \frac{1}{3}\Sc(p))^2 \nonumber\\&\quad+ 2g_{{\mathbb S}^2}^{ij}(-\frac{1}{6}\Ric(\Theta,\Theta))_i(-\frac{1}{6}\Ric(\Theta,\Theta))_j + \Ric(\Theta,\Theta)(-\frac{2}{3}\Ric(\Theta,\Theta) + \frac{2}{9}\Sc(p)) \nonumber\\&\quad- \frac{4}{3}\Ric(\Theta,\Theta)(\Ric(\Theta,\Theta) - \frac{1}{3}\Sc(p)) \nonumber\\
    &= -\frac{5}{3}\Ric(\Theta,\Theta)^2 + \frac{4}{9}\Sc(p)\Ric(\Theta,\Theta) + \frac{1}{27}\Sc(p)^2 + \frac{2}{9}(\Ric(\Theta,\Theta_1)^2 + \Ric(\Theta,\Bar{\Theta}_2)^2).
\end{align}
The second line of the terms multiplying $\rho^{4}$ in \eqref{eq:74} can be rewritten as:
\begin{align}\label{eq:66}
    &\frac{2}{3}\bar{w}_kg_{{\mathbb S}^2}^{ij}g_{{\mathbb S}^2}^{kn}g_{{\mathbb S}^2}^{ml}g({\mathcal R}(\Theta,\Theta_n)\Theta,\Theta_m)\Big(\partial_ig^{{\mathbb S}^2}_{jl} + \partial_jg^{{\mathbb S}^2}_{il} - \partial_lg^{{\mathbb S}^2}_{ij}\Big) \nonumber\\&= \sum_k - \frac{2}{3}\bar{w}_kg_{{\mathbb S}^2}^{22}g_{{\mathbb S}^2}^{kk}g_{{\mathbb S}^2}^{11}g({\mathcal R}(\Theta,\Theta_k)\Theta,\Theta_1)(2\sin{\theta^1}\cos{\theta^1}) \nonumber\\&= \frac{4\cot{\theta^1}}{9}\Ric(\Theta,\Theta_1)g({\mathcal R}(\Theta,\Theta_1)\Theta,\Theta_1) + \frac{4\cot{\theta^1}}{9}\Ric(\Theta,\Bar{\Theta}_2)g({\mathcal R}(\Theta,\Bar{\Theta}_2)\Theta,\Theta_1),
\end{align}
where the first equality follows because the only non-zero terms occur when $i=j$, $k=n$ and $m=l$ (note that $\partial_ig^{{\mathbb S}^2}_{jk}$ is only non-zero when $j=k=2$ and $i=1$), and for the second equality we  used (\ref{eq:63}). 
\\The third line of the terms multiplying $\rho^{4}$ in \eqref{eq:74} can be rewritten as:
\begin{align}\label{eq:67}
    &-\frac{2}{3}\bar{w}_kg_{{\mathbb S}^2}^{ij}g_{{\mathbb S}^2}^{kl}\Big(\partial_ig({\mathcal R}(\Theta,\Theta_j)\Theta,\Theta_l) + \partial_jg({\mathcal R}(\Theta,\Theta_i)\Theta,\Theta_l) - \partial_lg({\mathcal R}(\Theta,\Theta_i)\Theta,\Theta_j)\Big) \nonumber\\&= \sum_{i,k} -\frac{2}{3}\bar{w}_kg_{{\mathbb S}^2}^{ii}g_{{\mathbb S}^2}^{kk}\Big(2\partial_ig({\mathcal R}(\Theta,\Theta_i)\Theta,\Theta_k) - \partial_kg({\mathcal R}(\Theta,\Theta_i)\Theta,\Theta_i)\Big) \nonumber\\&= \sum_{i,k} -\frac{2}{3}\bar{w}_kg_{{\mathbb S}^2}^{ii}g_{{\mathbb S}^2}^{kk}\Big[2\big(g({\mathcal R}(\Theta,\Theta_{ii})\Theta,\Theta_k) + g({\mathcal R}(\Theta,\Theta_i)\Theta_i,\Theta_k) + g({\mathcal R}(\Theta,\Theta_i)\Theta,\Theta_{ki})\big) \nonumber\\&\quad- g({\mathcal R}(\Theta_k,\Theta_i)\Theta,\Theta_i) - g({\mathcal R}(\Theta,\Theta_{ik})\Theta,\Theta_i) - g({\mathcal R}(\Theta,\Theta_i)\Theta_k,\Theta_i) - g({\mathcal R}(\Theta,\Theta_i)\Theta,\Theta_{ik})\Big] \nonumber\\&= -\frac{2}{3}\bar{w}_2g_{{\mathbb S}^2}^{11}g_{{\mathbb S}^2}^{22}\Big[2\big(g({\mathcal R}(\Theta,\Theta_1)\Theta_1,\Theta_2) + g({\mathcal R}(\Theta,\Theta_1)\Theta,\Theta_{12})\big) \nonumber\\&\quad- g({\mathcal R}(\Theta_2,\Theta_1)\Theta,\Theta_1) - g({\mathcal R}(\Theta,\Theta_{12})\Theta,\Theta_1) - g({\mathcal R}(\Theta,\Theta_1)\Theta_2,\Theta_1) - g({\mathcal R}(\Theta,\Theta_1)\Theta,\Theta_{12})\Big] \nonumber\\&\quad- \frac{2}{3}\bar{w}_1g_{{\mathbb S}^2}^{22}g_{{\mathbb S}^2}^{11}\Big[2\big(g({\mathcal R}(\Theta,\Theta_{22})\Theta,\Theta_1) + g({\mathcal R}(\Theta,\Theta_2)\Theta_2,\Theta_1) + g({\mathcal R}(\Theta,\Theta_2)\Theta,\Theta_{21})\big) \nonumber\\&\quad- g({\mathcal R}(\Theta_1,\Theta_2)\Theta,\Theta_2) - g({\mathcal R}(\Theta,\Theta_{21})\Theta,\Theta_2) - g({\mathcal R}(\Theta,\Theta_2)\Theta_1,\Theta_2) - g({\mathcal R}(\Theta,\Theta_2)\Theta,\Theta_{21})\Big] \nonumber\\&\quad- \frac{2}{3}\bar{w}_2g_{{\mathbb S}^2}^{22}g_{{\mathbb S}^2}^{22}\Big[2\big(g({\mathcal R}(\Theta,\Theta_{22})\Theta,\Theta_2) + g({\mathcal R}(\Theta,\Theta_2)\Theta,\Theta_{22})\big) \nonumber\\&\quad- g({\mathcal R}(\Theta,\Theta_{22})\Theta,\Theta_2) - g({\mathcal R}(\Theta,\Theta_2)\Theta,\Theta_{22})\Big] \nonumber\\&= \frac{8}{9}\Ric(\Theta,\Bar{\Theta}_2)g({\mathcal R}(\Theta,\Theta_1)\Theta_1,\Bar{\Theta}_2) \nonumber\\&\quad +\frac{4}{9}\Ric(\Theta,\Theta_1)\big(-\cot{\theta^1}g({\mathcal R}(\Theta,\Theta_1)\Theta,\Theta_1) + 2g({\mathcal R}(\Theta,\Bar{\Theta}_2)\Bar{\Theta}_2,\Theta_1)\big) \\&\quad -\frac{4\cot{\theta^1}}{9}\Ric(\Theta,\Bar{\Theta}_2)g({\mathcal R}(\Theta,\Theta_1)\Theta,\Bar{\Theta}_2). \nonumber
\end{align}
For the second equality above we have used that $\Theta_{11} = -\Theta$ which means that all the $i=k=1$ terms are zero thanks to the symmetries of the Riemann tensor. The last equality follows by applying (\ref{eq:65}) and (\ref{eq:63}). 
\\Combining (\ref{eq:66}) and (\ref{eq:67}) shows that the second and third lines of the fourth order term in \eqref{eq:74} become:
\begin{align}\label{eq:68}
    \frac{8}{9}\Big(\Ric(\Theta,\Bar{\Theta}_2)g({\mathcal R}(\Theta,\Theta_1)\Theta_1,\Bar{\Theta}_2) + \Ric(\Theta,\Theta_1)g({\mathcal R}(\Theta,\Bar{\Theta}_2)\Bar{\Theta}_2,\Theta_1)\Big) = \frac{8}{9}\Big(\Ric(\Theta,\Bar{\Theta}_2)^2 + \Ric(\Theta,\Theta_1)^2\Big)
\end{align}
where we have used the definition of the Ricci tensor and the symmetries of the Riemann tensor to rewrite $\Ric(\Theta,\Bar{\Theta}_2) = g({\mathcal R}(\Theta,\Theta_1)\Theta_1,\Bar{\Theta}_2)$ and $\Ric(\Theta,\Theta_1) = g({\mathcal R}(\Theta,\Bar{\Theta}_2)\Bar{\Theta}_2,\Theta_1)$.
\\ Turning now to the next term in \eqref{eq:74}, we have:  
\begin{align}\label{eq:72}
    &-\frac{4}{3}g_{{\mathbb S}^2}^{il}g_{{\mathbb S}^2}^{kj}g({\mathcal R}(\Theta,\Theta_l)\Theta,\Theta_k)(\Hess_{{\mathbb S}^2}\bar{w})_{ij} \nonumber\\
   & = -\frac{4}{3}g({\mathcal R}(\Theta,\Theta_1)\Theta,\Theta_1)\big(\bar{w}_{11} - \Gamma^i_{11}\bar{w}_i\big) -\frac{8}{3\sin^2{\theta^1}}g({\mathcal R}(\Theta,\Theta_1)\Theta,\Theta_2)\big(\bar{w}_{12} - \Gamma^i_{12}\bar{w}_i\big) \nonumber\\
    & \quad -\frac{4}{3\sin^4{\theta^1}}g({\mathcal R}(\Theta,\Theta_2)\Theta,\Theta_2)\big(\bar{w}_{22} - \Gamma^i_{22}\bar{w}_i\big) \nonumber\\
    &= \frac{4}{9}g({\mathcal R}(\Theta,\Theta_1)\Theta,\Theta_1)\big(\Ric(\Theta_1,\Theta_1) - \Ric(\Theta,\Theta)\big) +\frac{8}{9}g({\mathcal R}(\Theta,\Theta_1)\Theta,\Bar{\Theta}_2)\Ric(\Bar{\Theta}_2,\Theta_1) \nonumber\\
    &  \quad +\frac{4}{9}g({\mathcal R}(\Theta,\Bar{\Theta}_2)\Theta,\Bar{\Theta}_2)\big(\Ric(\Bar{\Theta}_2,\Bar{\Theta}_2) - \Ric(\Theta,\Theta)\big) \nonumber\\
    &= \frac{4}{9}\big(g({\mathcal R}(\Theta,\Theta_1)\Theta,\Theta_1)\Ric(\Theta_1,\Theta_1) + g({\mathcal R}(\Theta,\Bar{\Theta}_2)\Theta,\Bar{\Theta}_2)\Ric(\Bar{\Theta}_2,\Bar{\Theta}_2)\big) + \frac{4}{9}\Ric(\Theta,\Theta)^2 - \frac{8}{9}\Ric(\Bar{\Theta}_2,\Theta_1)^2 \nonumber\\
    &= -\frac{4}{9}\Big[\Ric(\Theta_1,\Theta_1)\big(\Ric(\Theta_1,\Theta_1) + g({\mathcal R}(\Theta_1,\Bar{\Theta}_2)\Theta_1,\Bar{\Theta}_2)\big) + \Ric(\Bar{\Theta}_2,\Bar{\Theta}_2)\big(\Ric(\Bar{\Theta}_2,\Bar{\Theta}_2) + g({\mathcal R}(\Theta_1,\Bar{\Theta}_2)\Theta_1,\Bar{\Theta}_2)\big)\Big] \nonumber\\
    &  \quad + \frac{4}{9}\Ric(\Theta,\Theta)^2 - \frac{8}{9}\Ric(\Bar{\Theta}_2,\Theta_1)^2 \nonumber\\
    &   = -\frac{4}{9}\big(\Ric(\Theta_1,\Theta_1)^2 + \Ric(\Bar{\Theta}_2,\Bar{\Theta}_2)^2\big) + \frac{2}{9}\Sc(p)\big(\Ric(\Theta_1,\Theta_1) + \Ric(\Bar{\Theta}_2,\Bar{\Theta}_2)\big) \nonumber\\
    &  \quad- \frac{4}{9}\Ric(\Theta,\Theta)\big(\Ric(\Theta_1,\Theta_1) + \Ric(\Bar{\Theta}_2,\Bar{\Theta}_2)\big) + \frac{4}{9}\Ric(\Theta,\Theta)^2 - \frac{8}{9}\Ric(\Bar{\Theta}_2,\Theta_1)^2 \nonumber\\
    &= -\frac{2}{9}\Sc(p)^2 + \frac{2}{9}\Sc(p)\Ric(\Theta,\Theta) + \frac{8}{9}\Ric(\Theta_1,\Theta_1)\Ric(\Bar{\Theta}_2,\Bar{\Theta}_2) + \frac{4}{9}\Ric(\Theta,\Theta)^2 - \frac{8}{9}\Ric(\Bar{\Theta}_2,\Theta_1)^2
    \end{align}
where for the second equality we used (\ref{eq:63}), (\ref{eq:64}) and the fact that the only non-vanishing Christoffel symbols for ${\mathbb S}^2$ are $\Gamma^2_{12} = \Gamma^2_{21} = \cot{\theta^1}$ and $\Gamma^1_{22} = -\sin{\theta^1}\cos{\theta^1}$. For the third and fourth equalities we used the following identities:
\begin{align}\label{eq:69}
    g({\mathcal R}(\Theta,\Theta_1)\Theta,\Theta_1) &= -\Ric(\Theta_1,\Theta_1) - g({\mathcal R}(\Theta_1,\Bar{\Theta}_2)\Theta_1,\Bar{\Theta}_2) \nonumber\\ g({\mathcal R}(\Theta,\Bar{\Theta}_2)\Theta,\Bar{\Theta}_2) &= -\Ric(\Bar{\Theta}_2,\Bar{\Theta}_2) - g({\mathcal R}(\Theta_1,\Bar{\Theta}_2)\Theta_1,\Bar{\Theta}_2) \\ g({\mathcal R}(\Theta_1,\Bar{\Theta}_2)\Theta_1,\Bar{\Theta}_2) &= -\frac{1}{2}\Sc(p) + \Ric(\Theta,\Theta) \nonumber\\ \Ric(\Theta_1,\Theta_1) + \Ric(\Bar{\Theta}_2,\Bar{\Theta}_2) &= \Sc(p) - \Ric(\Theta,\Theta)\nonumber
\end{align}
Next, we note that the eleventh and fourteenth $4^{th}$ order terms in \eqref{eq:74} cancel out as 
\begin{align*}
    &\frac{4}{9}\delta^{\mu \nu}g_{{\mathbb S}^2}^{ij}g({\mathcal R}(\Theta,\Theta_i)\Theta,E_{\mu})g({\mathcal R}(\Theta,\Theta_j)\Theta,E_{\nu}) \\
    &= \frac{4}{9}g_{{\mathbb S}^2}^{ij}\big[g({\mathcal R}(\Theta,\Theta_i)\Theta,\Theta_1)g({\mathcal R}(\Theta,\Theta_j)\Theta,\Theta_1) + g({\mathcal R}(\Theta,\Theta_i)\Theta,\Bar{\Theta}_2)g({\mathcal R}(\Theta,\Theta_j)\Theta,\Bar{\Theta}_2)\big]  \\
    &= \frac{4}{9}g_{{\mathbb S}^2}^{ij}g_{{\mathbb S}^2}^{lk}g({\mathcal R}(\Theta,\Theta_i)\Theta,\Theta_l)g({\mathcal R}(\Theta,\Theta_j)\Theta,\Theta_k)
        \end{align*}
where we have used the orthonormal basis $\{E_1 = \Theta, E_2 = \Theta_1, E_3 = \Bar{\Theta}_2\}$. The ninth term becomes:

\begin{align}\label{eq:70}
    - \frac{2}{9}g({\mathcal R}(\Theta,\Theta_1)\Theta,\Bar{\Theta}_2)^2 = - \frac{2}{9}\Ric(\Theta_1,\Bar{\Theta}_2)^2
\end{align}
 and, finally, the twelfth:
 
 \begin{align}\label{eq:71}
     &\frac{2}{9}g({\mathcal R}(\Theta,\Theta_1)\Theta,\Theta_1)g({\mathcal R}(\Theta,\Bar{\Theta}_2)\Theta,\Bar{\Theta}_2) \nonumber\\
     &= \frac{2}{9}\big[-\Ric(\Theta_1,\Theta_1) - g({\mathcal R}(\Theta_1,\Bar{\Theta}_2)\Theta_1,\Bar{\Theta}_2)\big]\big[-\Ric(\Bar{\Theta}_2,\Bar{\Theta}_2) - g({\mathcal R}(\Theta_1,\Bar{\Theta}_2)\Theta_1,\Bar{\Theta}_2)\big] \nonumber\\
         &=\frac{2}{9} \left[ -\Ric(\Theta_{1}, \Theta_{1})+\frac{1}{2} \Sc(p) -\Ric(\Theta, \Theta) \right]\cdot \left[ -\Ric(\bar{\Theta}_{2}, \bar{\Theta}_{2})+\frac{1}{2} \Sc(p) -\Ric(\Theta, \Theta) \right]  \nonumber \\
     &= \frac{2}{9}  \Ric(\Theta_{1}, \Theta_{1}) \Ric(\bar{\Theta}_{2}, \bar{\Theta}_{2})+ \frac{1}{18} \Sc(p)^{2} -\frac{2}{9} \Ric(\Theta, \Theta) \left[  \Sc(p)- \Ric(\Theta_{1}, \Theta_{1})-\Ric(\bar{\Theta}_{2}, \bar{\Theta}_{2})-  \Ric(\Theta, \Theta) \right]\nonumber \\
     &\quad -\frac{1}{9} \Sc(p) \left[ \Ric(\Theta_1, \Theta_1) + \Ric(\bar\Theta_2, \bar \Theta_2) \right] \nonumber \\
      &=  \frac{2}{9}\Ric(\Theta_1,\Theta_1)\Ric(\Bar{\Theta}_2,\Bar{\Theta}_2) -\frac{1}{18}\Sc(p)^2 + \frac{1}{9}\Sc(p)\Ric(\Theta,\Theta).
 \end{align}
 again using (\ref{eq:69}). Substituting (\ref{eq:73}), \eqref{eq:66}, \eqref{eq:67}, (\ref{eq:68}), (\ref{eq:72}), (\ref{eq:70}) and (\ref{eq:71}) into (\ref{eq:74}) yields:
 
 \begin{align}\label{eq:75}
     H^2\sqrt{\det\mathring{g}} &= \sin{\theta^1}\Bigg[4 + \Bigg[4\Delta_{{\mathbb S}^2}\bar{w} - 2\Ric(\Theta,\Theta)\Bigg]\rho^2 + \Bigg[\frac{4}{3}g_{{\mathbb S}^2}^{ij}g(\nabla_{\Theta}{\mathcal R}(\Theta,\Theta_i)\Theta,\Theta_j)\Bigg]\rho^3 \nonumber\\&\quad+ \Bigg[-\frac{17}{18}\Ric(\Theta,\Theta)^2 -\frac{13}{54}\Sc(p)^2 + \frac{7}{9}\Sc(p)\Ric(\Theta,\Theta) + \frac{10}{9}\Big(\Ric(\Theta,\Theta_1)^2 + \Ric(\Theta,\Bar{\Theta}_2)^2\Big) \\&\quad+ \frac{10}{9}\Big(\Ric(\Theta_1,\Theta_1)\Ric(\Bar{\Theta}_2,\Bar{\Theta}_2) - \Ric(\Theta_1,\Bar{\Theta}_2)^2\Big) + \frac{1}{2}g_{{\mathbb S}^2}^{ij}g(\nabla_{\Theta \Theta}^2{\mathcal R}(\Theta,\Theta_i)\Theta,\Theta_j)\Bigg]\rho^4 + \cO(\rho^5)\Bigg].\nonumber
 \end{align}

 \textbf{Step 3}: computation of the integral of \eqref{eq:75}.
\\ We  now integrate each term, and finally we will combine the results. For the following, it will be useful to observe that
\begin{equation}\label{eq:intCoord}
    \int_{{\mathbb S}^2}(x^{\mu})^2\ dV_{{\mathbb S}^2} = \frac{4\pi}{3}, \quad  \int_{{\mathbb S}^2}(x^{\mu})^2(x^{\nu})^2\ dV_{{\mathbb S}^2} = \frac{4\pi}{15} \;\text{ for } \mu\neq\nu, \quad \int_{{\mathbb S}^2}(x^{\mu})^4\ dV_{{\mathbb S}^2} = \frac{4\pi}{5}. 
\end{equation}
 Recalling that $\Theta = (\sin{\theta^1}\cos{\theta^2}, \sin{\theta^1}\sin{\theta^2}, \cos{\theta^1}) = (x, y, z)$, we have 
 \begin{equation}\label{eq:Theta1Theta2}
 \begin{split}
    \Theta_1 &= (\cos{\theta^1}\cos{\theta^2}, \cos{\theta^1}\sin{\theta^2}, -\sin{\theta^1}) = \Bigg(\frac{xz}{\sqrt{x^2+y^2}}, \frac{yz}{\sqrt{x^2+y^2}}, -\sqrt{x^2+y^2}\Bigg) 
    \\ \Bar{\Theta}_2 &= (-\sin{\theta^2}, \cos{\theta^2},0) = \Bigg(-\frac{y}{\sqrt{x^2+y^2}}, \frac{x}{\sqrt{x^2+y^2}},0\Bigg) .
    \end{split}
\end{equation}
$\bullet$ \textbf{Terms of order $\rho^{2}$}. 
Notice that 
$$0= \int_{{\mathbb S}^{2}} \Delta_{{\mathbb S}^{2}} \bar{w} \, dV_{{\mathbb S}^{2}} = \int \Delta_{{\mathbb S}^{2}} \bar{w} \, \sin \theta^{1}\, d\theta^{1} d\theta^{2}$$
and use \eqref{eq:intCoord} to compute
$$
-2 \int_{{\mathbb S}^{2}} \Ric(\Theta, \Theta) dV_{{\mathbb S}^{2}}(\Theta)= -\frac{8 \pi}{3} \Sc(p).
$$
$\bullet$ \textbf{Terms of order $\rho^{3}$}.
Notice that by anti-symmetry it holds
%$$
%\int_{{\mathbb S}^{2}} g_{{\mathbb S}^2}^{ij}g(\nabla_{\Theta}{\mathcal R}(\Theta,\Theta_i)\Theta,\Theta_j) dV_{{\mathbb S}^{2}}(\Theta)= \int_{{\mathbb S}^{2}} g_{{\mathbb S}^2}^{ij}g(\nabla_{-\Theta}{\mathcal R}(-\Theta,\Theta_i)-\Theta,\Theta_j) dV_{{\mathbb S}^{2}}(-\Theta)=-\int_{{\mathbb S}^{2}} g_{{\mathbb S}^2}^{ij}g(\nabla_{\Theta}{\mathcal R}(\Theta,\Theta_i)\Theta,\Theta_j) dV_{{\mathbb S}^{2}}(\Theta),
%$$
%thus 
$$
\int_{{\mathbb S}^{2}} g_{{\mathbb S}^2}^{ij}g(\nabla_{\Theta}{\mathcal R}(\Theta,\Theta_i)\Theta,\Theta_j) dV_{{\mathbb S}^{2}}(\Theta)=0.
$$
$\bullet$ \textbf{Terms of order $\rho^{4}$}.
\\Using \eqref{eq:intCoord}, one can compute that (see for instance (4.17) in \cite{Pacard})
$$
\int_{{\mathbb S}^{2}} (\Ric(\Theta, \Theta))^{2} dV_{{\mathbb S}^{2}}(\Theta)=\frac{4\pi}{15} \left( 2 \|\Ric \|^{2} + \Sc(p)^{2} \right)
$$
and that (see for instance \cite[Page 780]{Mondino})
$$
\int_{{\mathbb S}^{2}} \left( \Ric(\Theta_1,\Theta_1)\Ric(\Bar{\Theta}_2,\Bar{\Theta}_2) - \Ric(\Theta_1,\Bar{\Theta}_2)^2 \right)\, dV_{{\mathbb S}^{2}}(\Theta)= -\frac{2\pi}{3} \left( \|\Ric\|^{2}- \Sc(p)^{2}\right).
$$
We next compute the term $\int_{{\mathbb S}^2} \Big(\Ric(\Theta,\Theta_1)^2 + \Ric(\Theta,\Bar{\Theta}_2)^2\Big) dV_{{\mathbb S}^{2}}$. Taking into account that $R_{ij} = R_{ji}$ and \eqref{eq:Theta1Theta2}, we can write:
\begin{align*}
    \Ric(\Theta,\Theta_1)^2 &= \left(R_{11}\frac{x^2z}{\sqrt{x^2+y^2}} + R_{12}\frac{2xyz}{\sqrt{x^2+y^2}} + R_{13}\frac{z^2x - x(x^2 + y^2)}{\sqrt{x^2+y^2}} + R_{22}\frac{y^2z}{\sqrt{x^2+y^2}} \right. \\&\quad \left.+ R_{23}\frac{z^2y - y(x^2 + y^2)}{\sqrt{x^2+y^2}} - R_{33}z\sqrt{x^2+y^2}\right)^2 \nonumber\\
    \Ric(\Theta,\Bar{\Theta}_2)^2 &= \left(-R_{11}\frac{xy}{\sqrt{x^2+y^2}} + R_{12}\frac{x^2 - y^2}{\sqrt{x^2+y^2}} + R_{22}\frac{xy}{\sqrt{x^2+y^2}}- R_{31}\frac{yz}{\sqrt{x^2+y^2}} + R_{32}\frac{xz}{\sqrt{x^2+y^2}}\right)^2.
\end{align*}
Observe that we can ignore any polynomial term in the integrand which has an odd power, since it will integrate to zero by anti-symmetry. Inspection shows that after the brackets are expanded, the only terms that will consist entirely of even powers of $x$, $y$ and $z$ are those containing the following coefficients of the Ricci tensor:
\begin{align*}
    R_{11}^2, R_{11}R_{22}, R_{11}R_{33}, R_{12}^2, R_{13}^2, R_{22}^2, R_{22}R_{33}, R_{23}^2, R_{33}^2.
\end{align*}
We compute these terms as follows, using that $x^2 + y^2 + z^2 = 1$:
\begin{align*}
    \left(R_{11}\frac{x^2z}{\sqrt{x^2+y^2}}\right)^2 + \left(-R_{11}\frac{xy}{\sqrt{x^2+y^2}}\right)^2 &= R_{11}^2\left(\frac{x^4z^2 + x^2y^2}{x^2+y^2}\right) = R_{11}^2\left(\frac{x^2(x^2(1-x^2-y^2) + y^2)}{x^2+y^2}\right) \\&= R_{11}^2\left(\frac{x^2(x^2 + y^2)(1 - x^2)}{x^2+y^2}\right) = R_{11}^2(x^2 - x^4),
\end{align*}
\begin{align*}
    &2\left(R_{11}\frac{x^2z}{\sqrt{x^2+y^2}}\right)\left(R_{22}\frac{y^2z}{\sqrt{x^2+y^2}}\right) - 2\left(R_{11}\frac{xy}{\sqrt{x^2+y^2}}\right)\left(R_{22}\frac{xy}{\sqrt{x^2+y^2}}\right) \\
    & \qquad = 2R_{11}R_{22}\left(\frac{x^2y^2(z^2 - 1)}{x^2+y^2}\right)= -2R_{11}R_{22}x^2y^2,
\end{align*}
\begin{align*}
    -2\left(R_{11}\frac{x^2z}{\sqrt{x^2+y^2}}\right)\left(R_{33}z\sqrt{x^2+y^2}\right) = -2R_{11}R_{33}x^2z^2,
\end{align*}
\begin{align*}
    \left(R_{12}\frac{2xyz}{\sqrt{x^2+y^2}}\right)^2 + \left(R_{12}\frac{x^2 - y^2}{\sqrt{x^2+y^2}}\right)^2 &= R_{12}^2\left(\frac{4x^2y^2z^2 + x^4 + y^4 - 2x^2y^2}{x^2+y^2}\right) \\
    &= R_{12}^2\left(\frac{4x^2y^2(1 - x^2 - y^2) + x^4 + y^4 - 2x^2y^2}{x^2+y^2}\right) \\
    &  = R_{12}^2\left(\frac{x^4 + y^4 + 2x^2y^2 - 4x^4y^2 - 4x^2y^4}{x^2+y^2}\right) \\
    &= R_{12}^2\left(\frac{(x^2 + y^2)^2 - 4x^2y^2(x^2 + y^2)}{x^2+y^2}\right) \\&= R_{12}^2(x^2 + y^2 - 4x^2y^2) \\&= R_{12}^2(1 - z^2 - 4x^2y^2),
\end{align*}
\begin{align*}
    \left(R_{13}\frac{z^2x - x(x^2 + y^2)}{\sqrt{x^2+y^2}}\right)^2 + \left(-R_{31}\frac{yz}{\sqrt{x^2+y^2}}\right)^2 &= R_{13}^2\left(\frac{(z^2x - x(1 - z^2))^2 + y^2z^2}{1 - z^2}\right) \\&= R_{13}^2\left(\frac{4z^4x^2 - 4z^2x^2 + x^2 + y^2z^2}{1 - z^2}\right) \\&= R_{13}^2\left(\frac{4z^4x^2 - 4z^2x^2 + 1 - y^2 - z^2 + y^2z^2}{1 - z^2}\right) \\&= R_{13}^2\left(\frac{4z^2x^2(z^2 - 1) + (1 - z^2)(1 - y^2)}{1 - z^2}\right) \\&= R_{13}^2(1 - y^2 - 4z^2x^2),
\end{align*}
\begin{align*}
    -2\left(R_{22}\frac{y^2z}{\sqrt{x^2+y^2}}\right)\left(R_{33}z\sqrt{x^2+y^2}\right) = -2R_{22}R_{33}y^2z^2,
\end{align*}
\begin{align*}
    \left(R_{22}\frac{y^2z}{\sqrt{x^2+y^2}}\right)^2 + \left(R_{22}\frac{xy}{\sqrt{x^2+y^2}}\right)^2 &= R_{22}^2\left(\frac{y^4z^2 + x^2y^2}{x^2+y^2}\right) \\&= R_{22}^2\left(\frac{y^2(y^2(1-x^2-y^2) + x^2)}{x^2+y^2}\right) \\&= R_{22}^2\left(\frac{y^2(x^2 + y^2)(1 - y^2)}{x^2+y^2}\right) = R_{22}^2(y^2 - y^4),
\end{align*}
\begin{align*}
    \left(R_{23}\frac{z^2y - y(x^2 + y^2)}{\sqrt{x^2+y^2}}\right)^2 + \left(R_{32}\frac{xz}{\sqrt{x^2+y^2}}\right)^2 &= R_{23}^2\left(\frac{(z^2y - y(1 - z^2))^2 + x^2z^2}{1 - z^2}\right) \\&= R_{23}^2\left(\frac{4z^4y^2 - 4z^2y^2 + y^2 + x^2z^2}{1 - z^2}\right) \\&= R_{23}^2\left(\frac{4z^4y^2 - 4z^2y^2 + 1 - x^2 - z^2 + x^2z^2}{1 - z^2}\right) \\&= R_{23}^2\left(\frac{4z^2y^2(z^2 - 1) + (1 - z^2)(1 - x^2)}{1 - z^2}\right) \\&= R_{23}^2(1 - x^2 - 4z^2y^2),
\end{align*}
\begin{align*}
    \left(-R_{33}z\sqrt{x^2+y^2}\right)^2 = R_{33}^2z^2(x^2 + y^2) = R_{33}^2z^2(1 - z^2) = R_{33}^2(z^2 - z^4).
\end{align*}
Thus, using \eqref{eq:intCoord}, we get:
\begin{align}\label{eq:76}
    \int_{{\mathbb S}^2}\Ric(\Theta,\Theta_1)^2 + \Ric(\Theta,\Bar{\Theta}_2)^2 dV_{{\mathbb S}^2} &= \int_{{\mathbb S}^2}R_{11}^2(x^2 - x^4) - 2R_{11}R_{22}x^2y^2 - 2R_{11}R_{33}x^2z^2 + R_{12}^2(1 - z^2 - 4x^2y^2) \nonumber\\
    &  \quad+ R_{13}^2(1 - y^2 - 4z^2x^2) + R_{22}^2(y^2 - y^4) - 2R_{22}R_{33}y^2z^2 \nonumber\\&\quad+ R_{23}^2(1 - x^2 - 4z^2y^2) + R_{33}^2(z^2 - z^4) dV_{{\mathbb S}^2} \nonumber\\
%    &= R_{11}^2(\frac{4\pi}{3} - \frac{4\pi}{5}) - R_{11}R_{22}\frac{8\pi}{15} - R_{11}R_{33}\frac{8\pi}{15} + R_{12}^2(4\pi - \frac{4\pi}{3} - \frac{16\pi}{15}) \nonumber\\
%    &\quad+ R_{13}^2(4\pi - \frac{4\pi}{3} - \frac{16\pi}{15}) + R_{22}^2(\frac{4\pi}{3} - \frac{4\pi}{5}) - R_{22}R_{33}\frac{8\pi}{15} \nonumber\\
%    &\quad+ R_{23}^2(4\pi - \frac{4\pi}{3} - \frac{16\pi}{15}) + R_{33}^2(\frac{4\pi}{3} - \frac{4\pi}{5}) \nonumber\\
%    &= \frac{8\pi}{15}R_{11}^2 - \frac{8\pi}{15}R_{11}R_{22} - \frac{8\pi}{15}R_{11}R_{33} + \frac{8\pi}{5}R_{12}^2 \nonumber\\
%    &\quad+ \frac{8\pi}{5}R_{13}^2 + \frac{8\pi}{15}R_{22}^2 - \frac{8\pi}{15}R_{22}R_{33} \nonumber\\&\quad+ \frac{8\pi}{5}R_{23}^2 + \frac{8\pi}{15}R_{33}^2 \nonumber\\
%    &= \frac{8\pi}{15}(R_{11}^2 + R_{22}^2 + R_{33}^2 - R_{11}R_{22} - R_{11}R_{33} - R_{22}R_{33}) + \frac{8\pi}{5}(R_{12}^2 + R_{13}^2 + R_{23}^2) \nonumber\\
    &= \frac{8\pi}{15}(R_{11}^2 + R_{22}^2 + R_{33}^2 + 3R_{12}^2 + 3R_{13}^2 + 3R_{23}^2) \nonumber\\&\quad- \frac{8\pi}{15}(R_{11}R_{22} + R_{11}R_{33} + R_{22}R_{33}) \nonumber\\
    &= \frac{4\pi}{5}(R_{11}^2 + R_{22}^2 + R_{33}^2 + 2R_{12}^2 + 2R_{13}^2 + 2R_{23}^2) \nonumber\\
    &\quad- \frac{4\pi}{15}(R_{11}^2 + R_{22}^2 + R_{33}^2 + 2R_{11}R_{22} + 2R_{11}R_{33} + 2R_{22}R_{33}) \nonumber\\
    &= \frac{4\pi}{5}\Big[\norm{\Ric}^2 - \frac{1}{3}\Sc(p)^2\Big] \nonumber\\
    &= \frac{4\pi}{5}\norm{S_{p}}^2,
\end{align}
where $ S := \Ric - \frac{1}{3} \Sc\,  g$ is the traceless Ricci tensor.
\\
To rewrite the final $\rho^4$ term, first note that $\nabla_{\Theta}\Theta_i=0$ because $\Theta_i$ does not depend on the radial coordinate and the Christoffel symbols vanish at $p$ since we chose normal coordinates. Thus, by compatibility of the metric:
\begin{align*}
    \frac{1}{2}g_{{\mathbb S}^2}^{ij}g(\nabla_{\Theta \Theta}^2{\mathcal R}(\Theta,\Theta_i)\Theta,\Theta_j) &= \frac{1}{2}g_{{\mathbb S}^2}^{ij}\nabla_{\Theta \Theta}^2g({\mathcal R}(\Theta,\Theta_i)\Theta,\Theta_j) +\cO_p(\rho)= \frac{1}{2}\nabla_{\Theta \Theta}^2\big(g_{{\mathbb S}^2}^{ij}g({\mathcal R}(\Theta,\Theta_i)\Theta,\Theta_j)\big)+ \cO_p(\rho)\\
    &= - \frac{1}{2}\nabla_{\Theta \Theta}^2\Ric(\Theta,\Theta)+\cO_p(\rho).
\end{align*}
 Now we can integrate, using index notation and \eqref{eq:intCoord}, as follows:
\begin{align}
    \int_{{\mathbb S}^2}-\frac{1}{2}\nabla_{\Theta \Theta}^2\Ric(\Theta,\Theta)\ dV_{{\mathbb S}^{2}}&= -\frac{1}{2}\int_{{\mathbb S}^2}\nabla_{\mu}\nabla_{\nu}R_{\sigma\tau}x^{\mu}x^{\nu}x^{\sigma}x^{\tau}\ dV_{{\mathbb S}^{2}} \nonumber \\ 
    &= -\nabla_{\mu}\nabla_{\nu}R_{\sigma\tau}\frac{1}{2}\int_{{\mathbb S}^2}x^{\mu}x^{\nu}x^{\sigma}x^{\tau}\ dV_{{\mathbb S}^{2}} \nonumber \\ 
    &= -\frac{4\pi}{30}\Big(3\sum_{\mu}\nabla_{\mu}\nabla_{\mu}R_{\mu\mu} + \sum_{\mu\neq\nu}\nabla_{\mu}\nabla_{\mu}R_{\nu\nu} + \sum_{\mu\neq\nu}\nabla_{\mu}\nabla_{\nu}R_{\mu\nu} + \sum_{\mu\neq\nu}\nabla_{\mu}\nabla_{\nu}R_{\nu\mu}\Big) \nonumber \\ 
       &= -\frac{2\pi}{15} \sum_{\mu\nu}\Big(\nabla_{\mu}\nabla_{\mu}R_{\nu\nu} + 2\nabla_{\mu}\nabla_{\nu}R_{\mu\nu}\Big) \nonumber\\ 
    &= -\frac{4\pi}{15}\Delta \Sc(p), \label{eq:DeltaSc}
\end{align}
where in the last line we have used the contracted Bianchi identity.
Thus, integrating integrating (\ref{eq:75}) we obtain:
\begin{align}\label{eq:77}
    W(S_{p,\rho}(w)) &= \int_{{\mathbb S}^2}H^2\sqrt{\det\mathring{g}}\ dV_{{\mathbb S}^2} \nonumber\\&= 16\pi - \frac{8\pi}{3}\Sc(p)\rho^2 + \Bigg[-\frac{34\pi}{135}(2\norm{\Ric}^2 + \Sc(p)^2) -\frac{26\pi}{27}\Sc(p)^2 + \frac{28\pi}{27}\Sc(p)^2 \nonumber\\
    &\quad  \quad+ \frac{8\pi}{9}\norm{S}^2 - \frac{20\pi}{27}(\norm{\Ric}^2 - \Sc(p)^2) - \frac{4\pi}{15}\Delta \Sc(p)\Bigg]\rho^4 + \cO_{p}(\rho^5)\nonumber \\
%    &= 16\pi - \frac{8\pi}{3}\Sc(p)\rho^2 + \Bigg[-\frac{56\pi}{45}\norm{\Ric}^2 + \frac{76\pi}{135}\Sc(p)^2 + \frac{8\pi}{9}\norm{S}^2 - \frac{4\pi}{15}\Delta \Sc(p)\Bigg]\rho^4 + \cO_{p}(\rho^5) \nonumber\\
%    &= 16\pi - \frac{8\pi}{3}\Sc(p)\rho^2 + \Bigg[-\frac{56\pi}{45}\big(\norm{\Ric}^2 - \frac{1}{3}\Sc(p)^2\big) + \frac{4\pi}{27}\Sc(p)^2 + \frac{8\pi}{9}\norm{S}^2 - \frac{4\pi}{15}\Delta \Sc(p)\Bigg]\rho^4 + \cO_{p}(\rho^5) \nonumber\\
    &   = 16\pi - \frac{8\pi}{3}\Sc(p)\rho^2 + \Bigg[\frac{4\pi}{27}\Sc(p)^2 - \frac{16\pi}{45}\norm{S_{p}}^2 - \frac{4\pi}{15}\Delta \Sc(p)\Bigg]\rho^4 + \cO_{p}(\rho^5).
\end{align}
 
  \textbf{Step 4}: conclusion. 
\\ In order to conclude, we have to compute the expansion for  $|S_{p,\rho}(w)|_g$. To this aim,  inserting \eqref{eq:estw} into \eqref{eq:61} yields the following expansion for the area of optimally perturbed spheres:
\begin{equation}\label{eq:AreaSprhow}
    |S_{p,\rho}(w)|_g = |{\mathbb S}^2|_{\mathring{\delta}}\rho^2\left[ 1 - \frac{1}{18}\Sc_p\rho^2 + \cO(\rho^4)\right].
\end{equation}
Substituting \eqref{eq:77} and \eqref{eq:AreaSprhow} into \eqref{eq:defHaw},  we thus obtain:
\begin{align*}
 m_H(S_{p,\rho}(w_{p,\rho}))&=   \frac{\sqrt{|{\mathbb S}^2|_{\mathring{\delta}}}\rho}{\sqrt{(16\pi)^3}}\left[ 1 - \frac{1}{36}\Sc_p\rho^2 + \cO(\rho^4)\right] \left[ \frac{8\pi}{3}\Sc_p\rho^{2} + \left(\frac{4\pi}{15}\Delta \Sc(p) +  \frac{16\pi}{45}\norm{S_p}^2- \frac{4\pi}{27}\Sc_p^2\right)\rho^4 + \cO(\rho^5) \right]\\
  %&= \frac{1}{32 \pi}  \left[ \frac{8\pi}{3}\Sc_p\rho^{3} + \left(\frac{4\pi}{15}\Delta \Sc(p) + \frac{16\pi}{45}\norm{S_p}^2 - \frac{2\pi}{9}\Sc_p^2\right) \rho^5 + \cO(\rho^6) \right]\\
  &= \frac{1}{12} \Sc_p\rho^{3} +  \left(\frac{1}{120} \Delta \Sc(p) + \frac{1}{90}\norm{S_p}^2 - \frac{1}{144}\Sc_p^2  \right)  \rho^5 + \cO(\rho^6).
\end{align*}

\end{proof}

\section{Rigidity results in case of non-negative scalar curvature}

\begin{theorem}\label{prop:Riem0}
Let $(M^3,g)$ be a  three-dimensional Riemannian manifold  and let $\Omega\subset M$ be an open subset with non-negative scalar curvature.
Assume that  every $p\in \Omega\setminus \partial M$ admits a neighbourhood $U\subset M\setminus \partial M$ such that
\begin{equation}\label{eq:mHUleq0}
\sup\{m_{H}(\Sigma) \,:\, \Sigma \subset U \emph{ is an immersed 2-dimensional surface}\} \leq 0,
\end{equation}
or, more generally, that
\begin{equation}\label{eq:HpRiem0}
\limsup_{\rho\downarrow 0} \rho^{-5} m_{H}(S_{p,\rho}(w_{p,\rho}))\leq 0, \quad \forall p\in \Omega \setminus \partial M, 
\end{equation}
where $S_{p,\rho}(w_{p,\rho})$ is the optimally perturbed geodesic sphere with $w_{p,\rho}$ given in Lemma \ref{lem:estw}.

Then $\Omega \setminus \partial M$ is locally isometric to Euclidean $\R^3$.
\end{theorem}

\begin{proof}
%Under the assumptions of Theorem  \ref{thm:1}, for every compact subset ${\mathcal K}\subset M$ there exists  $\rho_{0}>0$ such that 
%\begin{equation}\label{eq:mHleq0}
%m_H(S_{p,\rho}(w_{p,\rho})) \leq 0, \quad \forall \rho\in (0,\rho_{0}],\;\forall  p\in \Omega\cap \mathcal K.
%\end{equation}
The combination of the assumption \eqref{eq:HpRiem0} with the expansion \eqref{eq:mHSprhow} yields
\begin{equation}\label{eq:exprho2rho4}
 \limsup_{\rho\downarrow 0} \rho^{-5} \left[  \frac{1}{12} \Sc_p\rho^{3} +  \left(\frac{1}{120} \Delta \Sc(p) + \frac{1}{90}\norm{S_p}^2 - \frac{1}{144}\Sc_p^2  \right)  \rho^5 + \cO(\rho^6)   \right] \leq 0, \quad \forall p\in \Omega \setminus \partial M.
\end{equation}
Since $\Sc$ is assumed to be non-negative, letting $\rho \downarrow 0$ and looking at the order $\rho^{-2}$ in the expansion \eqref{eq:exprho2rho4},   we first infer that 
\begin{equation}\label{eq:Sc=0}
\Sc\equiv 0 \quad \text{ on } \Omega \setminus \partial M.
\end{equation}
Plugging \eqref{eq:Sc=0} into \eqref{eq:exprho2rho4} and looking now  at the order $\rho^{0}$ in the expansion, we deduce that 
\begin{equation}\label{eq:S=0}
S \equiv 0 \quad \text{ on } \Omega \setminus \partial M.
\end{equation}
Putting together \eqref{eq:Sc=0} and \eqref{eq:S=0} gives
\begin{equation}\label{eq:Ric=0}
\Ric \equiv 0 \quad \text{ on } \Omega \setminus \partial M.
\end{equation}
Recalling that in dimension three the Riemann curvature tensor can be written as $\Rm = \Ric\KN g - \frac{1}{4} \Sc \, g\KN g $ where $\KN$ is the Kulkarni-Nomizu product (see for instance \cite[Corollary 7.26]{LeeRM}) we conclude that $\Rm\equiv 0$ on $\Omega \setminus \partial M$, as desired.
\end{proof}

Recall the Definition \ref{def:2} of asymptotically locally simply connected (ALSC for short). The global rigidity result (Theorem \ref{thm:1}, below) will follow by the combination of Theorem \ref{prop:Riem0} with the following proposition of independent interest.

\begin{prop}\label{prop:ALSCR3}
Let $(M^3,g)$ be a connected, complete, ALSC Riemannian manifold without boundary and with zero sectional curvature. Then $(M^3,g)$ is globally isometric to Euclidean $\R^3$.
\end{prop}

\begin{proof}
By the classical Killing-Hopf Theorem, we have that $(M^3,g)$ is isometric to a space form $\R^3/ \Gamma$, for some discrete sub-group  $\Gamma$ of the affine transformations ${\mathbb E}(3)$ of $\R^{3}$ acting freely  (see for instance \cite[Theorem 3.3.2]{Wolf}). Denote with $\pi: \R^{3}\to \R^{3}/\Gamma\simeq M$ the covering map.
 Since  ${\mathbb E}(3)$ is isomorphic to the semi-direct product ${\rm O}(3)\ltimes \R^3$, we can write an arbitrary element $\gamma\in \Gamma$ as $\gamma = (r,a)$ where $r\in {\rm O}(3)$ and $a\in \R^3$.
\\Denote with  $n\in \{0,1,2,3\}$ the rank of the maximal abelian subgroup of $\Gamma$.   After choosing an appropriate orthonormal basis of $\R^{3}$ adapted to $\Gamma$, we can write  $\R^3 = \R^n\times\R^{3-n}$ in  such a way that  each $\gamma \in \Gamma$ can be expressed uniquely as 
$$
\gamma= \delta \times \psi,\quad \text{ where } \delta = (r|_{\R^n},a), \, a\in \R^{n}, \quad\text{and }  \psi = (r|_{\R^{3-n}},0).
$$
We next discuss separately the different cases for the value of $n$.

$\bullet$ $n=0$.  In this case,  $\Gamma\subset {\rm O}(3)$.  Since $\Gamma$ acts freely, it follows that $\Gamma=\{ {\rm Id}\}$ must be trivial and thus $(M, g)$ is isometric to $\R^3$.

$\bullet$ $n=1$. We claim that for any $\gamma\in \Gamma$, the corresponding translation $a\in \R$ must vanish.
\\Assume by contradiction this is not the case and fix such a $\gamma\in \Gamma$ with corresponding non-trivial translation $0\neq a\in \R$.
Fix an arbitrary point  $(u,v)\in \R\times\R^2$ and consider the curve $\alpha(s) = (u + sa,v)\subset \R^3$ for $s\in [0,1]$.
It is easily seen that, for any   radius $R>|a|$,  $\pi \circ \alpha$ is a loop contained in $B_R^M(\pi(u,v))\subset M$ and that $\pi \circ \alpha$ is not contractible in $B_R^M(\pi(u,v))$. Indeed, if  $\pi \circ \alpha$ were homotopic to the constant loop $c(\cdot) \equiv \pi(u,v)$ then, by the uniqueness of path liftings in the covering space (see for instance \cite[Lemma 54.1, Chapter 9]{Munkres}), the lifts of $c(\cdot)$ and $\pi\circ \alpha$ starting at $(u,v)$ would end at the same point. But this is not true as the lift of $c(\cdot)$ is the constant loop in $\R\times\R^2$ taking value $(u,v)$, whereas the lift of $\pi\circ \alpha$ is $\alpha$, which ends at $(u+a,v)\neq (u,v)$. Therefore $B_R^M(\pi(u,v))\subset M$ is not simply connected for every $R>a$ and every  $(u,v)\in \R\times\R^2$, contradicting the ALSC assumption on $M$. This concludes the proof that for every  $\gamma\in \Gamma$ the corresponding translation $a\in \R$ must vanish. It follows that $\Gamma\subset {\rm O}(3)$ and thus, as in the case $n=0$, we infer that $\Gamma=\{ {\rm Id}\}$ must be trivial and  $(M, g)$ is isometric to $\R^3$.

$\bullet$  $n=2$.  We claim that for any $\gamma\in \Gamma$, the corresponding translation $a\in \R^{2}$ must vanish.
\\Assume by contradiction this is not the case and fix such a $\gamma\in \Gamma$ with corresponding non-trivial translation $0\neq a\in \R^{2}$.
Fix an arbitrary point  $(u,v)\in \R^{2}\times\R$ and consider the curve $\alpha(s) = (u + sa,v)\subset \R^3$ for $s\in [0,1]$. The analogous argument as in the case $n=1$ shows that, for any   radius $R>|a|$,  $\pi \circ \alpha$ is a loop contained in $B_R^M(\pi(u,v))\subset M$ and that $\pi \circ \alpha$ is not contractible in $B_R^M(\pi(u,v))$. This contradicts the ALSC assumption and thus, as in the $n=1$ case, we conclude that $\Gamma=\{ {\rm Id}\}$ must be trivial and $(M, g)$ is isometric to $\R^3$.

$\bullet$  $n=3$.   In this case, the quotient $M = \R^3/ \Gamma$ is compact, contradicting the ALSC assumption. 
\\
Therefore, the only possibility is $\Gamma = \{ {\rm Id}\}$ and  $(M, g)$ isometric to $\R^3$.

\end{proof}

\begin{theorem}\label{thm:1}
Let $(M^3,g)$ be a connected, complete, ALSC Riemannian manifold without boundary and with non-negative scalar curvature.
Assume that every $p\in M$ admits a neighbourhood satisfying \eqref{eq:mHUleq0} or, more generally, that the infinitesimal non-positivity of $m_{H}$ \eqref{eq:HpRiem0} holds for every $p\in M$.

Then $(M^3,g)$ is globally isometric to Euclidean $\R^3$.
\end{theorem}

 \begin{proof}
 The result follows by the combination of Theorem \ref{prop:Riem0} and Proposition \ref{prop:ALSCR3}.  
 \end{proof}

\section{A lower bound on the Bartnik mass}

In this section we are interested in asymptotically flat Riemannian manifolds  as defined  in the introduction, with non-negative scalar curvature. 
In order to give a lower bound on the Bartnik mass, in the next subsection we show that  perturbed geodesic spheres $S_{p,\rho}(w)$ are outer-minimising. We will then combine the monotonicity of the Bartnik mass \eqref{eq:monmB}, with the fact that Bartnik mass is bounded from below by the Hawking mass \eqref{mBgeqmH} and the geometric expansions of the Hawking mass on perturbed geodesic spheres obtained in the previous sections to get the result. 

\subsection{Some auxiliary lemmas about sets of finite perimeter}

The goal of this section is to establish some auxiliary lemmas that will be useful in the proof of the main result in the next section, that is Theorem \ref{Prop:OuterMin}.
\\For this section we consider a general Riemannian manifold $(M,g)$ of dimension three. Given a Borel subset $E\subset M$ and an open subset $U\subset M$, we denote by
\begin{equation}\label{eq:defPer}
\rP_{g}(E,U):=\sup\left\{ \int_{E} {\rm div}_{g}(X) \, dV_{g} \ : X\text{ is a $C^{1}$ vector field with } {\rm supp}\, X \subset U\setminus \partial M, \, \| X \|_{\infty, g} \leq 1 \right\}
\end{equation} 
the perimeter of $E$ relative to $U$. When $U=M$, we simply write $\rP_{g}(E):=\rP_{g}(E, M)$.  Note that $\partial M$ never contributes to the perimeter.

Given a point $p\in M \setminus \partial M $, let ${\rm Inj}_{p}>0$ be the injectivity radius at $p$ and  denote with $\phi^{p}_{g}:B^{g}_{{\rm Inj}_{p}}(p)\to \R^{3}$ a normal coordinate chart centred at $p$ with respect to the Riemannian metric $g$. It will be convenient to consider the rescaled metrics $g_{\rho}:=\rho^{-2} g$, for $\rho\in (0,1]$. It is a standard fact that, in $g_{\rho}$-normal coordinates centred at $p$, it holds: 
\begin{equation}\label{eq:29}
\|(g_{\rho})_{\mu \nu}-\delta_{\mu \nu} \|_{C^{k}\big(B^{g_{\rho}}_{r \rho^{-1}} (p) \big)} \leq C \, \rho^{2} r^{2}, \quad \forall \rho\in (0, 1], \; r\in (0, {\rm Inj}_{p}], 
\end{equation}
for some suitable $C=C(p,k)>0$.
\\ From the very definition of perimeter \eqref{eq:defPer}, it is easily seen that 
\begin{equation}\label{eq;rescalPer}
\rP_{g_{\rho}}(E,U) = \rho^{-2}\; \rP_g(E,U), \quad \text{for every Borel subset  $E \subset M$  and $U \subset M$ open}.
\end{equation}

We next establish some lemmas linking  the relative perimeter $\rP_{g_{\rho}}(E,B_r^{g_{\rho}}(q))$ of a finite perimeter set $E\subset M$, and the Euclidean relative perimeter  $\rP_{\delta}(\phi_{g_{\rho}}(E),B_r^{\delta}(\phi_{g_{\rho}}(q)))$ of  its image $\phi_{g_{\rho}}(E) \subset \R^3$. We remark that, in such relative perimeter calculations, we can ignore the fact that $E$ may not be contained in the domain of $\phi_{g_{\rho}}$ because it is enough that $B_r^{g_{\rho}}(q)$ is contained there (which, for small $\rho$, it will be).
To this aim,  we will distinguish between the normal coordinate charts centred at either at a point $p$ or $q$ nearby, denoting them by $\phi^p_{g_{\rho}}$ and $\phi^q_{g_{\rho}}$ respectively. 
 
 \begin{lemma}\label{Lem:14}
Fix a point $p\in M \setminus \partial M$.  Then, there exist constants $C=C(p)>0$, $\bar{\rho}=\bar{\rho}(p)>0$ such that for every finite perimeter set $E\subset M$, every $\rho\in (0, \bar{\rho}]$ and every $q\in B^{g_{\rho}}_{10}(p)$ it holds:
	\begin{equation}\label{eq:30}
\left| \rP_{\delta}(\phi^p_{g_{\rho}}(E),B^{\delta}_r(\phi^p_{g_{\rho}}(q))) - \rP_{g_{\rho}}(E,B_r^{g_{\rho}}(q)) \right| \leq  C \, r^{2} \rP_{g_{\rho}}(E,B_r^{g_{\rho}}(q)) + C r^4, \quad \forall r\in (0,1].
	\end{equation}
\end{lemma}

\begin{proof}
Let $q\in B^{g_{\rho}}_{10}(p)$ and consider $g_{\rho}$-normal coordinates centred at  $q$. Notice that, by smooth dependence of initial data in the geodesic equation, the constant appearing in  \eqref{eq:29} (when applied to normal coordinates centred at $q\in B^{g_{\rho}}_{10}(p)$) may be chosen independently on $q$, so it still only depends on $p$.

For a $C^{1}$-vector field $X$ supported in $B_r^{g_{\rho}}(q)$ with $\|X\|_{\infty,\delta}\leq 1$, we thus  have:
\begin{align*}
	\int_{E\cap B_r^{g_{\rho}}(q)} {\rm div}_{g_{\rho}}(X)\ dV_{g_{\rho}} = \int_{\phi^q_{g_{\rho}}(E\cap B_r^{g_{\rho}}(q))}  \left[ (1+ \cO_{p}(r^{2})) {\rm div}_{\delta}(X) + \cO_{p}(r) \right] \ dV_{\delta}.
\end{align*}
We infer that
\begin{align*}
&	 \Big|\int_{E\cap B_r^{g_{\rho}}(q)} {\rm div}_{g_{\rho}}(X)\ dV_{g_{\rho}} - \int_{\phi^q_{g_{\rho}}(E\cap B_r^{g_{\rho}}(q))} {\rm div}_{\delta}(X)\ dV_{\delta}\Big| \leq Cr^2 \int_{\phi^q_{g_{\rho}}(E\cap B_r^{g_{\rho}}(q))}|{\rm div}_{\delta}(X)|\ dV_{\delta} + Cr V_{\delta}\big( B^{\delta}_r(\phi^q_{g_{\rho}}(q))\big) \\
	 & \qquad \leq Cr^2\int_{\phi^q_{g_{\rho}}(E\cap B_r^{g_{\rho}}(q))\cap \{{\rm div}_{\delta}(X)>0\}} {\rm div}_{\delta}(X)\ dV_{\delta} + Cr^2\int_{\phi^q_{g_{\rho}}(E\cap B_r^{g_{\rho}}(q))\cap \{{\rm div}_{\delta}(X)<0\}}- {\rm div}_{\delta}(X)\ dV_{\delta} +C r^{4} \\
	&\qquad \leq Cr^2\rP_{\delta}(\phi^q_{g_{\rho}}(E),\phi^q_{g_{\rho}}(B_r^{g_{\rho}}(q))) + Cr^4 ,
\end{align*}
where in the first inequality we used that $V_{\delta}\big(\phi^q_{g_{\rho}}(E)\cap B_r^{\delta}(\phi^q_{g_{\rho}}(q))\big) \leq V_{\delta}\big(B_r^{\delta}(\phi^q_{g_{\rho}}(q)) \big) \leq Cr^3$, and in the third inequality we have used that $X$ is admissible in the definition of perimeter if and only if $-X$ is too.	
\\Taking the supremum over all $C^{1}$-vector fields $X$ supported in $B_r^{g_{\rho}}(q)$ with $\|X\|_{\infty,\delta}\leq 1$ gives the claim.
\end{proof}

\begin{lemma}\label{Lem:29}
Fix a point $p\in M \setminus \partial M$.  Then, there exist constants $C=C(p)>0$, $\bar{\rho}=\bar{\rho}(p)>0$ such that for every finite perimeter set $E\subset M$, every $\rho\in (0, \bar{\rho}]$ and every $q\in B^{g_{\rho}}_{10}(p)$ it holds:
\begin{equation}\label{eq:151}
	\left| \rP_{\delta}\big(\phi^p_{g_{\rho}}(E), B_r^{\delta}(\phi^p_{g_{\rho}}(q))\big) - \rP_{\delta}\big(\phi^q_{g_{\rho}}(E), B_r^{\delta}(\phi^q_{g_{\rho}}(q))\big) \right|\leq  C \,\rP_{g_{\rho}} (E, B^{g_{\rho}}_{10}(p) )  \rho^2, \quad \forall r\in (0,1].
\end{equation}
\end{lemma}

\begin{proof}
From \eqref{eq:29} and the smooth dependence on coefficients in the geodesic equation, it is standard to check that for any point $q\in B^{g_{\rho}}_{10}(p)$  the $g_{\rho}$-exponential map satisfies:

\begin{equation*}
	\norm{\exp^{g_{\rho}}_{q} - \exp^{\delta}_{\phi^p_{g_{\rho}}(q)}}_{C^1\big(B^{g_{\rho}}_{10}(p)\big)} \leq C(p) \, \rho^{2}, 
\end{equation*}
where we consider both maps (for $\exp^{g_{\rho}}_{q}$, via the pullback metric) as diffeomorphisms on a ball in $\R^3$. Note that $\exp^{\delta}_{\phi^p_{g_{\rho}}(q)} = \mathcal{T}_{\phi^p_{g_{\rho}}(q)}$, where $\mathcal{T}_{\phi^p_{g_{\rho}}(q)}$ is the translation by the vector $\phi^p_{g_{\rho}}(q)$. Since the normal coordinate chart is the inverse of the exponential map, this means:
\begin{equation*}
	\norm{\phi^{q}_{g_{\rho}} - \mathcal{T}_{-\phi^p_{g_{\rho}}(q)}}_{C^1\big(B^{g_{\rho}}_{10}(p)\big)}\leq C(p) \, \rho^{2}.
\end{equation*}
It follows that 
$$
	\norm{\phi^p_{g_{\rho}} - \mathcal{T}_{\phi^p_{g_{\rho}}(q)} \circ \phi^q_{g_{\rho}}}_{C^1\big(B^{g_{\rho}}_{10}(p)\big)} \leq C(p) \, \rho^{2}.
	%\label{eq:149}
$$
Since the translation is an isometry of $\R^{3}$, it is not hard to see that the last estimate implies the claim.
\end{proof}

We next recall a well known consequence of the monotonicity formula for a finite perimeter set which is stationary for the perimeter functional; actually this holds more generally for stationary rectifiable varifolds, see for instance \cite[Chap. 4.3]{Simon}.

\begin{lemma}\label{Lem:7}
	Let $(M,g)$ be a Riemannian manifold, $U\subset M\setminus \partial M$ be a relatively compact open  subset, and $E\subset M$ be a set of finite perimeter which is stationary for $\rP_{g}(\cdot, U)$,   the perimeter functional relative to $U$ (i.e. zero first variation restricted to $U$). Then there exist  $C=C(U, \rP_{g}(E,U))>0, r_0=r_{0}(U, \rP_{g}(E,U)) > 0$ such that
		\begin{equation}\label{eq:6}
		\rP_{g}(E,B^{g}_r(q)) \leq Cr^2, \quad \text{for all $B^{g}_{r}(q)\subset U$ and all } r\in (0,r_{0}].
	\end{equation}
\end{lemma}

The following lemma can be proved along the same lines as \cite[(6--9)]{MondinoSpadaro}:

\begin{lemma}\label{Lem:30}
Let $(M,g)$ be a Riemannian manifold, $U\subset M \setminus \partial M$ be a relatively compact open  subset with $C^{1,1}$ boundary.
Then there exist $r_{0}=r_{0}(U)>0$ and $C=C(U)>0$ such that 
	\begin{equation*}
		\rP_{g}(U) \leq \rP_{g}(E) + Cr^3,  \quad \text{ for all finite perimeter sets $E$ with  $E\Delta U\Subset B^{g}_r(q)$, all  $q\in \bar{U}$ and all $r\in (0,r_{0}]$}.  
	\end{equation*}
\end{lemma}

We conclude this short section by recalling a regularity result of Tamanini \cite{Tamanini}, refining previous celebrated works by De Giorgi \cite{DeGiorgi}.
\\ To this aim recall that, given a finite perimeter set $E\subset \R^3$ and an open bounded subset $V\subset \R^{3}$, the excess of $E$ relative to $V$ is defined as
	\begin{equation*}
		\Psi(E,V) := \rP_{\delta}(E,V) - \inf \{\rP_{\delta}(F,V) | F\Delta E \Subset V\}.
	\end{equation*}

\begin{theorem}[\cite{Tamanini}, Theorem 1]\label{thm:5}
	Let $U\subset \R^{3}$ be an open subset  and $E\subset \R^{3}$  be a set of finite perimeter. Assume there exist $\alpha=\alpha(E,U)\in (0,1)$, $C=C(E,U)>0$ and $R=R(E,U)>0$ such that 
	\begin{equation}\label{eq:ExcEstHp}
		\Psi(E,B_r(q)) \leq Cr^{2+2\alpha}, \quad \text{for all $q\in U$ and $r\in (0,R)$.}
	\end{equation}
 Then, up to replacing $E$ by the suitable a.e. representative, it holds that the reduced boundary $\partial^{*} E$ coincides with the topological boundary   $\partial E$, and $\partial E$ is a $C^{1,\alpha}$-hypersurface in $U$.
 \\Moreover, assuming that \eqref{eq:ExcEstHp} holds uniformly for a sequence $(E_k)_{k\in \N}$  convergent to $E_{\infty}$ in $L^{1}_{loc}$-topology, then $E_{\infty}$ satisfies \eqref{eq:ExcEstHp} as well; moreover    for any sequence of points $q_k \in \partial E_k$ converging to $q_{\infty} \in \partial E_{\infty}$,  the outward-pointing unit normal to $\partial E_k$ at $q_k$ converges to the outward-pointing unit normal to $\partial E_{\infty}$ at $q_{\infty}$.
\end{theorem}

\subsection{Perturbed geodesic spheres are outer-minimising}
Let $(M,g)$ be a Riemannian manifold. For a Borel subset $\Omega\subset M$ we denote with $ |\Omega|_{g}$ its volume and, if in addition it is a set of finite perimeter, with  $\rP_{g}(\Omega)$ its perimeter.
\\The \emph{isoperimetric profile function} $\mathcal I_{(M,g)}:[0,\infty)\to [0,\infty)$ of $(M,g)$ is defined by
\begin{equation}\label{eq:defIsopProf}
\mathcal I_{(M,g)}(v):=\inf\{\rP_{g}(\Omega)\, :\, \Omega \subset M \text{ is a finite perimeter set with } |\Omega|_{g}=v   \}.
\end{equation}
If  $(M,g)$ is an AF, complete,  three dimensional Riemannian manifold, the following holds:
\begin{equation}\label{eq:Ismallv}
\lim_{v\downarrow 0} v^{-2/3} \mathcal I_{(M,g)}(v)= 
\begin{cases}
(36 \pi)^{1/3} &\quad \text{ if } \partial M =\emptyset \\
(18 \pi)^{1/3} & \quad \text{ if } \partial M \neq \emptyset. 
\end{cases}
\end{equation}
This can be proved along the same lines as  \cite{NardCV}, noticing that the AF assumption guarantees that the pointed limit manifolds at infinity used in the proof of  \cite{NardCV} coincide with the Euclidean space $\R^{3}$ (thus one can relax the assumption of $C^{4,\alpha}$ bounded geometry, used in \cite{NardCV} to guarantee that the limit manifolds at infinity are $C^{3,\alpha}$ with $C^{2,\alpha}$ Riemannian metric).

\begin{lemma}
Let $(M,g)$ be an AF, complete,  three dimensional Riemannian manifold with (possibly empty) horizon boundary. Then, for every $V_{0}>0$  there exists $C=C(V_{0})>0$ such that 
\begin{equation}\label{eq:EII}
\rP_{g}(\Omega) \geq C\,  |\Omega|_{g}^{2/3}, \quad \text{for every $\Omega\subset M$ subset of finite perimeter, with $ |\Omega|_{g}\in (0,V_{0}]$}.
\end{equation}
\end{lemma}

\begin{proof}
From \cite{FloresNardulli}  we know that  $\mathcal I:(0,\infty)\to (0,\infty)$  is a continuous function.
Using \eqref{eq:Ismallv}, we conclude that  the function $[0,V_{0}]\ni v\mapsto v^{-2/3} \mathcal I_{(M,g)}(v)$ is continuous and never vanishes, giving the claim \eqref{eq:EII}.
\end{proof}

%We next state two auxiliary lemmas which will be useful to establish the main result of this section, namely Theorem \ref{Prop:OuterMin}.
%\\To that aim, let us fix some notation. Given a point $p\in M$, denote with $\phi^{p}_{g}:B^{g}_{\bar r}(p)\to \R^{3}$ a normal coordinate chart centred at $p$ with respect to the Riemannian metric $g$. It will be convenient to consider the rescaled metrics $g_{\rho}:=\rho^{-2} g$, for $\rho\in (0,1]$. It is a standard fact that, in $g_{\rho}$-normal coordinates centred at $p$, it holds: 
%\begin{equation}\label{eq:29}
%\|(g_{\rho})_{\mu \nu}-\delta_{\mu \nu} \|_{C^{k}\big(B^{g_{\rho}}_{r \rho^{-2}} (p) \big)} \leq C \, \rho^{2} r^{2}, \quad \forall \rho\in (0, 1], \; r\in (0, \bar r], 
%\end{equation}
%for some suitable $C=C(p)>0,  \bar r=\bar r(p)>0$.

The goal of the present section is to prove the next theorem, which will allow us to use the expansion for the Hawking mass of perturbed spheres in order to get a lower bound on the Bartnik mass in Theorem \ref{thm:7-8}.

\begin{theorem}\label{Prop:OuterMin}
Let $(M,g)$ be an AF, complete,  three dimensional Riemannian manifold  with non-negative scalar curvature and with (possibly empty) horizon boundary $\partial M$. Fix $p\in M\setminus \partial M$. Then there exist $\rho_{0}=\rho_{0}(p)>0, r_{0}=r_{0}(p)>0$ such that the perturbed geodesic spheres $S_{p,\rho}(w)$ are outer-minimising for every $\rho\in (0,\rho_{0}]$ and every $w\in C^{1}({\mathbb S}^{2})$ with $\|w\|_{C^{1}({\mathbb S}^{2})}\leq r_{0}$. %\footnote{A: this theorem is not true if we allow $M$ to have boundary}
\end{theorem}

Throughout this section, we will denote by $B_{p,\rho}(w)$ the perturbed geodesic ball enclosed by $S_{p,\rho}(w)$.

\begin{proof}
%Recall that we are assuming the existence of $C>0$ such that
%\begin{equation}\label{eq:EII}
%\rP_{g}(\Omega) \geq C |\Omega|_{g}^{2/3}, \quad \text{for every $\Omega\subset M$ subset of finite perimeter and finite volume},
%\end{equation}
%where $\rP_{g}(\Omega)$ (resp. $|\Omega|_{g}$)  denotes the perimeter (resp. the volume) of $\Omega$.
%

Using the standard $L^{1}_{loc}$-compactness and lower-semicontinuity of the perimeter, observe that, for any $\rho>0$ and $w\in C^{1}({\mathbb S}^{2})$, there exists a set of finite perimeter $\Omega_{p,\rho,w}\subset M$ minimising the perimeter among all sets of finite perimeter and finite volume containing  $B_{p,\rho}(w)$. Such $\Omega_{p,\rho,w}\subset M$ is called \emph{minimising hull} of  $B_{p,\rho}(w)$. Clearly
\begin{equation}\label{eq:POmegaB}
\rP_{g} (\Omega_{p,\rho,w}) \leq \rP_{g}(B_{p,\rho}(w)), \quad 0<|B_{p,\rho}(w)|_{g}\leq |\Omega_{p,\rho,w}|_{g}.
\end{equation}
We will show that, up to choosing the suitable a.e. representative, $\Omega_{p,\rho,w}=B_{p,\rho}(w)$ for any $\rho, \|w\|_{C^{1}}$ small enough. Note that, under such a smallness condition, it holds that $B_{p,\rho}(w)\Subset M\setminus \partial M$; thus, in the blow up arguments of steps 2-4, it is not restrictive to assume that $\partial M = \emptyset$.
\\

\textbf{Step 1}. We claim that there exists $C=C_{p}\geq1$ such that
\begin{equation}\label{eq:volOmegaBdd}
0<C^{-1} \leq \liminf_{\rho\downarrow 0} \frac{|\Omega_{p, \rho, w}|_{g}}{\rho^{3}} \leq  \limsup_{\rho\downarrow 0}  \frac{|\Omega_{p, \rho, w}|_{g}}{\rho^{3}} \leq C<\infty,  \quad \forall  \|w\|_{C^{1}({\mathbb S}^{2})}\leq 1.
\end{equation}
The lower bound is a direct consequence of the second inequality in \eqref{eq:POmegaB}, thus we are left to show the upper bound.
\\ From \cite[Theorem C.2]{CESYCPAM} (see also \cite[Theorem 3]{JLCrelle} after \cite{Huisk2006}),  we know that $\lim_{v\to \infty} \mathcal I_{(M,g)}(v)=+\infty$. Thus, \eqref{eq:POmegaB} implies that there exists $V_{0}>0$ such that $|\Omega_{p,\rho,w}|_{g}\leq V_{0}$ for all $\rho\in (0,1], \|w\|_{C^{1}}\leq 1$.
\\Hence, the upper bound follows from the isoperimetric inequality \eqref{eq:EII} and the perimeter bound in \eqref{eq:POmegaB}:
$$
\limsup_{\rho \downarrow 0} \frac{|\Omega_{p, \rho, w}|_{g}}{\rho^{3}} \overset{\eqref{eq:EII} }{\leq} C \limsup_{\rho \downarrow 0}  \frac{\rP_{g}(\Omega_{p, \rho, w})^{3/2}}{\rho^{3}} \overset{\eqref{eq:POmegaB}}{\leq} C  \limsup_{\rho \downarrow 0} \frac{\rP_{g}(B_{p,\rho}(w))^{3/2}}{\rho^{3}}<\infty.  
$$

\textbf{Step 2}. Blow up and $L^{1}_{loc}$-convergence to a Euclidean ball.
\\
In this step we blow up the Riemannian manifold $(M,g)$ at $p$ with scaling rate  $\rho^{-1}$ as $\rho\downarrow 0$, and we show that  the ``rescaled $ \Omega_{p,\rho,w}$''  converge as finite perimeter sets to the Euclidean ball of unit radius  $B^{\delta}_{1}(0)\subset \R^{3}$.
\\To this aim, consider the rescaled Riemannian metric $g_{\rho}:= \rho^{-2} \,  g$  and observe that the rescaled pointed manifolds $(M,g_{\rho}, p)$ converge to $(\R^{3}, \delta, 0)$ in the smooth pointed Cheeger-Gromov sense as $\rho\downarrow 0$; i.e., calling $B_{\rho^{-1}}^{\delta}(0)\subset \R^{3}$ the Euclidean ball of radius $\rho^{-1}$ centred at $0$,   for every $\rho\in(0,1]$ there exists a smooth map $\psi_{\rho}:B_{\rho^{-1}}^{\delta}(0)\to M$ which is diffeomorphic onto its image such that $\psi_{\rho}(0)=p$ and $\psi_{\rho}^{*} \, g_{\rho}\to \delta$ smoothy locally on $\R^{3}$ as $\rho\downarrow 0$.
\\ Combining the smooth pointed Cheeger-Gromov convergence with the compactness/lower semicontinuity of the perimeter, it follows that for every sequence $\rho_{n}\downarrow 0$ there exists a subsequence (not relabeled) and a set of finite perimeter $\bar{\Omega}\subset \R^{3}$ such that $\psi_{\rho_{n}}^{-1}(\Omega_{p,\rho_{n},w_{n}})\subset B_{\rho_{n}^{-1}}^{\delta}(0)\subset \R^{3}$ converges in $L^{1}_{loc}(\R^{3})$ to $\bar{\Omega}$ and
\begin{equation} \label{eq:POmegaleq}
\begin{split}
|\bar{\Omega}|_{\delta} \leq  \liminf_{n\to \infty}   |\Omega_{p,\rho_{n}, w_{n}}|_{g_{\rho_{n}}} = \liminf_{n\to \infty}  \rho_{n}^{-3} |\Omega_{p,\rho_{n}, w_{n}}|_{g}\overset{\eqref{eq:volOmegaBdd}}<\infty  \\
\rP_{\delta}(\bar{\Omega}) \leq \liminf_{n\to \infty}  \rP_{g_{\rho_{n}}} (\Omega_{p,\rho_{n}, w_{n}})= \liminf_{n\to \infty}  \rho_{n}^{-2} \rP_{g} (\Omega_{p,\rho_{n}, w_{n}}) 
\end{split}
\qquad \text{for all } \rho_{n}, \, \|w_{n}\|_{C^{1}({\mathbb S}^{2})}\to 0.
\end{equation}
Since by construction $B_{p,\rho_{n}}(w_{n}) \subset \Omega_{p,\rho_{n},w_{n}}$ and $\psi_{\rho_{n}}^{-1} (B_{p,\rho_{n}}(w_{n}))\to B^{\delta}_{1}(0)$ smoothly as $n\to \infty$, it also holds 
\begin{equation}\label{eq:PB1delta=}
B_{1}^{\delta}(0)\subset \bar{\Omega}, \qquad \lim_{n\to \infty} \rho_{n}^{-2} \rP_{g} (B_{p,\rho_{n}}(w_{n}))=  \lim_{n\to \infty}  \rP_{g_{\rho_{n}}} (B_{p,\rho_{n}}(w_{n}))= \rP_{\delta} (B_{1}^{\delta}(0)), \quad \forall \rho_{n}, \, \|w_{n}\|_{C^{1}({\mathbb S}^{2})}\to 0.
\end{equation}
Recalling that $\rP_{g}(\Omega_{p,\rho,w})\leq \rP_{g}(B_{p,\rho_{n}}(w_{n}))$, the combination of \eqref{eq:POmegaleq} and \eqref{eq:PB1delta=} yields:
\begin{equation}\nonumber
B_{1}^{\delta}(0)\subset \bar{\Omega}, \; |\bar{\Omega}|_{\delta}<\infty,\; \rP_{\delta}(\bar{\Omega})\leq  \rP_{\delta}(B_{1}^{\delta}(0)).
\end{equation}
The rigidity in the Euclidean isoperimetric inequality yields that $|\bar{\Omega} \Delta B_{1}^{\delta}(0)|_{\delta}=0$. By the arbitrariness of the sequences $(\rho_{n})$ and  $(w_{n})$, we conclude that
\begin{equation}\label{eq:OmegatoB1L1}
\psi_{\rho}^{-1}(\Omega_{p,\rho,w}) \to B_{1}^{\delta}(0) \text{ in } L^{1}_{loc}(\R^{3}) \text{ as } \rho\to 0,\, \|w\|_{C^{1}({\mathbb S}^{2})}\to 0. 
\end{equation}

\textbf{Step 3}. Improving the convergence via regularity theory.
\\Let us first fix some notation. Given a point $p\in M$, denote with $\phi^{p}_{g}:B^{g}_{{\rm Inj}_{p}}(p)\to \R^{3}$ a normal coordinate chart centred at $p$ with respect to the Riemannian metric $g$. Again, we consider the rescaled metrics $g_{\rho}:=\rho^{-2} g$, for $\rho\in (0,1]$. Notice that we can (and will)  choose $\psi_{\rho}$ from step 2 to be  $\psi_{\rho}=(\phi^{p}_{g_{\rho}})^{-1}$.
\\

We first claim that there exist constants $C=C(p)>0$ and $\rho_{0}=\rho_{0}(p)$ such that%
	\begin{equation}\label{eq:32}
		\left| \rP_{\delta}(\phi^p_{g_{\rho}}(\Omega_{p,\rho,w}),B^{\delta}_r(\phi^p_{g_{\rho}}(q))) - \rP_{g_{\rho}}(\Omega_{p,\rho,w},B^{g_{\rho}}_r(q)) \right| \leq C r^{4}, \quad \text{for all $r\in (0,1]$, $\rho\in (0, \rho_{0}]$, $q\in B^{g_{\rho}}_{10}(p)$, $\|w\|_{C^{1}}\leq 1$}.
	\end{equation}
First, for any part of $\partial \Omega_{p,\rho,w}$ coinciding with the submanifold $S_{p,\rho}(w)=\partial B_{p,\rho}(w)$ we can use that
\begin{equation}\label{eq:AreaSprho}
\rP_{g_{\rho}}(B_{p,\rho}(w),  B^{g_{\rho}}_r(q)) \leq C\,  r^2,\quad  \text{for all $r\in (0,1]$, $\rho\in (0, \rho_{0}(p)]$, $q\in B^{g_{\rho}}_{10}(p)$, $\|w\|_{C^{1}}\leq 1$}.
\end{equation}
	Second,  away from the intersection points with $S_{p,\rho}(w)$, by construction $\Omega_{p,\rho,w}$ is locally perimeter minimising with respect to the metric $g_{\rho}$. Thus we can apply Lemma \ref{Lem:7} with $E= \Omega_{p,\rho,w}$  to obtain:
\begin{equation}\label{eq:157}
	\rP_{g_{\rho}}(\Omega_{p,\rho,w},B^{g_{\rho}}_r(q)) \leq Cr^2, \quad \text{for all $r\in (0,1]$, $\rho\in (0, \rho_{0}]$,  $\|w\|_{C^{1}}\leq 1$, $q\in B^{g_{\rho}}_{10}(p)$ with $B^{g_{\rho}}_r(q)\cap S_{p,\rho}(w)=\emptyset$}.
\end{equation}
The constant coming from Lemma \ref{Lem:7} is independent of $\rho$ because the $\Omega_{p,\rho,w}$ have uniformly bounded perimeter and $g_{\rho}$ is smoothy converging to the Euclidean metric $\delta$  on $B_{11}^{\delta}(0)$. 
\\Combining \eqref{eq:AreaSprho} and \eqref{eq:157} with Lemma  \ref{Lem:14} gives the claim \eqref{eq:32}.
\\

We next claim that the sequence $\phi^p_{g_{\rho}}(\Omega_{p,\rho,w})$ satisfies the conditions of Theorem \ref{thm:5}.
\\Let $U = B^{\delta}_{10}(0)$. Notice that for $\rho_{0}=\rho_{0}(p)>0$ small enough we have that ${\rm Inj}_{p}^{g_{\rho}}>11$, so that  $U = \phi^p_{g_{\rho}}(B^{g_{\rho}}_{10}(p))$ for all $\rho\in(0, \rho_{0}]$.

	Let $F\subset M$ be such that $F\Delta \Omega_{p,\rho,w} \Subset B^{g_{\rho}}_r(q)$, and define $F' := F\cup B_{p,\rho}(w)$. Then, by the minimising assumption on $\Omega_{p,\rho,w}$, we have:
	\begin{equation*}
		\rP_{g_{\rho}}(\Omega_{p,\rho,w}) \leq \rP_{g_{\rho}}(F').
	\end{equation*}
From standard properties of the perimeter, we have that
	\begin{equation*}
		\rP_{g_{\rho}}(F\cup B_{p,\rho}(w)) + \rP_{g_{\rho}}(F\cap B_{p,\rho}(w)) \leq \rP_{g_{\rho}}(F) + \rP_{g_{\rho}}(B_{p,\rho}(w)).
	\end{equation*}
Applying Lemma \ref{Lem:30} with $U = B_{p,\rho}(w)$ gives that there exists $\bar{r}=\bar{r}(p)\in (0,1], \, C=C(p)>0$ such that
	\begin{equation*}
		\rP_{g_{\rho}}(B_{p,\rho}(w)) \leq \rP_{g_{\rho}}(G) + Cr^3, \quad \text{for all $G\Delta B_{p,\rho}(w)\Subset B^{g_{\rho}}_r(q)$, $q\in B^{g_{\rho}}_{10}(p)$, $r\in (0,\bar{r}]$. }
	\end{equation*}
 Letting $G = F\cap B_{p,\rho}(w)$ and combining the three previous inequalities gives:
	\begin{equation*}
		\rP_{g_{\rho}}(\Omega_{p,\rho,w}) \leq \rP_{g_{\rho}}(F') \leq \rP_{g_{\rho}}(F) + Cr^3, \quad \text{for all $F\Delta \Omega_{p,\rho,w} \Subset B^{g_{\rho}}_r(q)$,  $q\in B^{g_{\rho}}_{10}(p)$, $r\in (0,\bar{r}]$.}
	\end{equation*}
 Since the sets $\Omega_{p,\rho,w}$ and $F$ coincide outside of $B^{g_{\rho}}_r(q)$, the last estimate is equivalent to
	\begin{equation}\label{eq:146}
		\rP_{g_{\rho}}(\Omega_{p,\rho,w},B^{g_{\rho}}_r(q)) \leq \rP_{g_{\rho}}(F,B^{g_{\rho}}_r(q)) + Cr^3, \quad \text{for all $F\Delta \Omega_{p,\rho,w} \Subset B^{g_{\rho}}_r(q)$,  $q\in B^{g_{\rho}}_{10}(p)$, $r\in (0,\bar{r}]$.}
	\end{equation}
We can finally estimate the excess:
	\begin{align}
		\Psi(\phi^p_{g_{\rho}}(\Omega_{p,\rho,w}),B^{\delta}_r(\phi^p_{g_{\rho}}(q)))  ) &:= \rP_{\delta}(\phi^p_{g_{\rho}}(\Omega_{p,\rho,w}),B^{\delta}_r(\phi^p_{g_{\rho}}(q))) \nonumber \\
		&\qquad - \inf \left\{\rP_{\delta}\big(\phi^p_{g_{\rho}}(F),B^{\delta}_r(\phi^p_{g_{\rho}}(q))\big) | \phi^p_{g_{\rho}}(F)\Delta \phi^p_{g_{\rho}}(\Omega_{p,\rho,w}) \Subset  B^{\delta}_r(\phi^p_{g_{\rho}}(q)) \right\} \nonumber\\
		&\overset{\eqref{eq:32}, \eqref{eq:30}}{=} \rP_{g_{\rho}}(\Omega_{p,\rho,w},B^{g_{\rho}}_r(q)) + \cO(r^4) \nonumber \\
		&\qquad - \inf \left\{(1+ \cO(r^2))\rP_{g_{\rho}}(F,B^{g_{\rho}}_r(q)) + \cO(r^4)|\phi^p_{g_{\rho}}(F)\Delta \phi^p_{g_{\rho}}(\Omega_{p,\rho,w}) \Subset B^{\delta}_r(\phi^p_{g_{\rho}}(q)) \right\} \nonumber \\
		&\overset{\eqref{eq:146} }\leq \rP_{g_{\rho}}(\Omega_{p,\rho,w},B^{g_{\rho}}_r(q)) + \cO(r^4) - [(1+ \cO(r^2))(\rP_{g_{\rho}}(\Omega_{p,\rho,w},B^{g_{\rho}}_r(q)) - Cr^3) + \cO(r^4)] \nonumber \\
		&\overset{\eqref{eq:157}}\leq C r^{3} , \quad \text{for all   $r\in (0,\bar{r}]$, $\|w\|_{C^{1}}\leq 1$, $q\in B^{\delta}_{10}(0)$,}
	\end{align}
	for some constant $C=C(p)>0$.
	Thus, the family $\phi^p_{g_{\rho}}(\Omega_{p,\rho,w})$ satisfies the assumptions of  Theorem \ref{thm:5}  with $\alpha = \frac{1}{2}$. 
	We infer that $\partial \Omega_{p,\rho,w}$ are $C^{1,1/2}$ surfaces, and  the outward pointing unit normals of $\partial \Omega_{p,\rho,w}$ converge to the outward pointing unit normal of $\partial B^{\delta}_1(0)$ as $\rho, \|w\|_{C^{1}}\to 0$. This implies that there exists $r_{0}=r_{0}(p)>0$ and $\rho_{0}=\rho_{0}(p)>0$ such that, for all $\|w\|_{C^{1}}\leq r_{0}$ and  $\rho\in (0,\rho_{0}]$, the surfaces $\partial \Omega_{p,\rho,w}$ are $C^{1, 1/2}$ graphs over $\partial B^{\delta}_1(0)$. Moreover, such graphs converge to $0$ in the $C^{1,1/2}$ topology as $\rho\to 0, \|w\|_{C^{1}}\to 0$.
\\

\textbf{Step 4}. Conclusion by a first variation argument.
\\First of all, notice that for $\rho>0$ small enough (depending only on $p$), the surface $\phi^p_{g_{\rho}}(S_{p,\rho}(w))$ is a graph over $\partial B^{\delta}_1(0) = {\mathbb S}^2$.
Combining this fact with step 3, we get that for $\rho>0$ and $\|w\|_{C^{1}}$ small enough  (depending only on $p$) the surface $\partial \phi^p_{g_{\rho}}(\Omega_{p,\rho,w})$ is a graph over $\phi^p_{g_{\rho}}(S_{p,\rho}(w))$. Thus $\partial\Omega_{p,\rho,w}$ is parameterized by:
	\begin{equation*}
		\partial\Omega_{p,\rho,w} = \{\exp^{g_{\rho}}_q{(u_{\rho,w}(q)\hat{N}(q))} : q\in S_{p,\rho}(w)\}
	\end{equation*}
	for some function $u_{\rho,w}\in C^1(S_{p,\rho}(w))$. Notice that $u_{\rho,w}\geq 0$, since  by assumption  $S_{p,\rho}(w) \subset \Omega_{p,\rho,w}$.
	\\Using that both $\partial \phi^p_{g_{\rho}}(\Omega_{p,\rho,w})$ and $\phi^p_{g_{\rho}}(S_{p,\rho}(w))$ converge to $B^{\delta}_1(0)$ as $\rho\to 0, \|w\|_{C^{1}}\to 0$, we also have 
	\begin{equation}\label{eq:urhowto0}
	\norm{u_{\rho,w}}_{C^1} \rightarrow 0 \text{ as $\rho \rightarrow 0$, $\| w\|_{C^{1}}\to 0$.}
	\end{equation}
	For fixed  $\rho>0$, consider the Banach space $C^1(S_{p,\rho}(w))$ of graph functions over $S_{p,\rho}(w)$. The area functional in $g_{\rho}$-metric, $\rA_{g_{\rho}} : C^1(S_{p,\rho}(w)) \rightarrow \R$, is Fr\'echet differentiable at 0, 	with derivative $d(\rA_{g_{\rho}})_0 \in {\mathcal L}(C^1(S_{p,\rho}(w)),\R)$ such that:
	\begin{equation*}
		\rA_{g_{\rho}}(0 + h) = \rA_{g_{\rho}}(0) + d(\rA_{g_{\rho}})_0(h) + o(\norm{h}_{C^1}), \quad \text{for all $h\in C^1(S_{p,\rho}(w))$}.
	\end{equation*}
	In particular, setting $h = u_{\rho,w}$ gives:
	\begin{equation*}
		\rA_{g_{\rho}}(0 + u_{\rho,w}) = \rA_{g_{\rho}}(0) + d(\rA_{g_{\rho}})_0(u_{\rho,w}) + o(\norm{u_{\rho,w}}_{C^1}) .
	\end{equation*}
	Comparing this to the first variation (Gateaux derivative) of $\rA_{g_{\rho}}$, we see that: 
	\begin{equation*}
		d(\rA_{g_{\rho}})_0(u_{\rho,w}) = \int_{S_{p,\rho}(w)}H^{g_{\rho}}_{S_{p,\rho}(w)}u_{\rho,w}\ dV_{g_{\rho}}.
	\end{equation*}
	Therefore:
	\begin{equation*}
		\rA_{g_{\rho}}(0 + u_{\rho,w}) = \rA_{g_{\rho}}(0) + \int_{S_{p,\rho}(w)}H^{g_{\rho}}_{S_{p,\rho}(w)}u_{\rho,w}\ dV_{g_{\rho}} + o(\norm{u_{\rho,w}}_{C^1}).
	\end{equation*}
	Notice that the left hand side coincides with $\rA_{g_{\rho}}(\partial\Omega_{p,\rho,w})=\rP_{{g_{\rho}}}(\Omega_{p,\rho,w})$. Moreover,  for small $\rho$, we have that 
	$$H^{g_{\rho}}_{S_{p,\rho}(w)} = 2 + \cO(\rho^2) > 0,$$
	indeed, by \eqref{eq:ExpH} we know $H^g_{S_{p,\rho}(w)} = 2\rho^{-1} + \cO(\rho)$, which gets multiplied by a factor $\rho$ due to the scaling of the metric.  Thus, for small $\rho$ we get: 
	\begin{equation*}
		\rA_{g_{\rho}}(\partial \Omega_{p,\rho,w}) \geq \rA_{g_{\rho}}(0) = \rA_{g_{\rho}}(S_{p,\rho}(w)),
	\end{equation*}
	with equality if and only if $u_{\rho,w}\equiv 0$.
	\\Since by construction $\partial \Omega_{p,\rho,w}$ is the minimising hull of   $B_{p,\rho}(w)$, we have $\rA_{g_{\rho}}(\partial \Omega_{p,\rho,w})\leq  \rA_{g_{\rho}}(S_{p,\rho}(w))$. Thus $u_{\rho,w}\equiv 0$, that is $\partial \Omega_{p,\rho,w} = S_{p,\rho}(w)$. Hence $S_{p,\rho}(w)$ is outer-minimising for  $\rho>0$ and $\|w\|_{C^{1}}$ sufficiently small (smallness depending only on $p$).

\end{proof}

%\addcontentsline{toc}{section}{Proof of Bartnik Mass Theorems}
\subsection{Proof of the Bartnik mass Theorem \ref{thm:7-8} }

Firstly, recall that  \eqref{mBgeqmH} gives: 
\begin{equation*}
    m_B(\Omega) \geq m_H(\partial\Omega).
\end{equation*}
For every perturbed geodesic sphere $S_{p,\rho}(w) \subset \Omega$, with $\rho>0$ and $\|w\|_{C^{1}}$ sufficiently small (only depending on $p$), we have:
\begin{equation}\label{eq:mBOmegamHSp}
    m_H(S_{p,\rho}(w)) \leq m_B(B_{p,\rho}(w)) \leq m_B(\Omega) ,
\end{equation}
where the first inequality follows from \eqref{mBgeqmH} applied to $S_{p,\rho}(w)$ and the second follows by the monotonicity property \eqref{eq:monmB}, which applies thanks to the outer-minimising property of $S_{p,\rho}(w)$ proved in Theorem \ref{Prop:OuterMin}. 
For every $p\in \Omega\setminus \partial M$ and for $\rho>0$ sufficiently small (depending only on $p$), let $w_{p,\rho}\in C^{4,\alpha}(\mathbb S^{2})^{\perp}\subset C^{1}(\mathbb S^{2})$ be given by Lemma \ref{lem:estw}, i.e. $w_{p,\rho}$ is the optimal perturbation extremizing (in $C^{4,\alpha}(\mathbb S^{2})^{\perp}$) the Hawking mass under area constraint. Combining the expansion \eqref{eq:mHSprhow} and the inequality \eqref{eq:mBOmegamHSp}, specialised to $w=w_{p,\rho}$, gives the claimed lower bound on the Bartnik mass:
$$
m_B(\Omega) \geq  \frac{1}{12} \Sc_p\rho^{3} +  \left(\frac{1}{120} \Delta \Sc(p) + \frac{1}{90}\norm{S_p}^2 - \frac{1}{144}\Sc_p^2  \right)  \rho^5 + \cO(\rho^6).
$$
If $m_B(\Omega)=0$, from \eqref{eq:mBOmegamHSp} we infer that $m_H(S_{p,\rho}(w))$ is non-positive for every perturbed geodesic sphere of sufficiently small radius. We can then apply Theorem \ref{prop:Riem0} to $\Omega$ to infer that $\Omega \setminus \partial M$ is locally isometric to Euclidean $\R^3$.

\hfill$\Box$

\section{Other rigidity results}\label{Sec:OtherRigRes}

In this section, we collect other rigidity results involving the Hawking mass in various different settings. As the reader will appreciate, the proofs will be quite straightforward thanks to the work done in the previous sections (in particular we will make repeated use of the expansion of the Hawking mass obtained in Proposition \ref{prop:expHawk}).

\subsection{Rigidity for the generalised Hawking mass (i.e. for non-zero cosmological constant)}

%We can obtain similar results for the other spaces of constant sectional curvature if we adjust our assumptions on $(M^3,g)$ accordingly and use a more general form of the Hawking mass.
The standard Hawking mass \eqref{eq:defHaw} is relevant when the ambient space is a Riemannian $3$-manifold  with non-negative scalar curvature. Such metrics are natural when the cosmological constant is zero. Instead,  when the cosmological constant $\Lambda$ is negative (resp. positive), it is is more natural to consider metrics with scalar curvature bounded below by a negative (resp. positive) constant. Indeed, the Dominant Energy Condition coupled with the  Einstein Constraint Equations imply that the scalar curvature of a totally geodesic space-like hypersurface (i.e. the so-called time-symmetric case) is bounded below by $2\Lambda$.
\\ When $\Lambda$ is negative (resp. positive) it is standard to choose the normalization $\Lambda=-3$ (resp. $\Lambda=3$) and it is natural to compare the geometry of a totally geodesic space-like hypersurface with a space-form of constant sectional curvature $K=-1$ (resp. $K=1$). 
\\When $\Lambda\in \{-3,0,3\}$, it is also natural to modify the Hawking mass as follows (see for instance \cite{neves}).
\begin{definition}\label{def:GenHaw}
Let  $K \in \{-1,0,1\}$ and let $(M^3,g)$ be a $3$-dimensional Riemannian manifold with $\Sc\geq 6K$. The generalized Hawking mass  of an immersed sphere  $\Sigma$ in $M$ is
\begin{equation}\label{eq:defGenHaw}
    m_{H}(\Sigma) := \sqrt{\frac{|\Sigma |}{(16\pi)^3}}\left(16\pi - \int_{\Sigma} \left( H^2 + 4K \right) \, dV_{\Sigma}\right).
\end{equation}
\end{definition}
Arguing along the lines as the proof Theorem \ref{thm:1} we obtain the following rigidity result involving the generalised Hawking mass.
%\begin{remark}\label{Rem:2}
%The generalised Hawking mass \eqref{eq:defGenHaw} allows for a non-zero cosmological constant $\Lambda$ in Einstein's field equation. Recall that $\Lambda \neq 0$ is equivalent to the energy density of a vacuum being non-zero.
% Setting $K\neq 0$ counteracts the effect of the ``curvature effects of the vacuum energy''. For example, $K=-1$ ensures that stable constant mean curvature spheres in manifolds with $\Sc\geq -6$ still have non-negative mass and that centred coordinate spheres in Schwartzschild Anti-de Sitter space have Hawking mass equal to the total mass \cite{Christodoulou,Sun,Chodosh}.
%\end{remark}
%
%Firstly, we can straight away obtain a generalisation of Theorem \ref{thm:1}:

\begin{theorem}\label{thm:9}
Let $(M^3,g)$ be a connected, complete Riemannian manifold without boundary,  with scalar curvature $\Sc\geq 6K$ where $K \in \{-1,0,1\}$. If for every $p \in M$ there is a neighbourhood $U$ of $p$ such that the generalised Hawking mass of every embedded sphere contained in $U$ is non-positive, then $(M^3,g)$ is isometric to a space form of constant sectional curvature $K$.
\end{theorem}
%Is there an obvious counter example to show why the rigidity part won't work for the case K = -1?? Hyperboloid?? Is this a space form?? If not then maybe it is true but unfortunately the Wolf proof wont work because I dont have the nice splitting of the isometry. Hyperbolic 3 manifolds are too complicated! No classification yet?? 
\begin{proof}
Let us compute the generalized Hawking mass of the perturbed geodesic spheres as before. Notice that the only difference with the standard Hawking mass is the extra term $4K\int_{{\mathbb S}^2}\sqrt{\det\mathring{g}}$.
Recalling \eqref{eq:AreaSprhow},  this is easily evaluated up to fourth order as
\begin{equation*}
    4K\int_{{\mathbb S}^2}\sqrt{\det\mathring{g}} = 16K\pi \rho^2 - \frac{8K\pi}{9}\Sc_p\rho^4+\cO_{p}(\rho^{5}).
\end{equation*}
Recalling \eqref{eq:77}, we obtain the following expansion for the  generalized Hawking mass of $S_{p,\rho}(w_{p,\rho})$:
\begin{equation}\label{eq:expmHSprhow}
    m_H(S_{p,\rho}(w_{p,\rho})) = \sqrt{\frac{|S_{p,\rho}(w_{p,\rho})|}{(16\pi)^3}}\left[\left(\frac{8\pi}{3}\Sc_p - 16K\pi\right)\rho^2 + \left(\frac{4\pi}{15}\Delta \Sc_p + \frac{16\pi}{45}\norm{S_{p}}^2 - \frac{4\pi}{27}\Sc_p^2 + \frac{8K\pi}{9}\Sc_p\right)\rho^4 + \cO_{p}(\rho^5)\right].
\end{equation}
Assuming that $m_H(S_{p,\rho}(w_{p,\rho}))\leq 0$ for $\rho>0$ sufficiently small yields  $\Sc_p \leq 6K$.
Since we assumed that $\Sc_p\geq 6K$ for all $p\in M$,  we have
\begin{equation}\label{eq:Sc=6K}
 \Sc\equiv 6K.
 \end{equation}
Inserting \eqref{eq:Sc=6K} into \eqref{eq:expmHSprhow} and evaluating at sufficiently small $\rho>0$ gives  $\frac{16\pi}{45}\norm{S_{p}}^2 \leq 0$ for every $p\in M$. Therefore the trace-free Ricci tensor vanishes: 
\begin{equation}\label{eq:S=0bis}
S \equiv 0.
 \end{equation}
Putting together \eqref{eq:Sc=6K} and \eqref{eq:S=0bis} gives
\begin{equation*}
\Ric \equiv 2K g.
\end{equation*}
Recalling that in dimension three the Riemann curvature tensor can be written as $\Rm = \Ric\KN g - \frac{1}{4} \Sc \, g\KN g $ where $\KN$ is the Kulkarni-Nomizu product (see for instance \cite[Corollary 7.26]{LeeRM}) we conclude that $g$ has  constant sectional curvature $K$.
\end{proof}

\begin{remark}\label{rem:HMHyp}
Notice that \eqref{eq:expmHSprhow} actually gives a strictly positive (yet small) lower bound on the generalised Hawking mass of the optimally perturbed geodesic sphere $S_{p,\rho}(w_{p,\rho})$, if $\Sc_p>6K$ or $\Sc \equiv 6K \, \& \, \|S_p\|\neq 0$ (and, as observed in the proof, such a point $p\in M$ always exists if $(M,g)$ does not have constant sectional curvature and $\Sc\geq 6K$). 

Even though the proof of Theorem \ref{Prop:OuterMin} made use of the AF assumption, we expect that for $\rho>0$ sufficiently small the surface $S_{p,\rho}(w_{p,\rho})$ is outward minimising also in a locally asymptotically hyperbolic framework.  We did not push in that direction since it does not seem to be useful in order to obtain a lower bound on the hyperbolic analogue of the Bartnik mass, in the spirit of Theorem \ref{thm:7-8}. Indeed, if for $\rho>0$ sufficiently small the surface $S_{p,\rho}(w_{p,\rho})$ is outward-minimising, one can start the (weak) Inverse Mean Curvature Flow in the sense of Huisken-Ilmanen \cite{HuiskenIlmanen} and the generalised Hawking mass is monotone non-decreasing also in this setting (see for instance \cite{neves} for more details). However, as proved by Neves \cite{neves}, it may happen that the asymptotic limit of the Hawking mass along the IMC flow exceeds the hyperbolic-ADM mass of the manifold, thus preventing us to repeat the proof of Theorem  \ref{thm:7-8} in the $K=-1$ case.  
\end{remark}

%Research hyperbolic space forms ("hyperbolic 3-manifolds") to see if any analogous rigidity statement can be proved.

%\begin{remark}
%We note that the rigidity part of of \ref{thm:1} does not generalise. For example the hyperboloid is ALSC. Is this not hyperbolic space though???
%\end{remark}

\subsection{$\mathbb{R}^3$ and $\mathbb{H}^3$ rigidity in the homogeneous setting}

In this section we replace the ALSC assumption in the rigidity Theorem \ref{thm:1} with the homogeneity condition.  Recall that  a Riemannian manifold $(M,g)$ is said to be \emph{homogeneous} if its isometry group ${\rm Isom}(M,g)$ acts transitively on $M$. In other words, if for every  $p,q \in M$, there exists $ \gamma \in {\rm Isom}(M,g)$ such that $\gamma(p) = q$.

Whilst of course our spatial universe is not  \textit{exactly} homogenous,  at cosmological scales the homogeneity provides a useful idealisation. Indeed,  spatial homogeneity is a standard assumption in Cosmology. For instance, it leads to an exact solution of Einstein's field equations, known as the Robertson-Walker metric for space-time \cite{Weinberg,HawkingEllis}.

\begin{theorem}\label{thm:2}
Let $K \in \{-1,0,1\}$ and let  $(M^3,g)$ be a connected, homogeneous Riemannian manifold with scalar curvature $\Sc\geq 6K$. If for every $p \in M$ there is a neighbourhood $U$ of $p$ such that the generalised Hawking mass \eqref{eq:defGenHaw} of every embedded sphere contained in $U$ is non-positive, then $(M^3,g)$ is isometric to
\begin{itemize}
\item  $\mathbb{H}^3$ (if $K=-1$); or 
\item $\R^m \times \mathbb{T}^{3-m}$, for some $0\leq m \leq 3$, where $\mathbb{T}^{3-m}$ is a flat torus of dimension $3-m$  (if $K=0$); or
\item ${\mathbb S}^3/\Gamma$  for some finite subgroup of isometries  $\Gamma< {\rm Iso} ({\mathbb S}^3)$ acting freely on ${\mathbb S}^3$  (if $K=1$).
\end{itemize}
\end{theorem}

\begin{proof}
Since homogeneity implies completeness and that $\partial M=\emptyset$, Theorem \ref{thm:9} yields that $M$ has constant sectional curvature $K$. The conclusion now follows from the classical classification of homogenous spaces of constant sectional curvature  (see for instance \cite[Theorem 2.7.1]{Wolf}). Notice that in the case $K=1$ one can be more precise about the type of quotients appearing, at the price of a more technical statement.
\end{proof}

%
%\begin{remark}
%We have used a shorter version than the one referenced because we have ignored the $K > 0$ possibility. One could use the full theorem to extend the conclusion of our Theorem \ref{thm:2} in the $K > 0$ homogeneous case. The proof of the theorem relies on the fact that (M,g) must be a quotient of either $\mathbb{R}^n$ or $\mathbb{H}^n$ by a group of isometries $\Gamma$ and in fact every $\gamma \in \Gamma$ is a so-called Clifford translation. This means that the Riemannian distance between $p$ and $\gamma(p)$ is constant for all $p\in \mathbb{R}^n$ or $\mathbb{H}^n$. In the Euclidean case $\gamma$ is a usual translation and in the hyperbolic case it turns out that $\gamma = Id_{\mathbb{H}^n}$.
%\end{remark}
%

\subsection{$\mathbb{R}^3$ and $\mathbb{H}^3$ rigidity under global asymptotic volume growth assumptions}

In this section we replace the ALSC assumption in the rigidity Theorem \ref{thm:1} with a global volume growth assumption (satisfied, for instance, by asymptotically flat  spaces and asymptotically hyperbolic spaces).  Such an assumption is obtained by comparing the volume growth of metric balls in the space under consideration with metric balls in an appropriate model space, as the radius goes to infinity.

\begin{definition}
Let $K\in \{-1,0\}$ and let $(M^3,g)$ be a complete Riemannian manifold without boundary and  with $\Sc\geq 6K$.
We say that $(M^3,g)$ satisfies the $K$-Global Asymptotic Volume property ($K$-GAVP) if:
\begin{equation}\label{eq:K-GAV}
    \limsup_{r\to\infty} \frac{{\rm Vol}_g(B^g_r(p))}{{\rm Vol}_K(r)} \geq 1, 
\end{equation}
where ${\rm Vol}_K(r)$ denotes the volume of a metric ball of radius $r$ in the $3$-dimensional simply connected space of constant sectional curvature $K$.
\end{definition}

There are various ways to define an asymptotically hyperbolic manifold. For the conformal compactification approach, see for instance Wang \cite{Wang}. In closer analogy to Definition \ref{def:AFRiemMan}, we take the asymptotic chart approach (see for instance \cite{ChruscielHerzlich, Herzlich, Sakovich} for discussions on the physical relevance of such metrics):

\begin{definition}\label{def:4}
A 3-dimensional Riemannian manifold $(M,g)$ is said to be \emph{asymptotically hyperbolic (AH)} if there is a compact subset $C\subset M$ and a diffeomorphism $\phi : M\setminus C \rightarrow \R^3 \setminus \overline{B_1(0)}$ such that the metric satisfies:
\begin{align*}
    \abs{g_{\mu\nu} - (g_{\mathbb{H}^3})_{\mu\nu}} &= \cO(r^{-s})
\end{align*}
for some $s>0$ in the chart $\phi$. Here $\mathbb{H}^3 = (\R^3,g_{\mathbb{H}^3})$ denotes the standard hyperbolic space with metric $g_{\mathbb{H}^3} = \frac{1}{1+r^2}dr^2 + r^2g_{\mathbb {\mathbb S}^2}$ (in polar coordinates).
\end{definition}

\begin{remark}\label{Rem1}
AF (resp. AH) manifolds satisfy the $0$-Global Asymptotic Volume property (resp. $-1$-GAVP). For the AF case, consider a straight line segment $\gamma$ in $\R^3$,  parameterised on the interval $(0,\sqrt{r})$ by $\gamma(t) = \sqrt{r}t\hat{a}$, for some unit vector $\hat{a}\in \R^3$. It is easy to see that if $(M,g)$ is AF  with metric $g|_{M\setminus C} = \bar{g} + h$, with  $h = \cO(r^{-s})$ for some $s>0$, then there exists a constant $A>0$ such that
$$
{\rm Length}_g(\gamma)\leq A + (1+ A r^{-s}) \,  {\rm Length}_{\delta} (\gamma),  \quad \text{for all } r\geq 1,
$$
giving that $B^{\delta}_r(p) \subset B^g_{A r(1 + Ar^{-s})}(p)$.  Therefore ${\rm Vol}_\delta(B^{\delta}_r(p)) \leq {\rm Vol}_g(B^g_{A+ r(1 + Ar^{-s})}(p))$. Sending $r\rightarrow \infty$ yields the $0$-GAVP:
\begin{equation*}
    \limsup_{r\to\infty} \frac{{\rm Vol_g}(B^g_r(p))}{{\rm Vol_\delta}(B^{\delta}_r(p))} \geq 1. 
\end{equation*}
The hyperbolic case is analogous, replacing straight lines by  minimising geodesics in $\mathbb{H}^3$ and applying Definition \ref{def:4}. 
\end{remark}

\begin{theorem}\label{thm:4}
Let $K\in \{ -1, 0\}$. Let $(M^3,g)$ be a connected, complete Riemannian manifold without boundary, with scalar curvature $\Sc\geq 6K$, and satisfying the $K$-Global Asymptotic Volume property. If for every $p \in M$ there is a neighbourhood $U$ of $p$ such that the generalised Hawking mass of every embedded sphere contained in $U$ is non-positive, then $(M,g)$ is isometric to $\mathbb{H}^3$ (if $K=-1$) or $\R^3$ (if $K=0$).
\end{theorem}

\begin{proof}
By recalling  Theorem \ref{thm:9}, the assumptions on the scalar curvature and Hawking mass imply that $(M^3,g)$ has constant sectional curvature $K$. Therefore the Ricci curvature of $(M^3,g)$ is identically equal to  $2Kg$.  The Bishop-Gromov Theorem (see for instance \cite{LeeRM,Petersen}) gives that the ratio
$$
r\mapsto  \frac{{\rm Vol}_g(B^g_r(p))}{{\rm Vol}_K(r)} \text{ is non-increasing and is bounded above by $1$}.
$$
Moreover, if equality holds for some $r>0$, then the metric ball $B^g_r(p)$ in $(M,g)$ is isometric to a metric ball of radius $r$ in the simply connected $3$-dimensional space of constant sectional curvature $K$.  It follows that
\begin{equation}\label{eq:RatVol1}
\limsup_{r\to \infty}  \frac{{\rm Vol}_g(B^g_r(p))}{{\rm Vol}_K(r)} \leq 1,
\end{equation}
with equality if and only if for every $r>0$ the metric ball  $B^g_r(p)$ in $(M,g)$ is isometric to a metric ball of radius $r$ in the simply connected $3$-dimensional space of constant sectional curvature $K$ or, equivalently, if  $(M,g)$ is globally isometric to the simply connected $3$-dimensional space of constant sectional curvature $K$.
\\Since the equality in \eqref{eq:RatVol1} is forced by the $K$-GAVP, the result follows.
\end{proof}

\section{Appendix}

\subsection{A $\sup$-Hawking mass for AF manifolds with non-negative scalar curvature}\label{SS:supHawk}
Inspired by the results of the present paper, it is natural to propose a slight variant of the Hawking mass. Indeed, the standard Hawking mass, while being very useful in applications (for instance in the proof of the Riemannian Penrose inequality via Inverse Mean Curvature Flow by Huisken-Ilmanen \cite{HuiskenIlmanen}), has some inconvenient features. For instance, it can be negative and it has no clear monotone property under inclusion. Even though the list of properties that are desirable for a quasi-local mass is open for debate, let us mention some natural ones.
\\

Let $(M^3,g)$ be an asymptotically flat Riemannian manifold with non-negative scalar curvature and with (possibly empty) horizon boundary $\partial M$.  
\\According to Bartnik \cite{Bartnik},  a  ``good'' notion of quasi-local mass $m(\Omega)$ for subsets $\Omega\subset M$ should satisfy:
\begin{enumerate}
\item [(i)]  $m(\Omega)$ should be uniquely defined for every domain $\Omega$;
\item [(ii)] Positivity: $m(\Omega)>0$ unless $\Omega\subset \R^{3}$, in which case $m(\Omega)=0$;
\item [(iii)] Monotonicity: if $\Omega_{1}\subset \Omega_{2}\subset M$, then $m(\Omega_{1})\leq m(\Omega_{2})$;
\item [(iv)] Asymptotic to ADM mass: If $\{\Omega_{k}\}_{k=1}^{\infty}$ is an exhaustion
of $M$, then $m(\Omega_{k})\to m_{\ADM}(M,g)$ as $k\to \infty$.
\end{enumerate}
Even if not explicitly requested by Bartnik, it is also natural to require: 
\begin{enumerate}
\item [(v)] Compatibility with Schwartzshild: Let $m_{Sch}>0$ and consider $\R^{3}\setminus B^{\delta}_{2m_{Sch}}(0) \simeq {\mathbb S}^{2}\times [2m_{Sch}, \infty)$ endowed with the Schwarzshild metric $g^{m_{Sch}}$ of mass $m_{Sch}>0$: 
\begin{equation}\label{eq:defgSch}
g^{m_{Sch}}:= \left(1- \frac{2m_{Sch}}{r}\right)^{-1} dr \otimes dr + r^{2} g_{{\mathbb S}^{2}}.
\end{equation}
Then  the quasi-local mass of every subset (satisfying suitable geometric conditions) containing the horizon $\{r=2 m_{Sch}\}$ is equal to $m_{Sch}$.
\end{enumerate}
\medskip

\noindent
We propose the next definition.
\begin{definition}[$\sup$-Hawking mass]\label{def:supHaw}
Let $(M^3,g)$ be an asymptotically flat Riemannian manifold  in the sense of Definition \ref{def:AFRiemMan} with non-negative scalar curvature and with (possibly empty) horizon boundary $\partial M$. 
For every open subset $\Omega\subset M$, denote with $\bar{\Omega}$ its topological closure and define the \emph{$\sup$-Hawking mass} as
\begin{equation}\label{eq:defmSH}
m_{SH}(\Omega):=\sup \{m_H(\partial \Omega') \mid \Omega'\subset \bar{\Omega} \text{ such that  $\partial \Omega'\simeq {\mathbb S}^2$ is smooth  and outer-minimising in $\bar{\Omega}$} \}.
\end{equation}

\end{definition}

It is clear that, if $\partial \Omega \simeq {\mathbb S}^2$ is smooth then
\begin{equation}\label{eq:mHleqmSH}
m_{H}(\partial \Omega)\leq m_{SH}(\Omega).
\end{equation}

A benefit of the proposed $\sup$-Hawking mass \eqref{eq:defmSH} is that it satisfies (a suitable version) of all the requirements (i)--(v). Property (i) is clearly satisfied, so let us discuss the others. The proof of (ii) is a nice application of the present paper as it involves basically all of the main results.
\\

\begin{prop}[Validity of (ii)] \label{prop:QLM(ii)}
Let $(M,g)$ and $\Omega\subset M$ be as in Definition \ref{def:supHaw}. 
%In case $\partial M\neq \emptyset$, denote with $(\Sigma_{i})_{i\in \N}$ the connected components of $\partial M$ and assume that either $\partial M \cap \Omega=\emptyset$ or  $\inf_{i: \Sigma_i\cap \Omega \neq \emptyset} {\rm Area} (\Sigma_{i}\cap \Omega)>0$.

Then $ m_{SH}(\Omega)\geq 0 $ with equality if and only if $\Omega\setminus \partial M$ is locally isometric to Euclidean $\R^{3}$.
\end{prop}

\begin{proof}
For every $p\in \Omega\setminus \partial M$ let $S_{p,\rho}(w_{p,\rho})$ be the optimally perturbed geodesic sphere constructed in Lemma \ref{lem:estw}. Of course, for $\rho>0$ small enough, it holds that  $S_{p,\rho}(w_{p,\rho})\subset \Omega\setminus \partial M$. 
%\medskip

%\textbf{Step 1:}  $\partial M=\emptyset$.  
From Theorem \ref{Prop:OuterMin},  $S_{p,\rho}(w_{p,\rho})$ is outer-minimising in $M$ (and thus in $\bar \Omega$) for $\rho>0$ small enough. Thus, %denoting with $B_{p,\rho}(w_{p,\rho})$ the region bounded by $S_{p,\rho}(w_{p,\rho})$, 
from the very definition of the sup-Hawking mass  \eqref{eq:defmSH} it holds that
\begin{equation}\label{eq:SHmOmegaHmSp}
m_{H}(S_{p,\rho}(w_{p,\rho})) \leq m_{SH}(\Omega), \quad \forall p\in \Omega\setminus \partial M, \, \forall \rho\in (0, \bar{\rho}_{p}),
\end{equation}
for some suitable $\bar{\rho}_{p}>0$ depending on $p\in \Omega\setminus \partial M$.
The combination of Proposition \ref{prop:expHawk} and Theorem \ref{prop:Riem0} then yields that $ m_{SH}(\Omega)\geq 0 $ with equality only if $\Omega\setminus \partial M$ is locally isometric to Euclidean $\R^{3}$. Conversely,  if $\Omega\setminus \partial M$ is locally isometric to $\R^{3}$ then one can choose sufficiently small  round spheres  as competitors in \eqref{eq:defmSH} and obtain that $ m_{SH}(\Omega)=  0 $. 
%\medskip
%
%\textbf{Step 2:}  $\partial M\neq\emptyset$.  By assumption,  either $\partial M \cap \Omega=\emptyset$ or $\inf_{i:\Sigma_i \cap \Omega\neq \emptyset} {\rm Area} (\Sigma_{i}\cap \Omega)=c>0$, where $(\Sigma_{i})_{i\in \N}$ are the connected components of  $\partial M$.  In  the latter case, for every $p\in \Omega$ and for small enough $\rho>0$, it holds that ${\rm Area}(S_{p,\rho}(w_{p,\rho}))<c$. In either case, one can follow verbatim the blow up argument in the proof of Theorem \ref{Prop:OuterMin}, prove that $S_{p,\rho}(w_{p,\rho})$ is outer-minimising in $\bar \Omega$ for $\rho>0$ sufficiently small, and argue as in step 1.  
\end{proof}

\begin{prop}[Validity of (iii)] \label{prop:MonmSH}
Let $(M,g)$ and $\Omega_{1}\subset \Omega_{2}\subset M$ be as in Definition \ref{def:supHaw} with $\partial \Omega_{1}$  outer-minimising in $\bar{\Omega}_{2}$. 
 Then $m_{SH}(\Omega_{1})\leq m_{SH}(\Omega_{2})$.
\end{prop}

\begin{proof}
If $\Omega'\subset \bar{\Omega}_{1}$ satisfies that $\partial \Omega'\simeq {\mathbb S}^2$ is smooth and outer-minimising in $\bar{\Omega}_{1}$ and $\partial \Omega_{1}$ is outer-minimising in $\bar{\Omega}_{2}$, then $\partial \Omega'\simeq {\mathbb S}^2$ is smooth and outer-minimising in $\bar{\Omega}_{2}$ as well. The monotonicity is then a direct consequence of the definition \eqref{eq:defmSH}.
\end{proof}

In order to prove (a suitable version of) property (iv), we first establish the next two results of independent interest.

\begin{prop} [$m_{\ADM}$ is an upper bound for $m_{SH}$]\label{prop:mSHleqmADM}
Let $(M,g)$ and $\Omega \subset M$ be as in Definition \ref{def:supHaw}. Assume in addition that $\partial \Omega$ is outer-minimising in $M$. Then $m_{SH}(\Omega)\leq m_{\ADM}(M,g)$.
\end{prop}

\begin{proof}
 Since $\partial \Omega$ is outer-minimising in $M$, every  $\Omega'\subset \bar{\Omega}$  with $\partial \Omega' \simeq {\mathbb S}^2$ smooth and outer-minimising in  $\bar{\Omega}$ is also outer-minimising in $M$. Thus, one can run the (weak version of the) Inverse Mean Curvature Flow by Huisken-Ilmanen \cite{HuiskenIlmanen} starting from $\Omega'$ and obtain that $m_{H}(\partial \Omega') \leq m_{\ADM}(M,g)$. The claim follows now from the very definition  \eqref{eq:defmSH} of $m_{SH}(\Omega)$.

\end{proof}

\begin{remark} [$m_{H}\leq m_{SH}\leq m_{B}$]\label{rem:mSHmB}
Since the upper bound in Proposition \ref{prop:mSHleqmADM} holds for any AF extension of $\Omega$, it follows that 
\begin{equation*}%\label{eq:mSHleqmBleqmADM}
m_{SH}(\Omega)\leq m_{B}(\Omega)\leq m_{\ADM}(M,g) \quad \text{for all } \Omega\subset M \text{ with $\partial \Omega$ outer-minimising},
\end{equation*}
where $m_{B}(\Omega)$ denotes the Bartnik mass of $\Omega$.
Thus, recalling \eqref{eq:mHleqmSH}, we obtain:
\begin{equation*}%\label{eq:mSHleqmBleqmADM}
m_{H}(\partial \Omega)\leq m_{SH}(\Omega)\leq m_{B}(\Omega)\leq m_{\ADM}(M,g), \quad \forall \Omega\subset M \text{ with $\partial \Omega\simeq {\mathbb S}^2$ smooth and outer-minimising}.
\end{equation*}
\end{remark}

\begin{lemma}[Existence of an exhaustion asymptotic to ADM mass] \label{lem:AsymptmSH}
Let $(M,g)$  be as in Definition \ref{def:supHaw}. Let $\Sigma_{\rho}$ be the coordinate sphere of radius $\rho \gg 1$ in an asymptotically flat coordinate chart and denote with $B_{\rho}$ the  bounded region enclosed by $\Sigma_{\rho}$. Then
$\lim_{\rho\to \infty} m_{SH} (B_{\rho})= m_{\ADM}(M,g)$. 
\end{lemma}

\begin{proof}
It is well known that for $\rho \gg 1$ sufficiently large, the coordinate sphere $\Sigma_{\rho}$ satisfies: 
\begin{itemize}
\item $\Sigma_{\rho}$ is outer-minimising in $M$ (a careful reader has probably noticed that this fact can also be proven by a blow-down argument analogous to the proof by blow-up of Theorem \ref{Prop:OuterMin});
\item $\lim_{\rho\to \infty} m_{H} (\Sigma_{\rho})= m_{\ADM}(M,g)$, see for instance \cite[Exercise 4.25]{LeeBook}.
\end{itemize}
The combination of the last property and \eqref{eq:mHleqmSH} gives on the one hand  that  
$$m_{\ADM}(M,g)= \lim_{\rho\to \infty} m_{H} (\Sigma_{\rho}) \leq \liminf_{\rho\to \infty} m_{SH} (B_{\rho}).$$
On the other hand, using Proposition \ref{prop:mSHleqmADM} and that   $\Sigma_{\rho}$ is outer-minimising, we infer that $m_{SH} (B_{\rho})\leq m_{\ADM}(M,g)$ for every $\rho \gg 1$. The conclusion follows.
\end{proof}

We can now prove the following (suitable version of) property (iv).

\begin{prop} [Validity of (iv)]
Let $(M,g)$ be as in Definition \ref{def:supHaw}.  If $\{\Omega_{k}\}_{k=1}^{\infty}$ is an exhaustion
of $M$ such that each $\partial \Omega_{k}$ is outer-minimising in $M$,  then $m_{SH}(\Omega_{k})\to m_{\ADM}(M,g)$ as $k\to \infty$. 
\end{prop}

\begin{proof}
First, using Proposition \ref{prop:mSHleqmADM} and that   $\partial \Omega_{k}$ is outer-minimising, we infer that 
$$m_{SH} (\Omega_{k})\leq m_{\ADM}(M,g), \quad \forall k\in \N.$$
Also, using that $\{\Omega_{k}\}_{k=1}^{\infty}$ is an exhaustion of $M$, we have that for every $\rho\gg 1$ there exists $k_{0}>0$ such that 
$$
B_{\rho}\subset \Omega_{k}, \quad \forall  k\geq k_{0},
$$
where $B_{\rho}$ is as in Lemma \ref{lem:AsymptmSH}. The monotonicity property  of $m_{SH}$ (see Proposition \ref{prop:MonmSH}) and the fact that $\partial B_{\rho}=\Sigma_{\rho}$ is outer-minimising in $M$ (see the proof of Lemma \ref{lem:AsymptmSH})  yield that  for every $\rho\gg 1$ sufficiently large there exists $k_{0}>0$ such that:
$$m_{SH}(B_{\rho})\leq  m_{SH} (\Omega_{k}), \quad \forall  k\geq k_{0}. $$
Recalling Lemma \ref{lem:AsymptmSH}, we thus obtain
$$m_{\ADM} (M,g)= \lim_{\rho\to \infty}m_{SH}(B_{\rho})\leq \liminf_{k\to \infty} m_{SH} (\Omega_{k}).$$
The conclusion follows by the combination of the first and last displayed formulas in the proof.
\end{proof}

We next establish property (v).% \footnote{A: discussion of boundary}A careful reader will notice that, strictly speaking,  the Schwarzschild space does not enter in the framework of Definition \ref{def:AFRiemMan} (as $M$ has as boundary the horizon $\{r=2m_{Sch}\}$, which is a closed minimal surface). However, the very same definition \eqref{eq:defmSH} can and shall be used.
\begin{prop} [Validity of (v)]\label{prop:(v)}
Let $m_{Sch}>0$ and consider $\R^{3}\setminus B^{\delta}_{2m_{Sch}}(0) \simeq {\mathbb S}^{2}\times [2m_{Sch}, \infty)$ endowed with the Schwarzshild metric $g^{m_{Sch}}$ of mass $m_{Sch}>0$ as in \eqref{eq:defgSch}.  Let  $\Omega\subset \R^{3}\setminus B^{\delta}_{2m_{Sch}}(0)$ be an open subset  such that $\partial \Omega$ is outer-minimising and $\Omega$ contains the horizon, i.e. $\{r=2 m_{Sch}\} \subset \Omega$. Then $m_{SH}(\Omega)= m_{Sch}$.  
\end{prop}

\begin{proof}
Since $m_{\ADM}(g^{m_{Sch}})= m_{Sch}$, the upper bound given in Proposition \ref{prop:mSHleqmADM} %(the very same proof applies to the Schwarzshild space, since it enters the framework of Huisken-Ilmanen \cite{HuiskenIlmanen}) 
 yields that 
$$m_{SH}(\Omega)\leq m_{Sch}.$$
To obtain the reversed inequality, observe that the horizon $\Sigma_{2 m_{Sch}}:=\{r=2 m_{Sch}\}$ is outer-minimising and satisfies that $m_{H}(\Sigma_{2 m_{Sch}})= m_{Sch}$. The  definition \eqref{eq:defmSH} of $m_{SH}$ yields:
$$
m_{Sch}=m_{H}(\Sigma_{2 m_{Sch}})\leq m_{SH}(\Omega).
$$
The conclusion follows by combining the two inequalities displayed in the proof.
\end{proof}

We wish to conclude this appendix by suggesting a possible research direction, drawing a more precise connection between the sup-Hawking mass and the Bartnik mass. 
\\Notice indeed that the sup-Hawking mass corresponds to ``defining a quasi-local mass by taking the supremum from inside'' while the Bartnik mass corresponds to ``defining a quasi-local mass by taking the infimum from outside''.  In Remark \ref{rem:mSHmB}, we observed that for every $\Omega\subset M$ with $\partial \Omega$ outer-minimising it holds that $m_{SH}(\Omega)\leq m_{B}(\Omega)$, however it is natural to expect that the two objects coincide (under appropriate conditions on $\Omega$ and/or  $M$): 
\\

\textbf{Open Problem}: find appropriate (necessary and/or sufficient) conditions  on $\Omega$ and/or  $M$ so that 
$$m_{SH}(\Omega)= m_{B}(\Omega).$$

 \begin{remark}[Extension to the case of non-null cosmological constant]
By using the generalized Hawking mass \eqref{eq:defGenHaw}, one can extend the definition \eqref{eq:defmSH} of $\sup$-Hawking mass to the case when $(M^3,g)$ has scalar curvature bounded below by $6K$, with $K\in \{-1,0,1\}$.
Some of the good properties above would be  retained:
\begin{enumerate}
\item [{\rm (i)}] is of course satisfied;
\item [{\rm (ii)}] is satisfied in an analogous form as in Proposition \ref{prop:QLM(ii)}: Let $\Omega\subset M$ be a bounded open set, then  $ m_{SH}(\Omega)\geq 0 $ with equality if and only if $\Omega\setminus \partial M$ is locally isometric to the space form of constant sectional curvature $K$. 
\\This follows by localizing the proof of Theorem \ref{thm:9} (analogous to the proof of the quasi-local rigidity Theorem   \ref{prop:Riem0}) and by extending the outer-minimizing property of $S_{p,\rho}(w)$ in $\Omega$ (see the proof of Theorem \ref{Prop:OuterMin}) to the case $\Sc \geq - 6K$ and bounded $\Omega$;
\item [{\rm (iii)}] is satisfied in the same form as in Proposition  \ref{prop:MonmSH}.
\end{enumerate}
However, property {\rm (iv)}, the upper bound analogous to Proposition \ref{prop:mSHleqmADM}, and the analogue of Proposition \ref{prop:(v)} have all little chances to hold for $K\neq 0$  (see Remark \ref{rem:HMHyp} for more details). This is why the presentation here is focused on the $AF$ case with non-negative scalar curvature.  
\end{remark}

%
%If $(M,g)$ is AF with non-negative scalar curvature and $\partial \Omega$ is outer-minimising, then
%\begin{equation}\label{eq:mSHleqmBleqmADM}
%m_{SH}(\Omega)\leq m_{B}(\Omega)\leq m_{\ADM}(M,g).
%\end{equation}
%Indeed,  the combination of \eqref{mBgeqmH} and \eqref{eq:monmB} gives
%$$
%m_{H}(\Omega')\leq m_{B}(\Omega') \leq m_{B}(\Omega),
%$$
%for all $\Omega'\subset \Omega$ outer-minimising in $\Omega$.

\subsection{Local area-constrained maximisers of $m_{H}$ are perturbed geodesic spheres}
In this paper, we often estimated the $\sup$ of the Hawking mass using perturbed geodesic spheres $S_{p.\rho}(w)$. Even if not strictly necessary for such arguments, in the next proposition we prove that such a choice of competitors is very natural. Indeed we show that optimal competitors  for  the supremum  of the Hawking mass  among surfaces contained in a small ball  are given by perturbed geodesic spheres $S_{p,\rho}(w)$ with $w\in C^{4,\alpha}({\mathbb S}^{2})^{\perp}$, $\|w\|_{C^{4, \alpha}}\leq C \rho^{2}$. Thus, such optimal competitors satisfy the expansions given in  Lemma \ref{lem:estw} and Proposition \ref{prop:expHawk}.   For related results in this direction, see  Lamm-Metzger \cite{LammMetzger2013}, who proved $W^{2,2}$-closeness to a geodesic sphere under a small energy assumption, and  Laurain-Mondino \cite{MondinoLaurain}, who proved smooth convergence to  a geodesic sphere under a milder energy assumption. 
\\Recall that, by definition 
$$ C^{4,\alpha}({\mathbb S}^{2})^{\perp}:=  C^{4,\alpha}({\mathbb S}^{2}) \cap \big[ {\rm Ker}\big( \Delta_{{\mathbb S}^{2}} (\Delta_{{\mathbb S}^{2}}+2 ) \big) \big]^{\perp},$$
 where $\big[ {\rm Ker}\big( \Delta_{{\mathbb S}^{2}} (\Delta_{{\mathbb S}^{2}}+2 ) \big) \big]^{\perp}\subset L^{2}({\mathbb S}^{2}) $ denotes the $L^{2}$-orthogonal space to the finite (actually four) dimensional space ${\rm Ker}\big( \Delta_{{\mathbb S}^{2}} (\Delta_{{\mathbb S}^{2}}+2 ) \big)$.
 
\begin{prop}\label{prop:SigmaSprhow}
Let $(M,g)$ be a three-dimensional Riemannian manifold and let $\Sigma_{j}\subset M$ be a sequence of maximisers of $m_{H}$ under area constraint and Hausdorff converging to a point $\bar p\in M$. Then $\nabla \Sc(\bar p)=0$ and there exist $p_{j}\to \bar{p}, \rho_{j}\downarrow 0, w_{j}\in C^{4,\alpha}({\mathbb S}^{2})^{\perp}$ with  $\limsup_{j\to \infty}\rho_{j}^{-2} \|w_{j}\|_{C^{4,\alpha}({\mathbb S}^{2})} <\infty$  such that, up to a subsequence,  $\Sigma_{j}=S_{p_{j}, \rho_{j}}(w_{j})$.
\end{prop}

\begin{proof}
First of all, recall that non-orientable closed two-dimensional surfaces cannot be embedded in $\R^{3}$, but only immersed (i.e. with self-intersections). Thus the Li-Yau inequality \cite{LiYau} implies that
\begin{equation}\label{eq:NonOrientable}
\inf\{ W(\Sigma): \Sigma \subset \R^{3} \text{closed non-orientable surface } \}\geq 32 \pi>16 \pi.
\end{equation}
Moreover, from the proof of the Willmore conjecture by Marques-Neves \cite{MarquesNeves} we know that
\begin{equation}\label{eq:WillConj}
\inf\{W(\Sigma): \Sigma \subset \R^{3} \text{closed surface with genus}(\Sigma)\geq 1  \}\geq 8\pi^{2}>16 \pi.
\end{equation}
Using normal coordinates centred at $\bar{p}$ and estimating the difference between the Riemannian and Euclidean geometric quantities (see for instance \cite{MondinoSchygulla}), it is not hard to check that  \eqref{eq:NonOrientable} and \eqref{eq:WillConj} yield the existence of $C>0$ such that for any sequence  $\Sigma_j\subset M$ Hausdorff converging to a point $p$, with $\Sigma_j$ either non-orientable or of genus at least one,  it holds:
\begin{equation}\label{eq:SigmajHGNO}
\limsup_{j\to \infty} \frac{m_H(\Sigma_j)}{ \sqrt{|\Sigma_j|}}\leq -C.
\end{equation}
On the other hand, it is easy to see that for geodesic spheres $S_{p,\rho}$ it holds
\begin{equation}\label{eq:Spro0limsup}
\limsup_{\rho\to 0} \frac{m_H(S_{p,\rho})}{ \sqrt{|S_{p,\rho}|}}=0.
\end{equation}
Thus, the assumptions that $\Sigma_{j}\subset M$ are a sequence of maximisers of $m_{H}$ under area constraint and Hausdorff converging to a point $\bar p\in M$ yield that $\Sigma_j$ must be a topological sphere, for $j$ large enough, with  $W(\Sigma_j)<8 \pi$. 
Now we can apply \cite[Corollary 1.3]{MondinoLaurain} (see also \cite{LammMetzger2010}) to infer that  $\nabla \Sc(\bar p)=0$ and that, if we rescale $(M,g)$ around $\bar p$ in such a way that the rescaled surfaces $\tilde{\Sigma}_{j}$ have fixed area 1, then $\tilde{\Sigma}_{j}$ converge smoothly (up to a subsequence) to a round sphere in the three-dimensional Euclidean space. In particular $\Sigma_j$ is a normal graph over a geodesic sphere, of small radius and graph function (in any $C^k$ norm). Now, by  a contraction mapping argument one can find  
$$p_{j}\to \bar{p}, \;\rho_{j}\downarrow 0, \;w_{j}\in C^{4,\alpha}({\mathbb S}^{2})^{\perp} \text{ with } \limsup_{j\to \infty}\rho_{j}^{-2} \|w_{j}\|_{C^{4,\alpha}({\mathbb S}^{2})} <\infty$$
   such that  $\Sigma_{j}=S_{p_{j}, \rho_{j}}(w_{j})$.  The proof of this last claim can be performed along the lines as \cite[Lemma 5.3]{Mondino2}:  although the statement of \cite[Lemma 5.3]{Mondino2}  is  for critical points of $W$, the same proof holds verbatim  more generally for area-constrained critical points using that the Lagrange multipliers are bounded, thanks to \cite[Lemma 2.2]{MondinoLaurain}.

\end{proof}

\subsection{Second fundamental form on a perturbed geodesic sphere}\label{SS:AppSecFF}
In this section we give a self-contained proof of the Taylor expansion \eqref{eq:Exphij}  for the second fundamental form on a perturbed geodesic sphere. We follow the approach in  \cite{Pacard}, however we
compute more terms, since they are needed for the main results in this paper.
\\First of all let $Z_{i}$, $i=1,2$, be  the coordinate vector fields on $S_{p,\rho}(w)$:
\begin{equation*}
    Z_i = \exp_*{(\partial_{\theta^i}\rho(1 - w)\Theta)} = \rho((1 - w)\Theta_i - w_i\Theta).
\end{equation*}
In order to find an expression for the inward pointing unit normal vector to $S_{p,\rho}(w)$, consider:
\begin{align*}
    \Tilde{N} := -\Theta + a^jZ_j
\end{align*}
where  $a^j$ are such that $\Tilde{N}$ is orthogonal to both $Z_1$ and $Z_2$. Computing, we get:
 \begin{align*}
     g(\Tilde{N}, Z_i) &= g(-\Theta + a^jZ_j, Z_i) = - g(\Theta, Z_i) + a^jg(Z_j, Z_i) = - g(\Theta, \rho((1 - w)\Theta_i - w_i\Theta)) + a^j\mathring{g}_{ij} \\ &= \rho w_i + a^j\mathring{g}_{ij}, 
 \end{align*}
where we used that $g(\Theta, \Theta) = 1$ and $g(\Theta, \Theta_i) = 0$. Therefore, to satisfy orthogonality, we need to choose $a^j$ such that $a^j\mathring{g}_{ij} = - w_i\rho$, or $a^j = - \mathring{g}^{ij}w_i\rho$. In order to find the normalization constant, compute:
\begin{align*}
    g(\Tilde{N}, \Tilde{N}) &= g(-\Theta + a^jZ_j, -\Theta + a^iZ_i) \\ &= g(\Theta, \Theta) - g(\Theta, a^iZ_i) - g(\Theta, a^jZ_j) + g(a^jZ_j, a^iZ_i) \\ &= 1 - a^ig(\Theta, \rho((1 - w)\Theta_i - w_i\Theta)) - a^jg(\Theta, \rho((1 - w)\Theta_j - w_j\Theta)) + a^ia^j\mathring{g}_{ij} \\ &= 1 + a^iw_i\rho + a^jw_j\rho + a^ia^j\mathring{g}_{ij} \\ &= 1 - \mathring{g}^{ij}w_iw_j\rho^2.
\end{align*}
Combining the last two relations, we get the (inward) unit normal vector:
\begin{align*}
    N = (1 - \mathring{g}^{kl}w_kw_l\rho^2)^{-\frac{1}{2}}\Big(- \Theta - \mathring{g}^{ij}w_i((1 - w)\Theta_j - w_j\Theta)\rho^2\Big).
\end{align*}
Using the Taylor expansion around $0$ of $(1-x)^{-\frac{1}{2}}$ with $x = \mathring{g}^{ij}w_iw_j\rho^2$, we  get:
\begin{align}\label{eq:48}
    g(\Tilde{N}, \Tilde{N})^{-\frac{1}{2}} =1 + \frac{1}{2}\mathring{g}^{ij}w_iw_j\rho^2+\cQ^{(4)(1)}_{p}(w).
\end{align}
%\begin{align}\label{eq:48}
%    g(\Tilde{N}, \Tilde{N})^{-\frac{1}{2}} = 1 + \frac{1}{2}\mathring{g}^{ij}w_iw_j\rho^2 + \frac{3}{8}(\mathring{g}^{ij}w_iw_j\rho^2)^2 + \cO(\rho^6) \cQ^{(2)(1)}(w).
%\end{align}
%
Combining with \eqref{eq:ExpgijInv}, we obtain:
\begin{equation*}
    |N + \Theta + g_{{\mathbb S}^2}^{ij}w_i\Theta_j | =  \rho^2 \cL_{p}^{(1)}(w) + \cQ_{p}^{(2)(1)}(w).
\end{equation*}
We next compute $\Tilde{h}_{ij} := -g(\nabla_{Z_i}\Tilde{N},Z_j)$ as a first step for obtaining $h_{ij}$. We have:
\begin{align}\label{eq:37}
    \Tilde{h}_{ij} &= -g(\nabla_{Z_i}(-\Theta + a^kZ_k),Z_j) \nonumber\\ &= g(\nabla_{Z_i}\Theta,Z_j) - g(\nabla_{Z_i}a^kZ_k,Z_j) \nonumber\\ &= \frac{w_i}{1-w}g(\Theta,Z_j) - \frac{w_i}{1-w}g(\Theta,Z_j) + g(\nabla_{Z_i}\Theta,Z_j) - g(\nabla_{Z_i}a^kZ_k,Z_j) \nonumber\\ 
    &= \frac{w_i}{1-w}g(\Theta,Z_j) + \frac{1}{1-w}[(1-w)g(\nabla_{Z_i}\Theta,Z_j) - w_ig(\Theta,Z_j)] - g(\nabla_{Z_i}a^kZ_k,Z_j) \nonumber\\ 
    &= \frac{w_i}{1-w}g(\Theta,Z_j) + \frac{1}{1-w}g((1-w)\nabla_{Z_i}\Theta - w_i\Theta,Z_j) - g(\nabla_{Z_i}a^kZ_k,Z_j) \nonumber\\ 
    &= \frac{w_i}{1-w}g(\Theta,Z_j) + \frac{1}{1-w}g(\nabla_{Z_i}((1-w)\Theta),Z_j) - g(\nabla_{Z_i}a^kZ_k,Z_j).
\end{align}
We compute the three terms in (\ref{eq:37}) seperately. For the first one, we use the definition of $Z_i$ and the fact that $g(\Theta,\Theta) = 1$ and $g(\Theta,\Theta_i) = 0$ to obtain:
\begin{align}\label{eq:38}
    \frac{w_i}{1-w}g(\Theta,Z_j) = \frac{w_i}{1-w}g(\Theta,\rho((1-w)\Theta_j - w_j\Theta)) = -\frac{w_iw_j\rho}{1-w}.
\end{align}
Now consider $\rho$ as a variable, giving:
\begin{equation*}
    Z_0 := \exp_*(\partial_{\rho}(\rho(1-w)\Theta)) = (1-w)\Theta.
\end{equation*}
Since
\begin{align*}
    g(\nabla_{Z_i}((1-w)\Theta),Z_j) &= Z_i(g((1-w)\Theta,Z_j)) - g((1-w)\Theta,\nabla_{Z_i}Z_j) \\&= Z_i(g((1-w)\Theta,\rho((1-w)\Theta_j - w_j\Theta))) - g((1-w)\Theta,\nabla_{Z_i}Z_j) \\&= Z_i(\rho(w-1)w_j) - g((1-w)\Theta,\nabla_{Z_i}Z_j) \\&= \rho(w_iw_j + ww_{ji} - w_{ji}) - g((1-w)\Theta,\nabla_{Z_i}Z_j)
\end{align*}
is symmetric in $i$ and $j$, we have:
\begin{align*}
    g(\nabla_{Z_i}((1-w)\Theta),Z_j) = g(\nabla_{Z_j}((1-w)\Theta),Z_i).
\end{align*}
Thus we compute:
\begin{align}\label{eq:39}
    g(\nabla_{Z_i}((1-w)\Theta),Z_j) \nonumber&= \frac{1}{2}(g(\nabla_{Z_i}((1-w)\Theta),Z_j) +g(\nabla_{Z_j}((1-w)\Theta),Z_i)) \nonumber\\ &= \frac{1}{2}(g(\nabla_{Z_i}Z_0,Z_j) + g(\nabla_{Z_j}Z_0,Z_i)) = \frac{1}{2}(g(\nabla_{Z_0}Z_i,Z_j) + g(\nabla_{Z_0}Z_j,Z_i)) \nonumber\\ &= \frac{1}{2} Z_0(g(Z_i,Z_j)) = \frac{1}{2} \partial_{\rho}\mathring{g}_{ij},
\end{align}
which sorts out the second term in (\ref{eq:37}). The final term becomes:
\begin{align}\label{eq:40}
    g(\nabla_{Z_i}a^kZ_k,Z_j) \nonumber &= Z_i(a^kg(Z_k,Z_j)) - a^kg(Z_k,\nabla_{Z_i}Z_j) \nonumber\\ &= Z_i(-\mathring{g}^{lk}w_l\rho\mathring{g}_{kj}) + \mathring{g}^{lk}w_l\rho\mathring{\Gamma}^m_{ij}\mathring{g}_{km} \nonumber\\ &= -w_{ji}\rho + w_l\mathring{\Gamma}^l_{ij}\rho \nonumber\\ &= -(\Hess_{\mathring{g}}w)_{ij}\rho.
\end{align}
Substituting (\ref{eq:38}), (\ref{eq:39}) and (\ref{eq:40}) into (\ref{eq:37}) gives:
\begin{align}\label{eq:41}
    \Tilde{h}_{ij} = -\frac{w_iw_j\rho}{1-w} + \frac{1}{2(1-w)}\partial_{\rho}\mathring{g}_{ij} + (\Hess_{\mathring{g}}w)_{ij}\rho.
\end{align}
To further expand $\Tilde{h}_{ij}$, we  combine the first two terms of (\ref{eq:41}). Differentiating (\ref{eq:Expgij}) with respect to $\rho$ gives:
\begin{align*}
    \partial_{\rho}\mathring{g}_{ij} &= 2g_{ij}^{{\mathbb S}^2}(1 - w)^2\rho + 2w_iw_j\rho + \frac{4}{3}g({\mathcal R}(\Theta,\Theta_i)\Theta,\Theta_j)(1 - w)^4\rho^3 \\ &\quad+ \frac{5}{6}g(\nabla_{\Theta}{\mathcal R}(\Theta,\Theta_i)\Theta,\Theta_j)(1 - w)^5\rho^4 + \frac{3}{10}g(\nabla_{\Theta \Theta}^2{\mathcal R}(\Theta,\Theta_i)\Theta,\Theta_j)(1 - w)^6\rho^5 \\ &\quad+ \frac{4}{15} \delta^{\mu\nu}g({\mathcal R}(\Theta,\Theta_i)\Theta,E_{\mu})g({\mathcal R}(\Theta,\Theta_j)\Theta,E_{\nu})(1 - w)^6\rho^5 + \cO_{p}(\rho^6) + \rho^6 \cL_{p}^{(0)}(w) + \rho^6 \cQ^{(2)(0)}_{p}(w).
\end{align*}
Hence, the first term in (\ref{eq:41}) cancels the second term of $\partial_{\rho}\mathring{g}_{ij}$ after it is multiplied by $\frac{1}{2(1-w)}$. This leaves:
\begin{align}\label{eq:43}
    \Tilde{h}_{ij} &= g_{ij}^{{\mathbb S}^2}(1 - w)\rho + (\Hess_{\mathring{g}}w)_{ij}\rho + \frac{2}{3}g({\mathcal R}(\Theta,\Theta_i)\Theta,\Theta_j)(1 - w)^3\rho^3 \nonumber\\ &\quad+ \frac{5}{12}g(\nabla_{\Theta}{\mathcal R}(\Theta,\Theta_i)\Theta,\Theta_j)(1 - w)^4\rho^4 + \frac{3}{20}g(\nabla_{\Theta\Theta}^2{\mathcal R}(\Theta,\Theta_i)\Theta,\Theta_j)(1 - w)^5\rho^5 \\ &\quad+ \frac{2}{15}\delta^{\mu\nu}g({\mathcal R}(\Theta,\Theta_i)\Theta,E_{\mu})g({\mathcal R}(\Theta,\Theta_j)\Theta,E_{\nu})(1 - w)^5\rho^5  + \cO_{p}(\rho^6) + \rho^6 \cL_{p}^{(0)}(w) + \rho^6 \cQ^{(2)(0)}_{p}(w). \nonumber
\end{align}
We finally compute $(\Hess_{\mathring{g}}w)_{ij}$. By definition we have:
\begin{equation}\label{eq:46}
    (\Hess_{\mathring{g}}w)_{ij} = w_{ij} - \mathring{\Gamma}^k_{ij}w_k = w_{ij} - \left[\frac{1}{2}\mathring{g}^{kl}(\partial_i\mathring{g}_{jl} + \partial_j\mathring{g}_{il} - \partial_l\mathring{g}_{ij})\right]w_k.
\end{equation}
Differentiating (\ref{eq:Expgij}) term by term shows:
\begin{align}\label{eq:44}
    \partial_i\mathring{g}_{jl} &= \partial_i(g_{jl}^{{\mathbb S}^2})(1 - 2w) \rho^2 - 2g_{jl}^{{\mathbb S}^2}w_i\rho^2 + \frac{1}{3}\partial_i(g({\mathcal R}(\Theta,\Theta_j)\Theta,\Theta_l))\rho^4 \\&\quad+ \cO_{p}(\rho^5) +  \rho^4 \cL_{p}^{(1)}(w) + \rho^2 \, \cQ_{p}^{(2)(2)}(w). \nonumber
\end{align}
Combining (\ref{eq:ExpgijInv}) and (\ref{eq:44}) we get:
\begin{align}\label{eq:47}
    \mathring{\Gamma}^k_{ij} &= \Gamma^k_{ij} + g_{{\mathbb S}^2}^{kl}(g_{ij}^{{\mathbb S}^2}w_l - g_{jl}^{{\mathbb S}^2}w_i - g_{il}^{{\mathbb S}^2}w_j) \nonumber\\&\quad+ \frac{1}{6}g_{{\mathbb S}^2}^{kl}(\partial_i(g({\mathcal R}(\Theta,\Theta_j)\Theta,\Theta_l)) + \partial_j(g({\mathcal R}(\Theta,\Theta_i)\Theta,\Theta_l)) - \partial_l(g({\mathcal R}(\Theta,\Theta_i)\Theta,\Theta_j)))\rho^2 \\&\quad- \frac{1}{6}g_{{\mathbb S}^2}^{kn}g_{{\mathbb S}^2}^{ml}g({\mathcal R}(\Theta,\Theta_n)\Theta,\Theta_m)(\partial_i(g_{jl}^{{\mathbb S}^2}) + \partial_j(g_{il}^{{\mathbb S}^2}) - \partial_l(g_{ij}^{{\mathbb S}^2}))\rho^2 \nonumber\\&\quad+ \cO_{p}(\rho^3) + \rho^{2} \cL_{p}^{(1)}(w)+ \cQ_{p}^{(2)(2)}(w),\nonumber
\end{align}
where $\Gamma^k_{ij}$ are the Christoffel symbols of $g_{{\mathbb S}^2}$. Substituting (\ref{eq:46}) and (\ref{eq:47}) into (\ref{eq:43}) gives:
\begin{align*}
    \Tilde{h}_{ij} &= g_{ij}^{{\mathbb S}^2}(1 - w)\rho + (\Hess_{{\mathbb S}^2}(w))_{ij}\rho \\ &\quad+ w_kg_{{\mathbb S}^2}^{kl}\Big(g_{jl}^{{\mathbb S}^2}w_i + g_{il}^{{\mathbb S}^2}w_j - g_{ij}^{{\mathbb S}^2}w_l\Big)\rho \\ &\quad+ \frac{2}{3}g({\mathcal R}(\Theta,\Theta_i)\Theta,\Theta_j)(1 - w)^3\rho^3 \\ &\quad+ \frac{1}{6}w_kg_{{\mathbb S}^2}^{kn}g_{{\mathbb S}^2}^{ml}g({\mathcal R}(\Theta,\Theta_n)\Theta,\Theta_m)\Big(\partial_ig_{jl}^{{\mathbb S}^2} + \partial_jg_{il}^{{\mathbb S}^2} - \partial_lg_{ij}^{{\mathbb S}^2}\Big)\rho^3 \\ &\quad- \frac{1}{6}w_kg_{{\mathbb S}^2}^{kl}\Big(\partial_ig({\mathcal R}(\Theta,\Theta_j)\Theta,\Theta_l) + \partial_jg({\mathcal R}(\Theta,\Theta_i)\Theta,\Theta_l) - \partial_lg({\mathcal R}(\Theta,\Theta_i)\Theta,\Theta_j)\Big)\rho^3 \\ &\quad+ \frac{5}{12}g(\nabla_{\Theta}{\mathcal R}(\Theta,\Theta_i)\Theta,\Theta_j)\rho^4 + \frac{3}{20}g(\nabla_{\Theta}^2{\mathcal R}(\Theta,\Theta_i)\Theta,\Theta_j)\rho^5 \\ &\quad+ \frac{2}{15}\delta^{\mu\nu}g({\mathcal R}(\Theta,\Theta_i)\Theta,E_{\mu})g({\mathcal R}(\Theta,\Theta_j)\Theta,E_{\nu})\rho^5 \nonumber\\
&\quad    + \cO_{p}(\rho^6) + \rho^4 \cL^{(1)}_{p}(w) + \rho \cQ^{(3)(2)}_{p}(w) +  \rho^3 \cQ_{p}^{(2)(1)}(w).
\end{align*}
Finally, to complete the proof,  note that:
\begin{align*}
    h_{ij} &:= -g(\nabla_{Z_i}N,Z_j) = -g(\nabla_{Z_i}g(\Tilde{N}, \Tilde{N})^{-\frac{1}{2}}\Tilde{N},Z_j) \\
  &  = - Z_i(g(\Tilde{N}, \Tilde{N})^{-\frac{1}{2}})g(\Tilde{N},Z_j) - g(\Tilde{N}, \Tilde{N})^{-\frac{1}{2}}g(\nabla_{Z_i}\Tilde{N},Z_j) \\&= g(\Tilde{N}, \Tilde{N})^{-\frac{1}{2}}\Tilde{h}_{ij},
\end{align*}
where we used that $\Tilde{N}$ is orthogonal to $Z_j$. Using \eqref{eq:48} we obtain:  
\begin{equation}
\begin{split}\label{eq:finalhAppendix}
   {h}_{ij} &= g_{ij}^{{\mathbb S}^2}(1 - w)\rho + (\Hess_{{\mathbb S}^2}(w))_{ij}\rho \\ &\quad+\frac{1}{2} g_{ij}^{{\mathbb S}^{2}}  g^{kl}_{{\mathbb S}^{2}} w_{k} w_{l} \rho + w_kg_{{\mathbb S}^2}^{kl}\Big(g_{jl}^{{\mathbb S}^2}w_i + g_{il}^{{\mathbb S}^2}w_j - g_{ij}^{{\mathbb S}^2}w_l\Big)\rho \\ &\quad+ \frac{2}{3}g({\mathcal R}(\Theta,\Theta_i)\Theta,\Theta_j)(1 - w)^3\rho^3 \\ &\quad+ \frac{1}{6}w_kg_{{\mathbb S}^2}^{kn}g_{{\mathbb S}^2}^{ml}g({\mathcal R}(\Theta,\Theta_n)\Theta,\Theta_m)\Big(\partial_ig_{jl}^{{\mathbb S}^2} + \partial_jg_{il}^{{\mathbb S}^2} - \partial_lg_{ij}^{{\mathbb S}^2}\Big)\rho^3 \\ &\quad- \frac{1}{6}w_kg_{{\mathbb S}^2}^{kl}\Big(\partial_ig({\mathcal R}(\Theta,\Theta_j)\Theta,\Theta_l) + \partial_jg({\mathcal R}(\Theta,\Theta_i)\Theta,\Theta_l) - \partial_lg({\mathcal R}(\Theta,\Theta_i)\Theta,\Theta_j)\Big)\rho^3 \\ &\quad+ \frac{5}{12}g(\nabla_{\Theta}{\mathcal R}(\Theta,\Theta_i)\Theta,\Theta_j)\rho^4 + \frac{3}{20}g(\nabla_{\Theta}^2{\mathcal R}(\Theta,\Theta_i)\Theta,\Theta_j)\rho^5 \\ &\quad+ \frac{2}{15}\delta^{\mu\nu}g({\mathcal R}(\Theta,\Theta_i)\Theta,E_{\mu})g({\mathcal R}(\Theta,\Theta_j)\Theta,E_{\nu})\rho^5 \nonumber\\
&\quad    + \cO_{p}(\rho^6) + \rho^4 \cL^{(1)}_{p}(w) + \rho \cQ^{(3)(2)}_{p}(w) +  \rho^3 \cQ_{p}^{(2)(1)}(w).
\end{split}
\end{equation}

\bibliography{Bibliography}{}
\bibliographystyle{plain}

\end{document}